\documentclass[11pt]{amsart}

\usepackage{amssymb}
\usepackage{amsthm}
\usepackage{amsmath}
\usepackage{mathtools}
\usepackage{latexsym}
\usepackage{comment}
\usepackage{hyperref}
\usepackage{cleveref}
\usepackage{verbatim}
\usepackage{xypic}
\usepackage{graphicx}
\usepackage{subcaption}
\usepackage{tikz-cd}
\usepackage[utf8]{inputenc}
\usepackage[margin=1in]{geometry}
\usepackage[shortlabels]{enumitem}
\usepackage{xcolor}

\usetikzlibrary{decorations.pathmorphing}

\newtheorem{thm}{Theorem}[section]
\newtheorem{prop}[thm]{Proposition}

\newtheorem{lemma}[thm]{Lemma}
\newtheorem{cor}[thm]{Corollary}
\newtheorem{dfn}[thm]{Definition}

\theoremstyle{definition} \newtheorem{ex}[thm]{Example}

\theoremstyle{definition} \newtheorem{rmk}[thm]{Remark}

\newcommand{\cc}{\mathbb{C}}
\newcommand{\rr}{\mathbb{R}}
\newcommand{\qq}{\mathbb{Q}}
\newcommand{\zz}{\mathbb{Z}}

\newcommand{\aff}{\mathbb{A}}
\newcommand{\proj}{\mathbb{P}}

\newcommand{\rank}{\mathrm{rank}}
\newcommand{\Gal}{\mathrm{Gal}}

\newcommand{\Spec}{\mathrm{Spec}}

\newcommand{\Pic}{\mathrm{Pic}}

\newcommand{\Xmini}{\mathcal{X}^{\mathrm{min}}}
\newcommand{\Cmini}{\mathcal{C}^{\mathrm{min}}}
\newcommand{\Dmini}{\mathfrak{D}^{\mathrm{min}}}
\newcommand{\Hmini}{H^{\mathrm{min}}}
\newcommand{\hatHmini}{\hat{H}^{\mathrm{min}}}
\newcommand{\Cst}{\mathcal{C}^{\mathrm{st}}}
\newcommand{\Irred}{\mathrm{Irred}}

\newcommand{\ord}{\mathrm{ord}}

\newcommand{\pfrac}{\frac{pv(p)}{p-1}}
\newcommand{\ptfrac}{\tfrac{pv(p)}{p-1}}
\newcommand{\SF}[1]{(#1)_s}
\newcommand{\truncate}[1]{\min\{{#1},\ptfrac\}}

\newcommand{\vfun}{\underline{v}}

\newcommand{\Berk}{\mathbb{P}_{\mathbb{C}_K}^{1, \mathrm{an}}}
\newcommand{\Hyp}{\mathbb{H}_{\mathbb{C}_K}^{1, \mathrm{an}}}

\graphicspath{{pics/}}

\title{A cluster criterion for potential degeneracy of superelliptic curves}
\author{Jeffrey Yelton}

\begin{document}

\maketitle

\begin{abstract}

Let $K$ be a field with a discrete valuation; let $p$ be a prime; and let $C$ be the curve defined by an equation of the form $y^p = f(x)$.  We show that the curve $C$ has a model over an algebraic extension of $K$ whose special fiber consists of genus-$0$ components and has at worst nodal singularities if and only if the cluster data of the roots of $f$ satisfies a certain criterion, and when these hold, we show explicitly how to build the minimal regular model of $C$.  We develop an interpretation of cluster data in terms of a convex hull in the Berkovich projective line and express the above results directly in terms of this convex hull.

\end{abstract}

\section{Introduction} \label{sec intro}

Let $K$ be a field complete with respect to a discrete valuation $v : K^\times \to \qq$ and with algebraically closed residue field denoted by $k$.  The completeness of $K$ implies that the valuation $v$ extends uniqely to a valuation on any finite extension $K' / K$ as well as an algebraic closure $\bar{K}$ of $K$ and its completion $\cc_K$; we denote all of these valuations also by $v$.  For any finite extension $K' / K$, we write $\mathcal{O}_{K'} := \{z \in K' \ | \ v(z) \geq 0\}$ for its ring of integers.

\begin{dfn} \label{dfn reduction type}

A \emph{semistable model} of a curve $C / K$ is a model $\mathcal{C} / \mathcal{O}_{K'}$ (for some finite extension $K' / K$) whose special fiber, denoted $\mathcal{C}_s / k$, satisfies that each of its components has multiplicity $1$ and that each of its singular points (if there are any) is a node (\textit{i.e.} an ordinary double point).  Such a model is said to be \emph{(split) degenerate} if the normalization of each component of $\mathcal{C}_s$ is a copy of the projective line $\proj_k^1$, \textit{i.e.} if the geometric genus of each component of $\mathcal{C}_s$ is $0$.

A curve $C / K$ is said to \emph{have semistable (resp. split degenerate) reduction over an extension $K' / K$} if there exists a semistable (resp. split degenerate) model $\mathcal{C} / \mathcal{O}_{K'}$ of $C$.  The curve $C / K$ is called \emph{potentially degenerate} if it obtains split degenerate reduction over a finite extension $K' / K$.

\end{dfn}

\begin{rmk} \label{rmk semistability theorem}

As was shown by Deligne and Mumford in \cite{deligne1969irreducibility} and then through independent arguments by Artin and Winters in \cite{artin1971degenerate}, every smooth, geometrically connected curve $C / K$ obtains semistable reduction over some finite extension $K' / K$.  If some semistable model of $C$ over an extension $K' / K$ is (split) degenerate, then any semistable model of $C$ over any (possibly different) extension $K'' / K$ is also (split) degenerate: this follows from \Cref{rmk abelian and toric}(a) and \Cref{rmk alternate definition of split degenerate} below.

\end{rmk}

Let $p$ be a prime not equal to the characteristic of $K$ (and which may or may not be equal to the residue characteristic of $K$).  In this paper we are concerned with whether or not a given curve $C / K$ is potentially degenerate when $C$ is a $p$-cyclic cover of $\proj_K^1$; such a curve $C$ can always be described by an affine model of the form 
\begin{equation} \label{eq superelliptic}
y^p = f(x) = \prod_{i = 1}^d (x - z_i)^{m_i} \in K[x],
\end{equation}
 where $1 \leq m_i \leq p - 1$ for $1 \leq i \leq d$\footnote{For any $i$ such that $2 \leq m_i \leq p - 1$, the point $(x, y) = (z_i, 0)$ on this curve is a unibranch singularity; by $C$ we will actually mean the curve obtained after resolving each such singularity: see \Cref{prop superelliptic model}(b) below.}.  The set of branch points is then given by $\mathcal{B} := \{z_1, \dots, z_d\}$ (resp. $\mathcal{B} := \{z_1, \dots, z_d, \infty\}$) if $p \mid \deg(f)$ (resp. $p \nmid \deg(f)$); see Appendix \ref{appendix a} for more details.

Throughout this paper, we make the simplifying assumption that we have $\infty \in \mathcal{B}$, so that $d = \#\mathcal{B} - 1$.  An application of the Riemann-Hurwitz formula shows that the genus of $C$ is given by $g = \frac{1}{2}(p - 1)(\#\mathcal{B} - 2) = \frac{1}{2}(p - 1)(d - 1)$; let us further assume for the rest of the paper that $C$ has positive genus, which implies that $d \geq 2$ (and even $d \geq 3$ if $p = 2$).

Let $\Berk$ denote the Berkovich projective line over $\cc_K$ (see \S\ref{sec preliminaries berk} for a definition of this uniquely path connected infinite metric space).  Given any subspace $\Lambda \subset \Berk$ and any real number $r \geq 0$, let $B(\Lambda, r) \subset \Berk$ denote the tubular neighborhood of $\Lambda$ of radius $r$; more precisely, we define $B(\Lambda, r)$ to consist of all points $\eta \in \Berk$ such that there exists a point $\xi \in \Lambda$ whose distance from $\eta$ is $\leq r$.  With this notation, we are able to state our theorem which provides a criterion for potential degeneracy.

\begin{thm} \label{thm main}

A superelliptic curve $C$ defined by an equation of the form given in (\ref{eq superelliptic}) is potentially degenerate if and only if there is a labeling $\mathcal{B} = \{z_0, w_0 := \infty, z_1, w_1, \dots, z_h, w_h\}$ of the branch points which satisfies the below two properties.
\begin{enumerate}[(i)]

\item For each $i \in \{1, \dots, h\}$, the sum of the multiplicities of $z_i$ and $w_i$ (as roots of $f$) equals $p$; in other words, the defining equation in (\ref{eq superelliptic}) can be rewritten as 
\begin{equation} \label{eq superelliptic degenerate model}
y^p = f(x) = (x - z_0)^{m_0} \prod_{i = 1}^h (x - z_i)^{m_i} (x - w_i)^{p - m_i}
\end{equation}
for some integers $m_i$ each satisfying $1 \leq m_i \leq p - 1$.

\item For $0 \leq i \leq h$, let $\Lambda_i \subset \Berk$ denote the path in $\Berk$ which connects the points of Type I which correspond to $z_i$ and $w_i$.  For distinct indices $i, j \in \{0, \dots, h\}$, we have 
\begin{equation}
B(\Lambda_i, \ptfrac) \cap B(\Lambda_j, \ptfrac) = \varnothing.
\end{equation}

\end{enumerate}

\end{thm}

\begin{rmk} \label{rmk translating main thm}

The language of cluster data, introduced in \cite{dokchitser2022arithmetic} and defined in \S\ref{sec preliminaries clusters} of this paper, is an increasingly popular framework for studying the arithmetic of covers of the projective line over discrete valuation rings.  Relatedly, hypotheses expressed in terms of the cluster data associated to a superelliptic curve have the advantage of being intuitive and easy to confirm via computation.  At the same time, cluster language would become unwieldy in the course of some of our arguments.

Property (ii) given in the above theorem can be restated in terms of the \emph{cluster data} of the branch points rather than the paths between them in the Berkovich projective line, by requiring the following (where (ii)(a) is implied by the results of \S\ref{sec preliminaries partitions} below to be equivalent to the disjointness of the spaces $\Lambda_i$ and (ii)(b) is \cite[Proposition 3.10]{yelton2024branch2}).

\begin{enumerate}
\item[(ii)(a)] Let $\mathfrak{s}_i \subseteq \mathcal{B}$ be the smallest cluster containing $z_i, w_i$ for $1 \leq i \leq h$.  The clusters $\mathfrak{s}_0, \dots, \mathfrak{s}_h$ are pairwise distinct and each coincide with $\{z_i, w_i\}_{i \in I}$ for some subset $I \subsetneq \{0, \dots, h\}$.
\item[(ii)(b)] Given any distinct clusters $\mathfrak{c}_1, \mathfrak{c}_2 \subset \mathcal{B}$ which are not themselves the disjoint union of $\geq 2$ even-cardinality clusters, such that $\mathfrak{c}_1$ has even cardinality and $\mathfrak{c}_2$ either has even cardinality or contains $\mathfrak{c}_1$, and letting $\mathfrak{s}$ be the smallest cluster containing both $\mathfrak{c}_1$ and $\mathfrak{c}_2$, we have 
\begin{equation}
d(\mathfrak{c}_1) + d(\mathfrak{c}_2) - 2d(\mathfrak{s}) > 2\ptfrac.
\end{equation}
\end{enumerate}

\end{rmk}

\subsection{The relationship between \Cref{thm main} and results from other works} \label{sec intro comparison}

One direction of \Cref{thm main} has already been proved elsewhere: given a superelliptic curve with split degenerate reduction, property (i) was shown by van Steen as \cite[Proposition 3.1(a)]{van1982galois}, and property (ii) was shown by the author as \cite[Corollary 5.1]{yelton2024branch2}.  Both properties were proved using the theory of Schottky groups and Mumford uniformization of curves with split degenerate reduction.  In \cite[Remark 5.2]{yelton2024branch2}, the author suggested that one should be able to prove not only this direction but its converse through purely geometric arguments.  This paper is the realization of the idea that was expressed in that remark.

Moreover, in certain cases, some rephrasing of both directions of \Cref{thm main} has been shown elsewhere.  In the tame case, one can readily deduce a criterion equivalent to the formulation in \Cref{thm main} for hyperelliptic curves (the $p = 2$ case) from the results of \cite{dokchitser2022arithmetic} and similar works of the same set of authors, whereas for a general tame $n$-cyclic cover of $\proj_K^1$, with some amount of work one can translate the statement of \cite[Proposition 4.4]{brouillard1994action} into our language of Berkovich curves or cluster data and yield \Cref{thm main}.  Meanwhile, in the wild setting with $p = 2$, the result \cite[Corollary 5.2.7]{gehrunger2025computing}, arrived at independently and appearing in a recent preprint of Tim Gehrunger, provides a sufficient property for a hyperelliptic curve $C$ to be potentially degenerate, which one can reformulate in our language to be equivalent to properties (i) and (ii) of \Cref{thm main}.

Finally, in collaboration with Leonardo Fiore, the author showed in \cite{fiore2026clusters} in the wild case for hyperelliptic curves, the associated cluster data can be used to estimate the toric rank (see \S\ref{sec preliminaries geometric models} below) but in general does not determine the toric rank.  The property of being potentially degenerate is equivalent to saying that the toric rank equals the genus $g$: see \Cref{rmk alternate definition of split degenerate} below.  Our main result (\Cref{thm main} with the help of \Cref{rmk translating main thm}) shows that this in turn is equivalent to certain conditions on the cluster data, thus providing a special case where the cluster data does determine the toric rank.  One should be able to generalize the results of \cite{fiore2026clusters} to $p$-cyclic covers for general $p$ and get similar results relating the cluster data to the toric rank but showing that the former does not in general determine the latter, with the case of toric rank $g$ again providing an exception.

\subsection{Outline of the paper} \label{sec intro outline}

In setting out to prove our main results, we begin in \S\ref{sec preliminaries} by defining the cluster data associated to the roots of $f$, then establishing a suitable definition of the Berkovich projective line together with some special subspaces $\Sigma_{\mathcal{B}}^*, \Sigma_{\mathcal{B}}^0 \subseteq \Sigma_{\mathcal{B}} \subset \Berk$, defined with respect to the polynomial $f$ in the equation in (\ref{eq superelliptic}), and then relate all of these objects.  We then use \S\ref{sec preliminaries geometric} to set up certain geometric preliminaries regarding minimal regular models and toric ranks of semistable models, using the structure of $C$ as a Galois cover of $\proj_K^1$ to obtain a particular characterization of the latter.  Each point of Type II of the Berkovich projective line corresponds to a smooth model of the projective line over some finite extension $K' / K$; in \S\ref{sec normalizations}, we compute the (relative) normalizations in the function field extension $K(C) / K(x)$ of these smooth models and investigate the relevant properties of their special fibers.

We begin \S\ref{sec t_f} by defining an important function on the Berkovich projective line measuring ``$p$-power closeness" or how close (with respect to the Gauss valuation induced by $v$) the polynomial $f$ is to being a $p$th power.  This function depends on the defining polynomial $f$ and is denoted $\mathfrak{t}_f$, and in the wild residue characteristic case, it will be used directly to find smooth models of the projective line whose normalizations in $K(C)$ are dominated by the minimal regular model of $C$.  The rest of \S\ref{sec t_f} is then devoted to proving a crucial result (\Cref{thm t_f formula one component} below) which can be used to describe the behavior of the function $\mathfrak{t}_f$ and to obtain a precise formula (given as \Cref{cor t_f formula}) for $\mathfrak{t}_f$ under the hypothesis that the subspaces $\Lambda_i \subset \Berk$ of \Cref{thm main} are mutually disjoint and are far enough apart.  We then tie all of this together to prove \Cref{thm main} in \S\ref{sec proof of main}.

Finally, we explicitly show how to build (split) degenerate models of potentially degenerate superelliptic curves in \S\ref{sec building models} (as \Cref{thm construction of mini}) and produce an explicit field extension over which they are defined.  We finish by classifying the structures of their special fibers for genus $1$ and $2$.

\subsection{Acknowledgments} \label{sec intro acknowledgments}

The author would like to thank Leonardo Fiore for contributing crucial ideas and results, in the context of a previous collaboration, which have been adapted to this project.  He also wishes to thank Alexander Betts for a useful discussion which alerted him to a mistake in a draft of this manuscript, as well as the anonymous referee of a previous preprint for directly inspiring him to make \cite[Remark 5.2]{yelton2024branch2} into its own paper.  Finally, he is extremely grateful to the anonymous referee of this paper, whose very careful and dedicated attention to every part of this long manuscript led to a number of helpful corrections and suggestions and, in a couple of cases, pointed towards a crucially flawed argument that needed to be replaced.

\section{Multisets, clusters, and the Berkovich projective line} \label{sec preliminaries}

Throughout this section, we will deal with the set $\mathcal{B}$ of branch points of the $p$-cyclic cover $C \to \proj_K^1$, which is the underlying set of a multiset $\underline{\mathcal{B}}$ that encodes the branching behavior of the cover.  More precisely, we define the \emph{branch multiset} $\underline{\mathcal{B}}$ to be the multiset consisting of each root of the polynomial $f$ given in (\ref{eq superelliptic}) occurring in the multiset with multiplicity equal to its multiplicity as a root of $f$, along with the element $\infty$ occurring with multiplicity equal to $p\lceil \frac{\deg(f)}{p} \rceil - \deg(f)$.  We note from (\ref{eq superelliptic}) that every element of $\mathcal{B}$ appears with multiplicity $\leq p - 1$.

\subsection{Multisets and cluster data} \label{sec preliminaries clusters}

In order to discuss properties of multisets, and in particular, sub-multisets of $\underline{\mathcal{B}}$, we adopt the following terminology.  When we speak of the cardinality of any multiset $\underline{A}$, we mean the number of elements of $\underline{A}$ counted with multiplicity (\textit{i.e.} the sum of the multiplicities of the elements of the underlying set), and we denote this positive integer by $\#\underline{A}$.  (Note that by construction of $\underline{\mathcal{B}}$, we have $p \mid \#\underline{\mathcal{B}}$.)

For our purposes, every sub-multiset $\underline{B} \subseteq \underline{A}$ is a multiset consisting of certain elements $b$ which appear in $\underline{A}$, where the multiplicity of $b$ in $\underline{B}$ equals the multiplicity of $b$ in $\underline{A}$.  Moreover, given any subset $A \subseteq \mathcal{B}$, we write $\underline{A}$ for the multiset consisting of all of the elements of $A$ each occurring with the same multiplicity that they have in $\underline{\mathcal{B}}$.  Whenever a sub-multiset of $\mathcal{B}$ is denoted using an underline (as in $\underline{A}$), it will be understood that removing the underline indicates the underlying set.  We will freely switch between multisets and their underlying sets using these notational conventions.

We will encounter unions, intersections, and complements of two multisets only when each element appearing in both multisets has the same multiplicity in each of them, and they, as well as the intersection of a multiset and a set, are then defined in the obvious way.

We define the notion of the \emph{cluster data} associated to the branch set $\mathcal{B}$ and all associated terminology following \cite{dokchitser2022arithmetic} while also adapting these definitions to the multiset $\underline{\mathcal{B}}$.

\begin{dfn} \label{dfn cluster}

A subset $\mathfrak{s} \subset \mathcal{B}$ is called a \emph{cluster} if there is some disc $D \subset \bar{K}$ which satisfies $\mathfrak{s} = \mathcal{B} \cap D$.  A sub-multiset of $\underline{\mathcal{B}}$ is called a \emph{cluster} if it coincides with $\underline{\mathfrak{s}}$ for some cluster $\mathfrak{s} \subset \mathcal{B}$.

Given a cluster $\mathfrak{s}$ of $\mathcal{B}$ which does not coincide with $\mathcal{B} \smallsetminus \{\infty\}$ (which itself is always a cluster), we write $\mathfrak{s}' \subset \mathcal{B}$ for the minimum cluster properly containing $\mathfrak{s}$, and we call this the \emph{parent cluster} of $\mathfrak{s}$.  Two clusters with the same parent cluster are called \emph{sibling clusters}.

The \emph{depth} of a cluster $\mathfrak{s} \subset \mathcal{B}$ is the rational number 
\begin{equation}
d(\mathfrak{s}) := \min_{z, z' \in \mathfrak{s}} v(z - z').
\end{equation}
We define the \emph{relative depth} $\delta(\mathfrak{s})$ of $\mathfrak{s}$ to be the difference $d(\mathfrak{s}) - d(\mathfrak{s}')$.  Given a cluster $\underline{\mathfrak{s}}$ of $\underline{\mathcal{B}}$, we define its depth and relative depth respectively to be those of its underlying set $\mathfrak{s}$.

The data of all clusters of $\mathcal{B}$ (resp. $\underline{\mathcal{B}}$) along with each of their depths (or relative depths) is called the \emph{cluster data} of $\mathcal{B}$ (resp. of $\underline{\mathcal{B}}$ or of the curve $C$).

\end{dfn}

\subsection{The Berkovich projective line and convex hulls} \label{sec preliminaries berk}

The \emph{Berkovich projective line} $\Berk$ over the complete algebraically closed field $\cc_K$ is a type of rigid analytification of the projective line $\proj_{\cc_K}^1$ and is typically defined in terms of multiplicative seminorms on $\cc_K[x]$ as in \cite[\S1]{baker2008introduction} and \cite[\S6.1]{benedetto2019dynamics}.  Points of $\Berk$ are identified with multiplicative seminorms which are each classified as Type I, II, III, or IV.  For the purposes of this paper, we may safely ignore points of Type IV and need only adopt a fairly rudimentary construction which does not directly involve seminorms.

\begin{dfn} \label{dfn berk}

We define the \textit{Berkovich projective line}, denoted $\Berk$, to be the topological space with points and topology given as follows.  The points of $\Berk$ are identified with  
\begin{enumerate}[(i)]
\item $\cc_K$-points of $\proj_{\cc_K}^1$, which we will call \emph{points of Type I}; and 
\item discs $D \subset \cc_K$; if $\min\{v(w - z)\}_{z, w \in D} \in \qq$ (resp. $\min\{v(w - z)\}_{z, w \in D} \notin \qq$), we call this a \emph{point of Type II} (resp. a \emph{point of Type III}).
\end{enumerate}
Let $\Hyp \subset \Berk$ be the subset consisting of the points of Type II or III.

A point of $\Berk$ which is identified with a point $z \in \cc_K \cup \{\infty\}$ (resp. a disc $D \subset \cc_K$) is denoted $\eta_z \in \Berk$ (resp. $\eta_D \in \Berk$).

Given two points $\eta, \eta' \in \Berk$, writing $D \subset \cc_K$ for the disc (resp. singleton subset consisting of the point) corresponding to $\eta$ if $\eta$ is of Type II or III (resp. of Type I), and defining $D' \subset \cc_K$ similarly, we write $\eta \vee \eta' \in \Berk$ for the point of Type II or III corresponding to the smallest disc containing both $D$ and $D'$.  We define an order relation (denoted by $>$) on $\Berk$ by decreeing that $\eta > \eta'$ if and only if the corresponding subsets $D, D' \subset \cc_K$ as defined above satisfy $D \supsetneq D'$.  Note that for any $\eta, \eta' \in \Berk$, the point $\eta \vee \eta'$ satisfies $\eta \vee \eta' > \eta, \eta'$ and is minimal (with respect to the ``greater than" relation) for this property.

We define a metric on $\Hyp$ given by the distance function 
\begin{equation*}
\delta : \Hyp \times \Hyp \to \rr
\end{equation*}
 as follows.  For points $\eta = \eta_D, \eta' = \eta_{D'} \in \Hyp$ satisfying $\eta > \eta'$, we let 
\begin{equation}
\delta(\eta, \eta') = d(D) - d(D'), 
\end{equation}
and for general $\eta, \eta' \in \Hyp$, we let 
\begin{equation}
\delta(\eta, \eta') = \delta(\eta, \eta \vee \eta') + \delta(\eta', \eta \vee \eta').
\end{equation}

We endow the subset $\Hyp \subset \Berk$ with the topology induced by the metric given by $\delta$, and we extend this to a topology on all of $\Berk$ in such a way that, for any $w \in \cc_K$, the sequence $\{\eta_{(w), i}\}_{i = 1, 2, 3, \dots}$ (resp. $\{\eta_{(w), i}\}_{i = -1, -2, -3, \dots}$) converges to $\eta_w$ (resp. $\eta_\infty$), where $\eta_{(w), i}$ corresponds to the disc $\{z \in \cc_K \ | \ v(z - w) \geq i\}$ for all $i \in \zz$.

\end{dfn}

It is well known that the space $\Berk$ is path connected and that there is a unique non-backtracking path between any pair of points in $\Berk$.  This allow us to set the following notation.  Below we denote the image in $\Berk$ of the non-backtracking path between two points $\eta, \eta' \in \Berk$ by $[\eta, \eta'] \subset \Berk$, and we will often refer to this image itself as ``the path'' from $\eta$ to $\eta'$.  We write $(\eta, \eta') := [\eta, \eta'] \smallsetminus \{\eta, \eta'\}$ for its interior, \textit{i.e.} the open path from $\eta$ to $\eta'$, and we use the notation $(\eta, \eta']$ and $[\eta, \eta')$ for the obvious half-open paths.  Note that the order of the endpoints of the paths given in the above notation does not affect the paths.  The above observations imply that, given a point $\eta \in \Hyp$ and a connected subspace $\Lambda \in \Berk$ that intersects nontrivially with $\Hyp$, there is a (unique) closest point $\xi \in \Lambda$ to $\eta$, which allows us to define the distance between $\eta$ and $\Lambda$ as $\delta(\eta, \Lambda) = \delta(\eta, \xi)$.  Given a subspace $\Lambda \subset \Berk$ and a real number $r \geq 0$, we define the \emph{(closed) tubular neighborhood} of $\Lambda$ of radius $r$ to be 
\begin{equation} \label{eq tubular neighborhood}
B(\Lambda, r) = \{\eta \in \Berk \ | \ \delta(\eta, \Lambda) \leq r\}.
\end{equation}

We make note of the very important fact that given points $\eta, \eta', \xi \in \Berk$ with $\eta, \eta' > \xi$, we must have $\eta \geq \eta'$ or $\eta' \geq \eta$, coming from the fact that the relation $\eta > \xi$ can also be defined as saying that $\eta$ lies in the path $[\xi, \eta_\infty]$.  This fact may be phrased as saying that ``there is only one direction upwards from any point in $\Berk \smallsetminus\{\eta_\infty\}$."  It is easy to verify from this and from unique path connectedness that each connected subspace of $\Berk$ has at most one unique maximal point with respect to $>$, a fact that we will often use freely in arguments below.

Given a cluster $\mathfrak{s} \subset \mathcal{B}$, we write $D_{\mathfrak{s}} \subset \cc_K$ for the smallest disc which contains $\mathfrak{s}$, and we write $\eta_{\mathfrak{s}} \in \Hyp$ for the corresponding point of Type II.

\begin{rmk} \label{rmk delta}

It is no coincidence that we are using $\delta$ to denote both the relative depth of a cluster and the distance function on $\Hyp$.  Indeed, one computes directly from definitions that we have 
\begin{equation}
\delta(\mathfrak{s}) = \delta(\eta_{\mathfrak{s}}, \eta_{\mathfrak{s}'}) \text{ for any cluster } \mathfrak{s} \subsetneq \mathcal{B} \smallsetminus \{\infty\}.
\end{equation}

\end{rmk}

\begin{dfn} \label{dfn Sigma}

Given any subset $A \subset \bar{K} \cup \{\infty\}$, we write $\Sigma_A \subset \Berk$ for the intersection of $\Hyp$ with the convex hull of the subspace $\{\eta_z\}_{z \in A} \subset \Berk$, \textit{i.e.} the smallest connected subspace of $\Hyp$ whose set of limit points contains $\{\eta_z\}_{z \in A}$.

We call a point of $v \in \Sigma_A$ a \emph{vertex} if the space $\Sigma_A \smallsetminus \{a\}$ has $\geq 3$ connected components.

\end{dfn}

Given a finite subset $A \subset \bar{K} \cup \{\infty\}$ of cardinality $\geq 2$, it is clear that the space $\Sigma_A$ coincides with the union of open paths between each pair of distinct points in $A$ (considered as points in $\Berk$ of Type I); that is, we have $\Sigma_A = \bigcup_{z, w \in A, \,z \neq w} (\eta_z, \eta_w)$ (note that this is not an empty union as we have $\#A \geq 2$).  Moreover, one sees directly from \Cref{dfn berk} that each open path $(\eta_z, \eta_w)$ consists of the points $\eta_D$ for all discs $D$ satisfying $\#(D \cap \{z, w\}) = 1$ as well as for the smallest disc containing both $z$ and $w$, which we denote by $D(z, w)$; the corresponding point $\eta_{D(z, w)}$ is the maximal point of $[\eta_z, \eta_w]$ with respect to the relation $>$.  It is also not hard to see from this characterization of $\Sigma_A$ that it contains all of its limit points not of Type I (or equivalently, $\Sigma_A \cup \{\eta_z\}_{z \in A}$ is closed in $\Berk$).  We freely use these facts in some arguments below.

The following proposition fully characterizes the points in the spaces $\Sigma_A$ when $A$ consists of branch points (but can easily be generalized to apply to any finite subset $A \subset \bar{K} \cup \{\infty\}$).

\begin{prop} \label{prop highest point of Lambda}

Let $\mathfrak{r} \subset \mathcal{B}$ be any subset.

\begin{enumerate}[(a)]
\item There is a (unique) maximal point of $\Sigma_{\mathfrak{r}}$ if and only if we have $\infty \notin \mathfrak{r}$, and in this case, it is the point $\eta_{\mathfrak{s}}$, where $\mathfrak{s}$ is the smallest cluster containing $\mathfrak{r}$.
\item A point $\eta = \eta_D \in \Hyp$ is a non-maximal point of $\Sigma_{\mathfrak{r}}$ if and only if there are points $z, w \in \mathfrak{r}$ with $z \in D$ and $w \notin D$.
\end{enumerate}

\end{prop}

\begin{proof}
For the proof of part (a), it is clear that there is no maximal point of $\Sigma_{\mathfrak{r}}$ if $\infty \in \mathfrak{r}$, so let us assume that we have $\infty \notin \mathfrak{r}$.  As above, for any distinct $z, w \in \mathfrak{r}$, let $D(z, w) \subset \cc_K$ denote the smallest disc containing both $z$ and $w$.  Now fixing $z = z'$ and letting $w$ vary among all points in $\mathfrak{r} \smallsetminus \{z'\}$, the set of discs $D(z', w)$ all have a common center (namely $z'$) and therefore form a well-ordered set with respect to inclusion.  Choose $w'$ such that $D(z', w')$ is maximal with respect to inclusion within this set of discs.  Then the disc $D(z', w')$ contains all other points in $\mathfrak{r}$ and thus coincides with $D_{\mathfrak{s}}$.  We therefore have $\eta_{\mathfrak{s}} = \eta_{D(z', w')} \in (\eta_{z'}, \eta_{w'}) \subseteq \Sigma_{\mathfrak{r}}$.  Moreover, for any $\eta \in \Sigma_{\mathfrak{r}}$, since we have $\eta \in [\eta_z, \eta_w]$ for some $z, w \in \mathfrak{r}$, we have $\eta_{\mathfrak{s}} = \eta_{D(z', w')} \geq \eta_{D(z, w)} \geq \eta$, thus demonstrating the maximality of $\eta_{\mathfrak{s}}$ and proving part (a).

Now let $\eta = \eta_D$ be a non-maximal point of $\Sigma_{\mathfrak{r}}$.  Then if $\infty \in \mathfrak{r}$ (resp. $\infty \notin \mathfrak{r}$), we have $w := \infty \notin D$ (resp. $D \subsetneq D_{\mathfrak{s}}$, so there is some $w \in \mathfrak{r}$ with $w \notin D$).  Since $\Sigma_{\mathfrak{r}}$ is the union of the open paths of the form $(\eta_z, \eta_{z'})$ with $z, z' \in \mathfrak{r}$, and since each point on such a path corresponds to a disc containing $z$ or $z'$, there is some point $z \in \mathfrak{r}$ with $z \in D$.

Conversely, suppose that $\eta = \eta_D \in \Hyp$ is a point such that there exist $z, w \in \mathfrak{r}$ with $z \in D$ and $w \notin D$.  Then we have $\eta \in [\eta_z, \eta_w] \subseteq \Sigma_{\mathfrak{r}}$.  At the same time, if $\infty \notin \mathfrak{r}$, we have $D \neq D_{\mathfrak{s}} \ni w$, or equivalently, $\eta \neq \eta_{\mathfrak{s}}$, meaning that $\eta$ is not maximal in $\Sigma_{\mathfrak{r}}$.
\end{proof}

\begin{prop} \label{prop removing a point from convex hull}

Given a point $\eta = \eta_D \in \Sigma_{\mathcal{B}}$, denote by $U_1, \dots, U_s, U_\infty$ the connected components of $\Sigma_{\mathcal{B}} \smallsetminus \{\eta\}$, named so that $\eta_\infty$ is a limit point of $U_\infty$, and for $1 \leq i \leq s$, let $\mathfrak{c}_i \subset \mathcal{B}$ be the subset consisting of the points $z$ such that $\eta_z$ is a limit point of $U_i$.

The subsets $\mathfrak{c}_1, \dots , \mathfrak{c}_s \subset \mathcal{R}$ are all clusters, and their disjoint union $\mathfrak{s} := \mathfrak{c}_1 \sqcup \dots \sqcup \mathfrak{c}_s$ is the cluster coinciding with $D \cap \mathcal{B}$.  We moreover have $s \geq 2$ (or equivalently, the point $\eta$ is a vertex of $\Sigma_{\mathcal{B}}$) if and only if $\eta = \eta_{\mathfrak{s}}$, and in this case, the clusters whose parents coincide with $\mathfrak{s}$ are precisely the clusters $\mathfrak{c}_1, \dots, \mathfrak{c}_s$.

\end{prop}

\begin{proof}
It follows from the real tree structure of $\Sigma_{\mathcal{B}}$ that the connected components of $\Sigma_{\mathcal{B}}$ are in bijection with the equivalence classes of (non-singleton) paths in $\Sigma_{\mathcal{B}}$ having $\eta$ as one of their endpoints, where two (non-singleton) paths $[\eta, \xi]$ and $[\eta, \xi']$ are considered equivalent if their intersection is not a singleton (and thus coincides with another non-singleton path $[\eta, \xi'']$); this bijection assigns a connected component $U_i$ to the equivalence class represented by the path $[\eta, \xi]$ where $\xi$ is any point in $U_i$.  For any $\xi, \xi' > \eta$, one easily sees that the paths $[\eta, \xi]$ and $[\eta, \xi']$ are equivalent as one is contained in the other.  One also sees that conversely, for any point $\xi''$ with $\eta > \xi''$, we have $[\eta, \xi''] \cap [\eta, \xi] = \{\eta\}$, and so $[\eta, \xi'']$ does not lie in this equivalence class.  It follows that any $z \in \mathcal{B}$ satisfies that $\eta_z \in \bar{U}_\infty$ (where $\bar{U}_\infty$ is the topological closure of $U_\infty$) if and only if we have $\eta \vee \eta_z > \eta$ (as we have $[\eta, \eta \vee \eta_z] \subset [\eta, \eta_z]$).  Thus, the elements $z \in \mathcal{B} \smallsetminus \bar{U}_\infty = \mathfrak{c}_1 \sqcup \dots \sqcup \mathfrak{c}_s = \mathfrak{s}$ are exactly those branch points satisfying $\eta \vee \eta_z = \eta$, or equivalently, $\eta > \eta_z$, which in turn is equivalent to $z \in D$.  We therefore have $\mathfrak{s} = \mathcal{B} \cap D$.

Now fix any index $i \neq \infty$.  Since the intersection of the paths $[\eta_z, \eta]$ and $[\eta_w, \eta]$ for any $z, w \in \mathfrak{c}_i$ is a (non-singleton) path of the form $[\xi, \eta]$ with $\eta > \xi$, the intersection $\bigcap_{z \in \mathfrak{c}_i} [\eta_z, \eta]$ is a (non-singleton) path of that form, which we again denote by $[\xi, \eta]$.  Meanwhile, for any $z \in \mathfrak{c}_i$ and $w \in \mathfrak{s} \smallsetminus \mathfrak{c}_i$, we have $[z, \eta] \cap [w, \eta] = \{\eta\}$; it follows that we have $[\xi, \eta] \cap [w, \eta] = \{\eta\}$.  Therefore, for any $\eta' = \eta_{D'}$ lying in the interior of $[\xi, \eta]$ and any $z \in \mathfrak{s}$, we have $\eta' > \eta_z$ (or equivalently, $z \in D'$) if and only if $z \in \mathfrak{c}_i$.  We then get $\mathcal{B} \cap D' = \mathfrak{c}_i$, so the subset $\mathfrak{c}_i \subset \mathcal{B}$ is cluster.

The topological structure of $\Sigma_{\mathcal{B}}$ as a real tree implies that removing the point $\eta$ from it makes it no longer a connected space; since the points of $\Sigma_{\mathcal{B}}$ of Type I are endpoints, we again get that the space $\Sigma_{\mathcal{B}} \smallsetminus \{\eta\}$ is not connected, implying that $s \geq 1$. Now if $s = 1$, then, as in the above paragraph, there is a point $\eta' = \eta_{D'} < \eta$ such that we have $z \in D'$ for all $z \in \mathfrak{c}_1 = \mathfrak{s}$.  We therefore have $\mathcal{B} \cap D' = \mathfrak{c}_1 = \mathfrak{s}$.  Since we have $D' \subsetneq D$, this implies that $D \neq D_{\mathfrak{s}}$.  Conversely, if we have $s \geq 2$, then, as argued in the above paragraph, for any $z \in \mathfrak{c}_1$ and $w \in \mathfrak{c}_2$, we have $[\eta_z, \eta] \cap [\eta_w, \eta] = \{\eta\}$, which in fact is equivalent to saying that $\eta = \eta_z \vee \eta_w$.  Now we clearly have $\eta_{\mathfrak{s}} > \eta_z, \eta_w$ and thus $\eta_{\mathfrak{s}} \geq \eta_z \vee \eta_w = \eta$.  The fact that $D \supset \mathfrak{s}$ gives us the reverse inequality $\eta \geq \eta_{\mathfrak{s}}$, and so we have $\eta = \eta_{\mathfrak{s}}$.

Now, retaining our assumption that $s \geq 2$ and choosing any index $i \in \{1, \dots, s\}$, the cluster $\mathfrak{s}$, as it strictly contains $\mathfrak{c}_i$, clearly contains the parent cluster of $\mathfrak{c}_i$, \textit{i.e.} we have $\mathfrak{s} \supseteq \mathfrak{c}_i'$.  If this containment is strict, then, by basic properties of clusters, there must be an index $j \neq i$ such that we have $\mathfrak{c}_i' \supseteq \mathfrak{c}_j$, and for any $z \in \mathfrak{c}_i, w \in \mathfrak{c}_j$, the paths $[\eta_z, \eta_{\mathfrak{s}}]$ and $[\eta_w, \eta_{\mathfrak{s}}]$ intersect at the (non-singleton) path $[\eta_{\mathfrak{c}_i'}, \eta_{\mathfrak{s}}]$, which contradicts the fact that $\eta_z$ and $\eta_w$ are limit points of distinct connected components $U_i$ and $U_j$.  It follows that we have $\mathfrak{s} = \mathfrak{c}_i'$.  Now if $\mathfrak{c}$ is any cluster whose parent is $\mathfrak{s} = \mathfrak{c}_1 \sqcup \dots \sqcup \mathfrak{c}_s$, then we have $\mathfrak{c} \subseteq \mathfrak{c}_i$ for some $i \in \{1, \dots, s\}$, and by definition of parent cluster, equality holds.  This proves the last statement of the proposition.
\end{proof}

\subsection{Partitions into multisets of cardinality divisible by $p$} \label{sec preliminaries partitions}

In this subsection, we will deal with situations where we are given a partition $\underline{\mathcal{B}} = \underline{\mathfrak{r}}_0 \sqcup \dots \sqcup \underline{\mathfrak{r}}_h$ into non-empty sub-multisets of cardinality divisible by $p$.  We note, and will freely and implicitly use below, the fact that each of the underlying sets $\mathfrak{r}_i$ must have cardinality $\geq 2$, because no element of $\mathcal{B}$ can have multiplicity $\geq p$.  Whenever we are given such a partition, we assume that $\infty \in \mathfrak{r}_0$, and we define $\mathfrak{s}_i \subset \mathcal{B}$ to be the smallest cluster containing $\mathfrak{r}_i$ for $1 \leq i \leq h$.

The following properties that the branch multiset $\underline{\mathcal{B}}$ may have will be crucial going forward, in particular because being \emph{clustered in pairs} will turn out to be a necessary (and, in the tame situation, sufficient) condition for the associated curve $C$ to be potentially degenerate.

\begin{dfn}[cluster version] \label{dfn clustered in pairs clusters}

We say that $\underline{\mathcal{B}}$ is \emph{clustered in $p$-multisets} if there is a partition $\underline{\mathcal{B}} = \underline{\mathfrak{r}}_0 \sqcup \dots \sqcup \underline{\mathfrak{r}}_h$ into mutually disjoint sub-multisets of cardinality $p$ satisfying that the corresponding clusters $\underline{\mathfrak{s}}_0, \dots, \underline{\mathfrak{s}}_h$ are pairwise distinct and each have cardinality divisible by $p$.

If, in addition, we have that the underlying set $\mathfrak{r}_i$ of each cardinality-$p$ multiset $\underline{\mathfrak{r}}_i$ has cardinality $2$, we say that $\underline{\mathcal{B}}$ is \emph{clustered in pairs}.


\end{dfn}

We will later reframe this property in topological terms as \Cref{dfn clustered in pairs berk} below.  With an aim towards this, we begin by defining some important subspaces of $\Sigma_{\mathcal{B}} \subset \Hyp$.

\begin{dfn} \label{dfn Sigma^*}

We define $\Sigma_{\mathcal{B}}^* \subset \Sigma_{\mathcal{B}}$ to be the subspace consisting of the points $\eta \in \mathcal{B}$ such that, letting $U_1, \dots, U_s, U_\infty$ denote the connected components of $\mathcal{B} \smallsetminus \{\eta\}$ and $\mathfrak{c}_1, \dots, \mathfrak{c}_s$ be the corresponding clusters as in \Cref{prop removing a point from convex hull}, we have $p \nmid \#\underline{\mathfrak{c}}_i$ for some $i \in \{1, \dots, s\}$.

We define the subspace $\Sigma_{\mathcal{B}}^0 \subset \Sigma_{\mathcal{B}}$ to be the topological closure of $\Sigma_{\mathcal{B}} \smallsetminus \Sigma_{\mathcal{B}}^*$ in $\Sigma_{\mathcal{B}}$.

\end{dfn}

\begin{prop} \label{prop connected components Sigma}

Each connected component of $\Sigma_{\mathcal{B}}^*$ coincides with $\Sigma_{\mathfrak{r}}$ for some (non-empty) subset $\mathfrak{r} \subset \mathcal{B}$ with $p \mid \#\underline{\mathfrak{r}}$ such that (if $\infty \notin \mathfrak{r}$) the smallest cluster $\mathfrak{s}$ containing $\mathfrak{r}$ satisfies $p \mid \#\underline{\mathfrak{s}}$.

The set of the connected components of $\Sigma_{\mathcal{B}}^*$ therefore determines a partition $\underline{\mathcal{B}} = \underline{\mathfrak{r}}_0 \sqcup \dots \sqcup \underline{\mathfrak{r}}_h$ into sub-multisets (whose underlying sets are mutually disjoint) with $\infty \in \mathfrak{r}_0$, satisfying that, for $1 \leq i \leq h$, the smallest cluster $\mathfrak{s}_i$ containing $\mathfrak{r}_i$ satisfies $p \mid \#\underline{\mathfrak{s}}_i$.

Moreover, the clusters $\mathfrak{s}_i$ are pairwise distinct, and each is the disjoint union of a collection among the sets $\mathfrak{r}_1, \dots, \mathfrak{r}_h$.

\end{prop}

\begin{ex} \label{ex connected components Sigma}

Suppose that $\pi \nmid 2, 3$ is a uniformizer of $K$, that $p = 3$, and that 
$$\underline{\mathcal{B}} = \{0, \pi^2, 2\pi^2, 3\pi^2, \pi, \pi, 1, 1, 1 + \pi, 1 - \pi, 1 - \pi, 1 - \pi + \pi^2, \pi^{-1}, -\pi^{-1}, \infty\}.$$
In this case, the sub-multisets $\underline{\mathfrak{r}}$ and their corresponding clusters $\underline{\mathfrak{s}}$ given by \Cref{prop connected components Sigma} are 
\begin{equation}
\begin{aligned}
&\underline{\mathfrak{r}}_0 = \{\infty, \pi^{-1}, -\pi^{-1}\}, \ \underline{\mathfrak{r}}_1 = \underline{\mathfrak{s}}_1 = \{0, \pi^2, 2\pi^2, 3\pi^2, \pi, \pi\}, \\
\underline{\mathfrak{r}}_2 = \underline{\mathfrak{s}}_2 &= \{1 - \pi, 1 - \pi, 1 - \pi + \pi^2\}, \ \underline{\mathfrak{r}}_3 = \{1, 1, 1 + \pi\}, \ \underline{\mathfrak{s}}_3 = \underline{\mathfrak{r}}_2 \sqcup \underline{\mathfrak{r}}_3.
\end{aligned}
\end{equation}
One verifies that the cardinalities of these sub-multisets of $\underline{\mathcal{B}}$ are each $3$ or $6$ (so divisible by $p = 3$), and from the diagram of $\Sigma_{\mathcal{B}}$ in \Cref{fig connected components Sigma} below (where solid lines indicate $\Sigma_{\mathcal{B}}^0$, dashed lines indicate $\Sigma_{\mathcal{B}}^*$, and doubled lines go towards multiplicity-$2$ elements of $\underline{\mathcal{B}}$) one can readily check that we have the inclusions $\Sigma_{\mathfrak{r}_i} \subset \Sigma_{\mathcal{B}}^*$ for $1 \leq i \leq 4$ and that these subspaces are connected and mutually disjoint.

\end{ex}

\begin{figure}[h!]

\centering
\begin{tikzpicture}

    \coordinate (O) at (0, 0);
    \coordinate (S0) at (0, 2.236);
    \coordinate (S1) at (-2,-1);
    \coordinate (S1') at (-4,-2);
    \coordinate (S3) at (2,-1);
    \coordinate (S2) at (4,-2);

    \draw[thick] (O) -- (S0);
    \draw[thick] (O) -- (S1);
    \draw[dashed] (S1) -- (S1');
    \draw[thick] (O) -- (S2);

    \draw[dashed] (S0) -- (0,3);
    \draw[dashed] (S0) -- (7,-3);
    \draw[dashed] (S0) -- (8.5,-3);
    \draw[dashed] (S1') -- (-5,-3);
    \draw[dashed] (S1') -- (-4.25,-3);
    \draw[dashed] (S1') -- (-3.5,-3);
    \draw[dashed] (S1') -- (-2.75,-3);
    \draw[dashed] (S1) -- (-1,-3);
    \draw[dashed] (-2.1,-1.05) -- (-1.125,-3);
    \draw[dashed] (S3) -- (.5,-3);
    \draw[dashed] (1.9,-.95) -- (.3625,-3);
    \draw[dashed] (S3) -- (1.75,-3);
    \draw[dashed] (S2) -- (3,-3); 
    \draw[dashed] (3.9,-1.95) -- (2.85,-3);
    \draw[dashed] (S2) -- (5,-3);

    \node[right=-2em] at (S0) {$\eta_{\mathfrak{s}_0}$};
    \node[right=-1em, above=-.2em] at (S1) {$\eta_{\mathfrak{s}_1}$};
    \node[right=.2em, above=.1em] at (S2) {$\eta_{\mathfrak{s}_2}$};
    \node[right=.2em, above=.1em] at (S3) {$\eta_{\mathfrak{s}_3}$};

    \node [above=.1em] at (0,3) {$\infty$};
    \node[below=.1em] at (7,-3) {$\pi^{-1}$};
    \node[below=.1em] at (8.5,-3) {$-\pi^{-1}$};
    \node[below=.1em] at (-5,-3) {$0$};
    \node[below=.1em] at (-4.25,-3) {$\pi^2$};
    \node[below=.1em] at (-3.4375,-3) {$2\pi^2$};
    \node[below=.1em] at (-2.625,-3) {$3\pi^2$};
    \node[below=.1em] at (-1,-3) {$\pi$};
    \node[below=.1em] at (.3625,-3) {$1$};
    \node[below=.1em] at (1.6125,-3) {$1 + \pi$};
    \node[below=.1em] at (3,-3) {$1 - \pi$};
    \node[below=.1em] at (5,-3) {$1 - \pi + \pi^2$};



\end{tikzpicture}

\caption{The space $\Sigma_{\mathcal{B}}$ in \Cref{ex connected components Sigma}}

\label{fig connected components Sigma}
\end{figure}

\begin{proof}[Proof (of \Cref{prop connected components Sigma})]
Let $\Sigma$ be a connected component of the space $\Sigma_{\mathcal{B}}^*$, and choose any point $\eta_{(0)} \in \Sigma$.  Let $\mathfrak{c}_1, \dots, \mathfrak{c}_s, \mathfrak{c}_\infty$ be the corresponding clusters as in \Cref{prop removing a point from convex hull}.  By definition of $\Sigma_{\mathcal{B}}^*$, there is an index $i$ such that $p \nmid \#\underline{\mathfrak{c}}_i$.  Now for each point $\eta = \eta_D \in \Hyp$ where $D$ is a disc satisfying $D \cap \mathcal{B} = \mathfrak{c}_i$, it is clear that we have $\eta_D \in \Sigma_{\mathcal{B}}^*$, so that given such a point $\eta_{(1)}$, the whole path $[\eta_{(0)}, \eta_{(1)}]$ is contained in $\Sigma_{\mathcal{B}}^*$; by connectedness of paths, this means that we have $[\eta_{(0)}, \eta_{(1)}] \subset \Sigma$.  If $\#\mathfrak{c}_i \geq 2$, let us fix $\eta_{(1)}$ to be the vertex $\eta_{\mathfrak{c}_i}$.  Now let $\mathfrak{c}^{(1)}_1, \dots, \mathfrak{c}^{(1)}_{s_1}, \mathfrak{c}^{(1)}_\infty$ be the clusters corresponding to $\eta_{(1)}$ as in \Cref{prop removing a point from convex hull}.  We have $\mathfrak{c}_i = \mathfrak{c}^{(1)}_1 \sqcup \dots \sqcup \mathfrak{c}^{(1)}_{s_1}$, so there is some index $i_1$ such that $p \nmid \#\underline{\mathfrak{c}}^{(1)}_{i_1}$.  If $\#\mathfrak{c}^{(1)}_{i_1} \geq 2$, then a repetition of our previous arguments shows that we have $[\eta_{(1)}, \eta_{(2)}] \subset \Sigma$, where $\eta_{(2)}$ is the vertex corresponding to the cluster $\mathfrak{c}^{(1)}_{i_1}$.  Repeating this process, for some $r \geq 0$, we have $\#\mathfrak{c}^{(r)}_{i_r} = 1$, so that $\mathfrak{c}^{(r)}_{i_r} = \{\eta_z\}$ for some $z \in \mathfrak{c}_i$; then our arguments show that the half-open path $[\eta_{(r-1)}, \eta_z)$ is contained in $\Sigma$.  We have thus shown that for each point $\eta_{(0)} \in \Sigma$, there is a point $z \in \mathcal{B} \smallsetminus \{\infty\}$ such that we have $[\eta_{(0)}, \eta_z) \subset \Sigma$ and get that $\eta_z$ is a limit point of $\Sigma$.

Let $\mathfrak{r} \subset \mathcal{B}$ be the subset consisting of all points $z$ such that $\eta_z$ is a limit point of $\Sigma$.  By connectedness of $\Sigma$, we have $\Sigma \supseteq \Sigma_{\mathfrak{r}}$.  We now set out to show the reverse containment.  Choose any point $\eta = \eta_D \in \Sigma$, and let $z \in \mathcal{B} \smallsetminus \{\infty\}$ be a point such that the half-open path $[\eta, \eta_z)$ is contained in $\Sigma$.  By construction of $\Berk$, we certainly have $\eta \in (\eta_z, \eta_\infty)$; moreover, if $\infty \in \mathfrak{r}$, we also have $(\eta_z, \eta_\infty) \subset \Sigma_{\mathfrak{r}}$, so in this case we get $\eta \in \Sigma_{\mathfrak{r}}$, as desired.  We therefore now assume that $\infty \notin \mathfrak{r}$.  Then the space $\Sigma$, being connected, has a unique maximal point.  Denote this maximal point by $\xi$, and now let $\mathfrak{c}_1, \dots, \mathfrak{c}_s, \mathfrak{c}_\infty$ be the clusters corresponding to $\xi$ as in \Cref{prop removing a point from convex hull}.  For any point $\eta' = \eta_{D'} > \xi$, noting that $\infty \in \mathcal{B}$, we have $\eta' \in (\xi, \eta_\infty) \subset \Sigma_{\mathcal{B}}$.  If $\eta'$ is chosen to be close enough to $\xi$, then we also have $\eta' \notin \Sigma_{\mathcal{B}}^*$ as well as $D' \cap \mathcal{B} = \mathfrak{c}_1 \sqcup \dots \sqcup \mathfrak{c}_s$.  By definition of $\Sigma_{\mathcal{B}}^*$, we then get 
\begin{equation} \label{eq p divides union of clusters}
p \mid \#(D \cap \mathcal{B}) = \#(\underline{\mathfrak{c}}_1 \sqcup \dots \sqcup \underline{\mathfrak{c}}_s) = \sum_{i = 1}^s \#\underline{\mathfrak{c}}_i.
\end{equation}
By definition of $\Sigma_{\mathcal{B}}^*$, there is some index $i$ such that $p \nmid \#\underline{\mathfrak{c}}_i$; then by (\ref{eq p divides union of clusters}), there is another index $j \neq i$ such that $p \nmid \#\underline{\mathfrak{c}}_j$.  In particular, we have $s \geq 2$, and by \Cref{prop removing a point from convex hull}, the maximal point $\xi = \eta_{\mathfrak{c}_1 \sqcup \dots \sqcup \mathfrak{c}_s}$ is a vertex of $\Sigma_{\mathcal{B}}$.  Now, recalling that we have a chosen point $\eta = \eta_D \in \Sigma$ along with a point $z \in \mathfrak{r}$ satisfying $[\eta, \eta_z) \subset \Sigma$, for any point $\eta' = \eta_{D'} \in (\eta, \xi)\subset \Sigma$ sufficiently close to $\xi$, we have $D' \cap \mathcal{B} = \mathfrak{c}_i$ for some index $i$ (noting that by \Cref{prop removing a point from convex hull}, the common parent cluster of $\mathfrak{c}_1, \dots, \mathfrak{c}_s$ is their disjoint union).  As $\eta' > \eta_z$, it follows that we have $z \in \mathfrak{c}_i$.  As was argued earlier in this proof for a path denoted $[\eta_{(0)}, \eta_{(1)}]$, we have $[\xi, \eta_{\mathfrak{c}_j}] \subset \Sigma$.  As $\eta_{\mathfrak{c}_j}$ is a point of $\Sigma$, there is a point $w \in \mathfrak{r}$ such that the half-open path $[\eta_{\mathfrak{c}_j}, \eta_w)$ is contained in $\Sigma$.  Then we have $w \in \mathfrak{c}_j$ and so $w \notin \mathfrak{c}_i \supseteq D \cap \mathcal{B}$.  But from $\eta_D = \eta > \eta_z$, we get $z \in D \cap \mathcal{B}$.  It follows that $\eta = \eta_D$ lies in the open path $(\eta_z, \eta_w) \subset \Sigma_{\mathfrak{r}}$, and we are done showing the reverse inclusion.

We have now established that the spaces $\Sigma$ and $\Sigma_{\mathfrak{r}}$ are equal; in particular, by \Cref{prop highest point of Lambda}(a), their maximal points are the same and so we have $\mathfrak{s} = \mathfrak{c}_1 \sqcup \dots \sqcup \mathfrak{c}_s$.  Then by (\ref{eq p divides union of clusters}), we have $p \mid \#\underline{\mathfrak{s}}$.

It follows from all of this that the space $\Sigma_{\mathcal{B}}^*$, written as the disjoint union of its connected components, is $\Sigma_{\mathfrak{r}_0} \sqcup \dots \sqcup \Sigma_{\mathfrak{r}_h}$, where $\mathfrak{B} = \mathfrak{r}_0 \sqcup \dots \sqcup \mathfrak{r}_h$ is a partition, and where the smallest cluster $\mathfrak{s}_i$ containing $\mathfrak{r}_i$ satisfies $p \mid \#\underline{\mathfrak{s}}_i$ for $0 \leq i \leq h$.  If we have $\mathfrak{s}_i = \mathfrak{s}_j$ for some $i \neq j$, then by \Cref{prop highest point of Lambda}(a), the spaces $\Sigma_{\mathfrak{r}_i}$ and $\Sigma_{\mathfrak{r}_j}$ share the same maximal point, contradicting the fact that they are distinct connected components; therefore, the clusters $\mathfrak{s}_i$ are pairwise distinct.

We now show that each $\mathfrak{s}_i$ is a disjoint union of some collection of the $\mathfrak{r}_j$'s.  If this were not the case, then we would have $\mathfrak{r}_j \cap \mathfrak{s}_i \neq \varnothing$ but $\mathfrak{r}_j \subsetneq \mathfrak{s}_i$ for some $j \in \{0, \dots, h\}$.  Taking points $z \in \mathfrak{r}_j \cap \mathfrak{s}_i$ and $w \in \mathfrak{r}_j \smallsetminus \mathfrak{s}_i$, we have $\eta_{\mathfrak{s}_i} \in (\eta_z, \eta_w) \subseteq \Sigma_{\mathfrak{r}_j}$, but as $\eta_{\mathfrak{s}_i}$ is the maximal point of $\Sigma_{\mathfrak{r}_i}$ by \Cref{prop highest point of Lambda}(a), we get $\Sigma_{\mathfrak{r}_i} \cap \Sigma_{\mathfrak{r}_j} \neq \varnothing$, a contradiction.

Now suppose that there is an index $i$ such that we have $p \nmid \#\underline{\mathfrak{r}}_i$.  As we have $p \mid \#\underline{\mathcal{B}} = \sum_{i = 0}^h \#\underline{\mathfrak{r}}_i$, we may assume that $i \neq 0$; we may also assume that $i$ is minimal in the sense that for any other index $j \neq 0$ with $p \nmid \#\underline{\mathfrak{r}}_j$, the cluster $\mathfrak{s}_j$ is not properly contained in $\mathfrak{s}_i$.  We have shown that the cluster $\mathfrak{s}_i$ must be the union of $\mathfrak{r}_i$ and a collection of $\mathfrak{r}_j$'s; since $p$ divides $\#\underline{\mathfrak{s}}_i$ but not $\#\underline{\mathfrak{r}}_i$, there must be an index $j$ satisfying $\mathfrak{r}_j \subset \mathfrak{s}_i$ (which implies $\mathfrak{s}_j \subseteq \mathfrak{s}_i$) and $p \nmid \#\underline{\mathfrak{r}}_j$.  Our minimality assumption then implies $\mathfrak{s}_j = \mathfrak{s}_i$, contradicting pairwise distinctness.
\end{proof}

\begin{prop} \label{prop cardinality-p clusters}

Let $\mathfrak{s} \subset \mathcal{B}$ be a cluster with $p \mid \#\underline{\mathfrak{s}}$ and which is not the disjoint union of $\geq 2$ clusters that also, when considered as multisets, have cardinality divisible by $p$.  We have that $\mathfrak{s}$ is one of the clusters $\mathfrak{s}_1, \dots, \mathfrak{s}_h \subset \mathcal{B}$ which are defined by \Cref{prop connected components Sigma}.

\end{prop}

\begin{proof}
Let $\mathfrak{s}$ be a cluster satisfying the hypotheses of the proposition, and consider the corresponding point $\eta_{\mathfrak{s}} \in \Hyp$.  Let $\mathfrak{c}_1, \dots, \mathfrak{c}_s$ be the clusters defined by \Cref{prop removing a point from convex hull}, so that we have $\mathfrak{s} = \mathfrak{c}_1 \sqcup \dots \sqcup \mathfrak{c}_s$ for some $s \geq 2$.  Then for some $i \in \{1, \dots, s\}$, we have $p \nmid \#\mathfrak{c}_i$, which implies $\eta_{\mathfrak{s}} \in \Sigma_{\mathcal{B}}^*$.  However, for any point $\eta = \eta_D > \eta_{\mathfrak{s}}$ close enough that $\eta$ is not a vertex and we have $D \cap \mathcal{B} = D_{\mathfrak{s}} \cap \mathcal{B} = \mathfrak{s}$, from applying \Cref{prop removing a point from convex hull} we get $\eta \notin \Sigma_{\mathcal{B}}^*$.  The point $\eta_{\mathfrak{s}}$ is therefore maximal in some connected component of $\Sigma_{\mathcal{B}}^*$ and thus, by \Cref{prop highest point of Lambda}(a), we get $\mathfrak{s} = \mathfrak{s}_i$ for some $i \in \{1, \dots, h\}$.
\end{proof}

\begin{cor} \label{cor cardinality-p clusters}

Every cluster $\mathfrak{s} \subset \mathcal{B}$ with $p \mid \#\underline{\mathfrak{s}}$ is the disjoint union of some collection among the sets $\mathfrak{r}_1, \dots, \mathfrak{r}_h$ defined by \Cref{prop connected components Sigma}.

\end{cor}

\begin{proof}
It follows immediately from \Cref{prop cardinality-p clusters} that each such cluster $\mathfrak{s}$ is a disjoint union of some of the $\mathfrak{s}_i$'s defined by \Cref{prop connected components Sigma}.  Then the desired statement follows by applying the last assertion of \Cref{prop connected components Sigma} to each of the $\mathfrak{s}_i$'s.
\end{proof}

We are now ready to introduce our alternate definition of the property of being clustered in pairs and show that the two definitions are equivalent; in doing so, we define a stronger property as well.

\begin{dfn}[Berkovich version] \label{dfn clustered in pairs berk}

We say that $\underline{\mathcal{B}}$ is \emph{clustered in $p$-multisets} if the space $\Sigma_{\mathcal{B}}^*$ has $\#\underline{\mathcal{B}} / p$ connected components $\Lambda_0, \dots, \Lambda_h$ (where $h = \#\underline{\mathcal{B}}/p - 1$).

When this property holds, we say that $\underline{\mathcal{B}}$ is \emph{clustered in $r$-separated $p$-multisets} for some real number $r \geq 0$ if, for indices $i \neq j$, we have 
\begin{equation}
B(\Lambda_i, r) \cap B(\Lambda_j, r) = \varnothing.
\end{equation}

If, in addition, each subspace $\Lambda_i \subset \Berk$ consists of a single open path, we say that $\underline{\mathcal{B}}$ is \emph{clustered in $r$-separated pairs}.

\end{dfn}

\begin{prop} \label{prop clustered in pairs}

The multiset $\underline{\mathcal{B}}$ is clustered in $p$-multisets (resp. pairs) in the sense of \Cref{dfn clustered in pairs berk} if and only if it is clustered in $p$-multisets (resp. pairs) in the sense of \Cref{dfn clustered in pairs clusters}.

\end{prop}

\begin{proof}
First assume that the multiset $\underline{\mathcal{B}}$ satisfies \Cref{dfn clustered in pairs berk} (with all of the notation involved).  Letting $h = \#\underline{\mathcal{B}} / p - 1$, by \Cref{prop connected components Sigma}, this gives us a partition $\mathcal{B} = \mathfrak{r}_0 \sqcup \dots \sqcup \mathfrak{r}_h$ with pairwise distinct corresponding clusters $\mathfrak{s}_i$ and with $p \mid \#\underline{\mathfrak{r}}_i, \#\underline{\mathfrak{s}}_i$ for $0 \leq i \leq h$.  The formula for $h$ in terms of the cardinality of $\underline{\mathcal{B}}$ forces $\#\underline{\mathfrak{r}}_i = p$ for $0 \leq i \leq h$, and thus, \Cref{dfn clustered in pairs clusters} is satisfied for $\underline{\mathcal{B}}$.

Now assume conversely that $\underline{\mathcal{B}}$ is clustered in $p$-multisets in the sense of \Cref{dfn clustered in pairs clusters} (with all of the notation involved), and consider the subspaces $\Sigma_{\mathfrak{r}_0}, \dots, \Sigma_{\mathfrak{r}_h} \subset \Sigma_{\mathcal{B}}$ corresponding to the partition $\mathcal{B} = \mathfrak{r}_0 \sqcup \dots \sqcup \mathfrak{r}_h$.  Suppose that there are indices $i \neq j$ such that we have $\Sigma_{\mathfrak{r}_i} \cap \Sigma_{\mathfrak{r}_j} \neq \varnothing$.  Choose a point $\eta \in \Sigma_{\mathfrak{r}_i} \cap \Sigma_{\mathfrak{r}_j}$.  According to \Cref{prop highest point of Lambda}(a), we have $\eta_{\mathfrak{s}_i}, \eta_{\mathfrak{s}_j} \geq \eta$.  We then have $\eta_{\mathfrak{s}_j} \geq \eta_{\mathfrak{s}_i}$ (implying $\eta_{\mathfrak{s}_i} \in [\eta, \eta_{\mathfrak{s}_j}] \subset \Sigma_{\mathfrak{r}_j}$) or $\eta_{\mathfrak{s}_i} \geq \eta_{\mathfrak{s}_j}$ (implying $\eta_{\mathfrak{s}_j} \in [\eta, \eta_{\mathfrak{s}_i}] \subset \Sigma_{\mathfrak{r}_i}$); without loss of generality, let us assume the former.  If we have $\eta_{\mathfrak{s}_i} = \eta_{\mathfrak{s}_j}$, this is the same as saying that $\mathfrak{s}_i = \mathfrak{s}_j$, which violates \Cref{dfn clustered in pairs clusters}.  If instead we have $\eta_{\mathfrak{s}_i} \in (\eta, \eta_{\mathfrak{s}_j})$, then the point $\eta_{\mathfrak{s}_i}$ is a non-maximal point of $\Sigma_{\mathfrak{r}_j}$.  By \Cref{prop highest point of Lambda}(b), we then have points $z, w \in \mathfrak{r}_j$ such that $z \in \mathfrak{s}_i$ and $w \notin \mathfrak{s}_i$.  This contradicts the final statement of \Cref{prop connected components Sigma}.  We have thus shown that the spaces $\Sigma_{\mathfrak{r}_i}$ are mutually disjoint; since each is clearly a closed subspace of $\Hyp$, we get that the spaces $\Sigma_{\mathfrak{r}_i}$ are the connected components of their disjoint union.

We now set out to prove the equality $\Sigma_{\mathcal{B}}^* = \Sigma_{\mathfrak{r}_0} \sqcup \dots \sqcup \Sigma_{\mathfrak{r}_h}$, which will imply that \Cref{dfn clustered in pairs berk} is satisfied for $\underline{\mathcal{B}}$.  Choose any point $\eta = \eta_D \in \Sigma_{\mathcal{B}}$, and let $\mathfrak{c}_1, \dots, \mathfrak{c}_s, \mathfrak{c}_\infty$ be the corresponding clusters as in \Cref{prop removing a point from convex hull}, with $\mathfrak{s} := D \cap \mathcal{B} = \mathfrak{c}_1 \sqcup \dots \sqcup \mathfrak{c}_s$.  Suppose first that we have $\eta \in \Sigma_{\mathcal{B}}^*$.  We have $p \nmid \#\underline{\mathfrak{c}}_i$ for some $i$.  We see from the cardinalities of the multisets $\underline{\mathfrak{r}}_j$ that there is some index $j \in \{0, \dots, h\}$ and points $z, w \in \mathfrak{r}_j$ such that we have $z \in \mathfrak{c}_i$ and $w \notin \mathfrak{c}_i$.  Meanwhile, as the disjoint union $\mathfrak{c}_1 \sqcup \dots \sqcup \mathfrak{c}_s$ is the parent of the cluster $\mathfrak{c}_i$, each point $\eta'$ of the open path $(\eta_{\mathfrak{c}_i}, \eta)$ corresponds to a disc $D'$ satisfying $D' \cap \mathcal{B} = \mathfrak{c}_i$; we then have $\eta' \in (\eta_z, \eta_w) \subseteq \Sigma_{\mathfrak{r}_j}$.  We therefore have $(\eta', \eta) \subset \Sigma_{\mathfrak{r}_j}$, and since $\Sigma_{\mathfrak{r}_j}$ is closed, this implies $\eta \in \Sigma_{\mathfrak{r}_j}$.  Now suppose conversely that we have $\eta \in \Sigma_{\mathfrak{r}_j}$ for some $j$.  By \Cref{prop highest point of Lambda}(a), the maximal point of $\Sigma_{\mathfrak{r}_j}$ is $\eta_{\mathfrak{s}_j}$; since this is a limit point of $\Sigma_{\mathfrak{r}_j}$ and $\Sigma_{\mathcal{B}}^* \subset \Sigma_{\mathcal{B}}$ is closed, in order to prove that $\eta \in \Sigma_{\mathcal{B}}^*$, it suffices to assume that $\eta$ is not the maximal point of $\Sigma_{\mathfrak{r}_j}$.  By \Cref{prop highest point of Lambda}(b), this implies $1 \leq \#(\underline{\mathfrak{s}} \cap \underline{\mathfrak{r}}_j) \leq p - 1$.  Meanwhile, the disjointness of the spaces $\Sigma_{\mathfrak{r}_i}$ implies that we have $\eta \notin \Sigma_{\mathfrak{r}_i}$ for $i \neq j$; by \Cref{prop highest point of Lambda}, this means that either $\mathfrak{s} \supset \mathfrak{r}_i$ or $\mathfrak{s} \cap \mathfrak{r}_i = \varnothing$ for $i \neq j$.  Putting these together and keeping in mind that the cardinality of each multiset $\underline{\mathfrak{r}}_i$ equals $p$, we have $p \nmid \#\underline{\mathfrak{s}} = \sum_i \#\underline{\mathfrak{c}}_i$, implying that $\eta \in \Sigma_{\mathcal{B}}^*$, as desired.

We have just shown that the two definitions of being clustered in $p$-multisets are equivalent.  To show the equivalence of the two definitions of being clustered in pairs, we only need to note that for each index $i$, the subspace $\Lambda_i = \bigcup_{z, w \in \mathfrak{r}_i, \, z \neq w} (\eta_z, \eta_w) \subset \Berk$ consists of a single path if and only if there is only one path in the union given above, which happens if and only if there is only one pair of distinct elements of $\mathfrak{r}_i$.
\end{proof}

\begin{prop} \label{prop properties (i) and (ii) of main theorem}

Properties (i) and (ii) of \Cref{thm main} are equivalent to the property of being clustered in $\pfrac$-separated pairs, which in the tame case (where $\pfrac = 0$) is the same as being clustered in pairs.

\end{prop}

\begin{proof}
Writing $\mathfrak{r}_0 = \{z_0, w_0 := \infty\}, \mathfrak{r}_1 = \{z_1, w_1\}, \dots, \mathfrak{r}_h = \{z_h, w_h\}$, letting $m_i$ be the multiplicity of the root $z_i$ for $0 \leq i \leq h$, and noting that each $\Sigma_{\mathfrak{r}_i}$ is simply the path $[\eta_{z_i}, \eta_{w_i}]$, we see that being clustered in pairs implies property (i) of \Cref{thm main} and also implies mutual disjointness of the subspaces $\Lambda_i = \Sigma_{\mathfrak{r}_i}$; the $\pfrac$-separatedness property then implies property (ii).

Conversely, property (i) furnishes the partition $\underline{\mathcal{B}} = \underline{\mathfrak{r}}_0 \sqcup \dots \sqcup \underline{\mathfrak{r}}_h$, with each multiset $\underline{\mathfrak{r}}_i$ consisting of the element $z_i$ appearing with multiplicity $m_i$ and the element $w_i$ appearing with multiplicity $p - m_i$.  We thus have $\#\underline{\mathfrak{r}}_i = p$ for $0 \leq i \leq h$.  Moreover, property (ii) implies in particular that the subspaces $\Lambda_i = \Sigma_{\mathfrak{r}_i} \subset \Sigma_{\mathcal{B}}^*$ are mutually disjoint.  As the cardinalities of the multisets $\#\underline{\mathfrak{r}}_i$ are minimal with respect to the condition of being divisible by $p$ as in \Cref{prop connected components Sigma}, the partition $\underline{\mathcal{B}} = \underline{\mathfrak{r}}_0 \sqcup \dots \sqcup \underline{\mathfrak{r}}_h$ must be the one guaranteed by that proposition.  Now property (ii) directly implies that $\underline{\mathcal{B}}$ is clustered in $\pfrac$-separated pairs.
\end{proof}

We finish this section by presenting some properties of the space $\Sigma_{\mathcal{B}}^0$ which will be used in \S\ref{sec proof of main}.

\begin{prop} \label{prop boundary points}

For any point $\eta = \eta_D \in \Sigma_{\mathcal{B}}$, write $\mathfrak{s} = D \cap \mathcal{B}$ and $\mathfrak{s} = \mathfrak{c}_1 \sqcup \dots \sqcup \mathfrak{c}_s$ for the partition given by \Cref{prop removing a point from convex hull}.  Let $U \subseteq \Sigma_{\mathcal{B}}^0$ be the connected component in which $\eta$ lies.  The points $\eta$ of the subspace $\Sigma_{\mathcal{B}}^0$ are characterized as those for which we have $p \mid \#\underline{\mathfrak{s}}$ or $p \mid \#\underline{\mathfrak{c}}_i$ for some $i \in \{1, \dots, r\}$.  Moreover, the boundary points of $\Sigma_{\mathcal{B}}^0$ may be characterized as follows.

\begin{enumerate}[(a)]

\item The point $\eta$ is a boundary (\textit{i.e.} non-interior) point of the subspace $\Sigma_{\mathcal{B}}^0 \subset \Sigma_{\mathcal{B}}$ if and only if we have $(\eta, \eta_z) \cap U = \varnothing$ for some $z \in \mathcal{B}$.

\item We have $(\eta, \eta_z) \cap U = \varnothing$ for some $z \in \mathcal{B} \smallsetminus D$ if and only if we have $p \nmid \#\underline{\mathfrak{s}}$.  In this case, the point $\eta$ is the (unique) maximal point of $U$.

\item We have $(\eta, \eta_z) \cap U = \varnothing$ for some $z \in D$ if and only if some cluster $\mathfrak{c}$ whose parent is $\mathfrak{s}$ (namely the one containing $z$) satisfies $p \nmid \#\underline{\mathfrak{c}}$.

\item We have $(\eta, \eta_z) \cap U = \varnothing$ for some $z \in D$ but $(\eta, \eta_w) \cap U \neq \varnothing$ for all $w \in \mathcal{B} \smallsetminus D$ if and only if $\eta$ is the (unique) maximal point of a connected component $\Sigma_{\mathfrak{r}_i} \subset \Sigma_{\mathcal{B}}^*$ for some $i \in \{1, \dots, h\}$, in the notation of \Cref{prop connected components Sigma}.

\end{enumerate}

\end{prop}

\begin{proof}
Choose any point $\eta = \eta_D \in \Sigma_{\mathcal{B}}$, and assume that we have $p \mid \#\underline{\mathfrak{s}}$ (resp. $p \mid \#\underline{\mathfrak{c}}_i$ for some $i \in \{1, \dots, h\}$).  Let $z$ be any element of $\mathfrak{s}$ (resp. $\mathfrak{c}_i$).  Now it is clear that for any disc $D' \supseteq D$ (resp. $D' \subsetneq D$ with $z \in D'$) such that $\eta_{D'}$ is close enough to $\eta = \eta_D$, we have $D' \cap \mathcal{B} = \mathfrak{s}$ (resp. $D' \cap \mathcal{B} = \mathfrak{c}_i$, which we see by noting that $\mathfrak{s}$ is the parent of $\mathfrak{c}_i$) while $D'$ is not the smallest disc containing $\mathfrak{s}$ (resp. $\mathfrak{c}_i$), so that $\eta_{D'} \neq \eta_{\mathfrak{s}}$ (resp. $\eta_{D'} \neq \eta_{\mathfrak{c}_i}$).  Then it follows from \Cref{prop removing a point from convex hull} that $(\eta = \eta_D, \eta_{D'}) \subset \Sigma_{\mathcal{B}} \smallsetminus \Sigma_{\mathcal{B}}^*$, which means that $\eta$ lies in the closure of this complement and thus in $\Sigma_{\mathcal{B}}^0$.  Now assume conversely that we have $\eta \in \Sigma_{\mathcal{B}}^0$, and let $(\eta, \eta')$ be the open path in $\Sigma_{\mathcal{B}} \smallsetminus \Sigma_{\mathcal{B}}^*$ whose existence is implied by the definition of $\Sigma_{\mathcal{B}}^0$.  If $\eta'$ is chosen close enough to $\eta$, then the cluster $\mathfrak{c} := D'' \cap \mathcal{B}$ for a disc $D''$ such that $\eta_{D''} \in (\eta, \eta') \subset \Sigma_{\mathcal{B}} \smallsetminus \Sigma_{\mathcal{B}}^*$ (so that we have $p \mid \#\underline{\mathfrak{c}}$) does not depend on the choice of $\eta_{D''}$ in this open path.  Moreover, the cluster $\mathfrak{s} = D \cap \mathcal{B}$ either coincides with $\mathfrak{c}$ or is the parent of $\mathfrak{c}$, in which case we have $\mathfrak{c} = \mathfrak{c}_i$ for some $i \in \{1, \dots, h\}$.  Therefore, the cardinality of one of the multisets $\underline{\mathfrak{s}}, \underline{\mathfrak{c}}_1, \dots, \underline{\mathfrak{c}}_s$ is divisible by $p$.  This proves the first statement of the proposition.

Now suppose that $\eta$ is a boundary point of $\Sigma_{\mathcal{B}}^0 \subset \Sigma_{\mathcal{B}}$.  Then there is an open path of the form $(\eta, \eta')$ contained in $\Sigma_{\mathcal{B}} \smallsetminus U$; if $\eta' = \eta_{D'}$ is chosen close enough to $\eta$, then we have $\eta' > \eta$ or $\eta > \eta'$.  In the former case, we have $(\eta, \eta') \subset (\eta, \eta_\infty)$ and, by unique path connectedness and the fact that $U$ is connected, we have $(\eta, \eta_\infty) \cap U = \varnothing$.  In the latter case, by \Cref{prop removing a point from convex hull}, there exists $z \in D' \cap \mathcal{B}$ and so we have $(\eta, \eta') \subset (\eta, \eta_z)$, and similarly we get $(\eta, \eta_z) \cap U = \varnothing$.  The converse is immediate, and part (a) is proved.

Suppose that for some $z \in \mathcal{B} \smallsetminus D$ (which implies $\eta \vee \eta_z > \eta$), we have $(\eta, \eta_z) \cap U = \varnothing$.  Then for any $\eta' \in (\eta, \eta \vee \eta_z)$, we have $\eta' > \eta$ and, if $\eta'$ is chosen close enough to $\eta$, we have $(\eta, \eta') \cap \Sigma_{\mathcal{B}}^0 = \varnothing$ and thus $(\eta, \eta') \subset \Sigma_{\mathcal{B}}^*$; we may further guarantee that for each $\eta_{D''} \in (\eta, \eta')$, we have $D'' \cap \mathcal{B} = \mathfrak{s}$ and $\eta_{D''} \neq \eta_{\mathfrak{s}}$ (see the arguments of the first paragraph).  Then by \Cref{prop removing a point from convex hull} and the definition of $\Sigma_{\mathcal{B}}^*$, we have $p \nmid \#\underline{\mathfrak{s}}$.  Conversely, if $p \nmid \#\underline{\mathfrak{s}}$, then all points $\eta_{D''} > \eta$ close enough to $\eta$ that we have $D'' \cap \mathcal{B} = \mathfrak{s}$ lie in $\Sigma_{\mathcal{B}}^*$ and thus not in $U$.  We thus have $(\eta, \eta') \cap U = \varnothing$ for some $\eta' > \eta$, and this again implies $(\eta, \eta_\infty) \cap U = \varnothing$.  This proves the ``if and only if" statement of part (b) as well as the fact that $\eta$ is maximal in $U$.

Part (c) is proved analogously to the ``if and only if" statement of part (b) and again uses arguments from the first paragraph.

To prove part (d), we observe that the first hypothesis of part (d) is equivalent to the existence of a cluster $\mathfrak{c} \ni z$ whose parent is $\mathfrak{s}$ and satisfies $p \nmid \#\underline{\mathfrak{c}}$ thanks to part (c), while the second hypothesis of part (d) is equivalent to $p \mid \#\underline{\mathfrak{s}}$ thanks to part (b).  The above divisibility properties imply that $s \geq 2$ and so $\eta = \eta_{\mathfrak{s}}$ by \Cref{prop removing a point from convex hull}.  Now since there exists a cluster (namely $\mathfrak{c}$) whose parent is $\mathfrak{s}$ and whose cardinality as a multiset is not divisible by $p$, the cluster $\mathfrak{s}$ is not the disjoint union of $\geq 2$ sub-clusters whose cardinalities as multisets are divisible by $p$, so by \Cref{prop cardinality-p clusters} it is the smallest cluster containing a subset $\mathfrak{r} \subset \mathcal{B} \smallsetminus \{\infty\}$ for which $\Sigma_{\mathfrak{r}}$ is a connected component of $\Sigma_{\mathcal{B}}^*$.  Then by \Cref{prop highest point of Lambda}(a), the point $\eta$ is the maximal point of $\Sigma_{\mathfrak{r}}$; this proves the first direction of part (d).  The converses of each of the last several implications hold as well: by \Cref{prop highest point of Lambda}(a), the maximal point of any component $\Sigma_{\mathfrak{r}} \not\ni \eta_\infty$ of $\Sigma_{\mathcal{B}}^*$ is $\eta_{\mathfrak{s}}$, where $\mathfrak{s}$ is the smallest cluster containing $\mathfrak{r}$; and \Cref{prop cardinality-p clusters} implies that we have $p \mid \#\underline{\mathfrak{s}}$ but that some cluster $\mathfrak{c}$ whose parent is $\mathfrak{s}$ satisfies $p \nmid \#\underline{\mathfrak{c}}$.  Therefore, the other direction of part (d) also holds.
\end{proof}

\section{Preliminaries on semistable models of Galois covers} \label{sec preliminaries geometric}

In \S\ref{sec preliminaries geometric models}, we recall some well known facts about models of curves (in particular, regular and semistable models), while in \S\ref{sec preliminaries geometric components} we use them to characterize the components of the special fiber of the minimal regular model by considering certain invariants of each component.  Then in \S\ref{sec preliminaries geometric Galois}, we narrow our focus to the case where $C$ is a $p$-cyclic Galois cover of the projective line and study the toric rank of $C$ through this framework.

\subsection{Minimal models of curves} \label{sec preliminaries geometric models}

A \emph{model} of $C / K$ (over a ring of integers $\mathcal{O}_{K'}$ for some finite extension $K' / K$) is a normal, flat, projective $\mathcal{O}_{K'}$-scheme $\mathcal{C}$ whose generic fiber is identified with $C$.  (We say that such a model is \emph{defined over $K'$}.)  The models of $C$ form a preordered set with the order relation being given by \emph{dominance}: given two models $\mathcal{C}$ and $\mathcal{C}'$ of $C$, we will write $\mathcal{C}' \geq \mathcal{C}$ (or $\mathcal{C} \leq \mathcal{C}'$) to mean that $\mathcal{C}'$ dominates $\mathcal{C}$, \textit{i.e.} that the identity map $C\to C$ extends to a birational morphism $\mathcal{C}'\to \mathcal{C}$.  For a given model $\mathcal{C}$ of $C$, the \emph{minimal desingularization} of $\mathcal{C}$ is the regular model dominating $\mathcal{C}$ which is minimal (with respect to dominance) for these properties.  The \emph{minimal regular model} of $C$ is the regular model which is minimal among models of $C$.

The special fiber $\mathcal{C}_s$ of a model $\mathcal{C}$ of $C$ is connected and consists of a number of irreducible components $V_1, \ldots, V_n$; these components are projective, possibly singular curves over the residue field $k$, each one appearing in $\mathcal{C}_s$ with a certain multiplicity (which is defined as the length of the local ring of $\mathcal{C}_s$ at the generic point of the component). We denote by $\Irred(\mathcal{C}_s)=\{ V_1,\ldots, V_n\}$ the set of such components. It follows from the properness and flatness of the morphism $\mathcal{C} \to \Spec(\mathcal{O}_{K'})$ that the genus of the smooth $K$-curve $C$ coincides with the arithmetic genus of the $k$-curve $\mathcal{C}_s$.

A model $\mathcal{C}$ of $C$ is said to be \emph{semistable} if its special fiber is reduced and its singularities (if there are any) are all nodes (\textit{i.e.} ordinary double points).  More generally, we say that a model $\mathcal{C}$ of $C$ is \emph{semistable at a point} $P\in \mathcal{C}_s$ if $\mathcal{C}_s$ is reduced at $P$ and if $P$ is either a smooth point or a node of $\mathcal{C}_s$; in the latter case, the completed local ring at $P$ has the form $\mathcal{O}_{K'}[[t_1,t_2]]/(t_1 t_2 - a)$ for some $a \in \pi \mathcal{O}_{K'}$, where $\pi$ is a uniformizer of $K'$.  The integer $v(a)/v(\pi) \geq 1$ is known as the \emph{thickness} of the node. A semistable model is regular precisely when all of its nodes have thickness equal to $1$.  We remark that while semistability of a model $\mathcal{C}$ is preserved over an extension of $\mathcal{O}_{K'}$, the property of regularity is generally not preserved, as changing the base to an extension $\mathcal{O}_{K''} / \mathcal{O}_{K'}$ results in the thickness of each node of $\mathcal{C}$ being multiplied by the ramification index $e_{K''/K'}$ in $\mathcal{C} \otimes \Spec(\mathcal{O}_{K''})$.

To describe the combinatorics of a semistable model $\mathcal{C}$ of a curve $C$, one can form the \emph{dual graph} $\Gamma(\mathcal{C}_s)$ of its special fiber, whose set of vertices is $\Irred(\mathcal{C}_s)$ and whose edges correspond to the nodes connecting them.

Given a semistable model $\mathcal{C}$ of $C$, the (generalized) Jacobian $\Pic^0(\mathcal{C}_s)$ of the (possibly singular) $k$-curve $\mathcal{C}_s$ is an extension of an abelian variety $A$ by a torus $T$ (see \cite[Example 9.2.8 and Corollary 9.7.2.]{bosch2012neron}). The ranks of $A$ and $T$ are known respectively as the \emph{abelian rank} and the \emph{toric rank}, and we denote them respectively by $a(\mathcal{C}_s)$ and $t(\mathcal{C}_s)$.

\begin{rmk} \label{rmk abelian and toric}

The following facts about the ranks defined above are well known.

\begin{enumerate}[(a)]
\item The abelian and toric ranks do not depend on the choice of semistable model of a curve $C$.
\item For all models $\mathcal{C}$, the abelian rank $a(\mathcal{C}_s)$ equals the sum $\sum_{V\in \Irred(\mathcal{C}_s)} a(V)$, where $a(V)$ is the genus of the normalization of $V$.
\item If $\mathcal{C}$ is a semistable model, the toric rank $t(\mathcal{C}_s)$ can be computed as $\rank_\zz H^1(\Gamma(\mathcal{C}_s), \zz)$ (see \cite[Example 9.2.8]{bosch2012neron} for a proof); this is the number of edges minus the number of vertices plus $1$.  In other words, we have $t(\mathcal{C}_s)=N_{\mathrm{nodes}}(\mathcal{C}_s) - N_{\mathrm{irred}}(\mathcal{C}_s) + 1$, where $N_{\mathrm{nodes}}(\mathcal{C}_s)$ denotes the number of nodes and $N_{\mathrm{irred}}(\mathcal{C}_s) := \#\Irred(\mathcal{C}_s)$ is the number of irreducible components.
\end{enumerate}

\end{rmk}

\begin{rmk} \label{rmk alternate definition of split degenerate}

The ranks $a(\mathcal{C}_s)$ and $t(\mathcal{C}_s)$ are non-negative integers which add up to the genus $g$ of $C$ (see, for example, the results of \cite[Section 7.5]{liu2002algebraic}).  Meanwhile, by \Cref{rmk abelian and toric}(a), the condition on genus of components for a model $\mathcal{C}$ to be (split) degenerate is equivalent to $a(\mathcal{C}_s) = 0$.  It follows that an alternate definition for split degeneracy is the one given in \Cref{dfn reduction type} with the genus condition replaced by requiring $t(\mathcal{C}_s) = g$, or in other words, that the special fiber have the ``greatest possible" toric rank.

\end{rmk}

We now define the notion of (-1)-lines in the context of semistable models.

\begin{dfn} \label{dfn (-1)-lines}

If  $\mathcal{C}$ is a model, $V \in \Irred(\mathcal{C}_s)$, and $\mathcal{C}$ is semistable at the points of $V$, then $V$ is said to be a \emph{(-1)-line} if it is a line (\textit{i.e.} $V \cong \proj^1_k$) and there is exactly $1$ node of $\mathcal{C}_s$ lying on it.

\end{dfn}

\begin{rmk}

It is possible to show that, if $\mathcal{C}$ is regular model that is semistable at the points of a component $V \in \Irred(\mathcal{C}_s)$, then the definition above is consistent with the definition involving self-intersection numbers commonly given in the more general context of regular models: this follows, for example, from the formula for self-intersection numbers given in \cite[Proposition 9.1.21(b)]{liu2002algebraic}.

\end{rmk}

\begin{rmk} \label{rmk minimal means no (-1)-lines}

The minimal desingularization of $\mathcal{C}$ (resp. regular minimal model of $C$) can be characterized as the unique regular model dominating $\mathcal{C}$ (resp. regular model of $C$) whose special fiber does not contain (-1)-lines which are contracted to obtain $\mathcal{C}$ (resp. does not contain (-1)-lines), and the minimal regular model over $\mathcal{O}_K$ of a curve $C$ that has semistable reduction over $K$ is itself a semistable model: see for instance the results of \cite[\S9.3]{liu2002algebraic}.

\end{rmk}

\subsection{Finding components of the minimal regular model} \label{sec preliminaries geometric components}

Given a model $\mathcal{C}$ of $C$ and a component $V \in \Irred(\mathcal{C}_s)$, we consider several invariants attached to $V$, listed as follows:
\begin{itemize}
\item $m(V)$ denotes the multiplicity of $V$ in $\mathcal{C}_s$;
\item $a(V)$ denotes the \emph{abelian rank} of $V$, \textit{i.e.} the genus of the normalization of $V$;
\item $w(V)$ is defined only when $m(V)=1$, and it denotes the number of singular points of $\mathcal{C}_s$ that belong to $V$, each one counted as many times as the number of branches of $V$ at that point; in other words, writing $\widetilde{V}$ for the normalization of $V$, then $w(V)$ is the number of points of $\widetilde{V}$ that lie over the intersection of $V$ with the set of singular points of $\mathcal{C}_s$.
\end{itemize}

We now show that, under appropriate assumptions, the integers $m$, $a$, and $w$ are left invariant when the model is changed.

\begin{lemma} \label{lemma inv m a w}

Let $\mathcal{C}'$ be another model of $C$ which dominates $\mathcal{C}$, and let $V'$ denote the strict transform of $V$ in $(\mathcal{C}')_s$.  Then we have $m(V') = m(V)$ and $a(V') = a(V)$. Moreover, if $\mathcal{C}'$ is dominated by the minimal desingularization of $\mathcal{C}$, we also have $w(V') = w(V)$.

\end{lemma}

\begin{proof}
For $m$ and $a$, the lemma immediately follows from the consideration that $\mathcal{C}'_s \to \mathcal{C}_s$ is an isomorphism away from a finite set of points of $\mathcal{C}_s$. We will now prove the result for $w$.
        
Let $Q$ be a point of $V'$ which lies over some $P\in V$. We claim that $\mathcal{C}'_s$ is smooth (resp. singular) at $Q$ if and only if $\mathcal{C}_s$ is smooth (resp. singular) at $P$. This is obvious whenever $\mathcal{C}' \to \mathcal{C}$ is an isomorphism above $P$. If $\mathcal{C}' \to \mathcal{C}$ is not an isomorphism above $P$, the claim follows from the two following observations.  Firstly, since we are assuming that $\mathcal{C}'$ is dominated by the minimal desingularization of $\mathcal{C}$, it must be the case that $\mathcal{C}$ is not regular at $P$ and consequently that $\mathcal{C}_s$ is singular at $P$.  Secondly, the fiber of $\mathcal{C}'\to \mathcal{C}$ above $P$ is pure of dimension 1, and it consists of those components $E_i$ of $\mathcal{C}'_s$ that contract to $P$; the point $Q$ will thus belong not only to $V'$, but also to one of the $E_i$'s, so that $\mathcal{C}'_s$ will certainly be singular at $Q$. This completes the proof of the claim.
        
Now, since $\mathcal{C}' \to \mathcal{C}$ restricts to a birational morphism $V' \to V$ of $k$-curves, the set of branches of $V$ at a point $P \in \mathcal{C}_s$ equals the set of branches of $V'$ at the points of $\mathcal{C}'_s$ lying above $P$. If we combine this consideration with the claim we have just proved, we see that the morphism $V' \to V$ induces a bijection between the set of the branches of $V$ at the singular points of $\mathcal{C}_s$ and the set of the branches of $V'$ at the singular points of $\mathcal{C}'_s$. The equality $w(V') = w(V)$ follows.
\end{proof}

We now describe how the invariants we have defined allow us to detect (-1)-lines.

\begin{lemma} \label{lemma detecting (-1)-lines}

Let $\mathcal{C}$ be any model of $C$, and let $V$ be an irreducible component of $\mathcal{C}_s$.

\begin{enumerate}[(a)]
\item If $\mathcal{C}$ is regular and $V$ is a (-1)-line of multiplicity 1, then $a(V)=0$ and $w(V)=1$.
\item If $\mathcal{C}$ is semistable at the points of $V$, then $V$ is a (-1)-line if and only if $a(V)=0$ and $w(V)=1$.
\end{enumerate}

\end{lemma}

\begin{proof}
If $V$ is a component of multiplicity $1$ in the special fiber $\mathcal{C}_s$ of a regular model $\mathcal{C}$, then it follows from the intersection theory of regular models (see \cite[Chapter 9]{liu2002algebraic}) that its self-intersection number is equal to minus the number of points at which $V$ intersects the other components of $\mathcal{C}_s$, each counted with a certain (positive) multiplicity. Once this has been observed, part (a) follows immediately from the usual definition involving self-intersection numbers for (-1)-lines in the context of regular models.
        
Suppose now that $V$ is a component of the special fiber $\mathcal{C}_s$ of a model $\mathcal{C}$ that is semistable at the points of $V$ (which, in particular, implies that $m(V)=1$). From the definition of the invariant $w$ and the structure of semistable models, it is clear that $w(V)$ equals the sum $2w_{\text{self}}(V) + w_{\text{other}}(V)$, where $w_{\text{self}}(V)$ is the number of self-intersections of $V$, while $w_{\text{other}}(V)$ is the number of intersections of $V$ with other components of $\mathcal{C}_s$; moreover, we have $w_{\text{self}}(V)=0$ if and only if $V$ is smooth.  But the line $\mathbb{P}^1_k$ is the unique smooth $k$-curve with abelian rank 0, so the component $V$ is a line if and only if $a(V)=0$ and $w_{\text{self}}(V)=0$; according to \Cref{dfn (-1)-lines}, the component $V$ is thus a (-1)-line if and only if $a(V)=0$, $w_{\text{self}}(V)=0$, and $w_{\text{other}}(V) = 1$.  This immediately implies part (b).
\end{proof}

The above results and discussion have led us to a criterion for a model $\mathcal{C}$ to be part of the minimal regular model of $C$ in the case that $C$ has semistable reduction over $K' / K$.

\begin{prop} \label{prop part of mini}

Assume that the curve $C / K'$ has positive genus and semistable reduction.  Then a model $\mathcal{C}$ of $C$ (defined over $K'$) is dominated by the minimal regular model of $C$ if and only if for each component $V$ of $\mathcal{C}_s$, we have
\begin{enumerate}[(i)]
\item $m(V)=1$, and 
\item $a(V) \geq 1$ or $w(V) \geq 2$.
\end{enumerate}

\end{prop}

\begin{proof}
Let $\mathcal{C}'$ denote the minimal regular model of $C$ over $K'$.  First assume that we have $\mathcal{C} \leq \mathcal{C}'$.  Then it follows from \Cref{lemma inv m a w} that the invariants $a$, $m$, and $w$ of a component $V$ of the special fiber $\mathcal{C}_s$ must be equal to those of its strict transform $V'$ in $\SF{\mathcal{C}'}$.  Since $\mathcal{C}'$ is semistable, its special fiber is reduced; thus, we get $m(V) = m(V') = 1$.  Let us now assume that $a(V) = a(V') = 0$.  If it were the case that $w(V) = 0$, then $\mathcal{C}_s = V$ would be a line; since the arithmetic genus of $\mathcal{C}_s$ equals the genus of $C$, this contradicts the condition that the genus is positive.  If we had $w(V) = w(V') = 1$, then, thanks to \Cref{lemma detecting (-1)-lines}(b), the component $V'$ would be a (-1)-line, which is impossible by \Cref{rmk minimal means no (-1)-lines}.  Thus, the quantity $w(V)$ is necessarily $\geq 2$.

Now assume that for each component $V$ of $\mathcal{C}_s$, the conditions (i) and (ii) given in the statement hold.  Let $\tilde{\mathcal{C}}$ be the minimal desingularization of $\mathcal{C}$.  Assume by way of contradiction that $\tilde{\mathcal{C}}_s$ contains a (-1)-line.  Since the desingularization $\tilde{\mathcal{C}}$ is minimal, such a (-1)-line must necessarily be the strict transform $V''$ of some component $V \in \Irred(\mathcal{C}_s)$.  By \Cref{lemma inv m a w}, the quantities $a(V'')$, $m(V'')$ and $w(V'')$ are equal to $a(V)$, $m(V)$ and $w(V)$ respectively.  Thus, from the condition $m(V)=1$, we deduce $m(V'')=1$.  Moreover, since $V''$ is a (-1)-line of multiplicity $1$, \Cref{lemma detecting (-1)-lines}(a) ensures that $a(V'') = 0$ and $w(V'') = 1$, and thus we have $a(V)=0$ and $w(V)=1$.  But this contradicts our hypothesis, so we conclude that $\tilde{\mathcal{C}}_s$ cannot contain a (-1)-line.  It follows from \Cref{rmk minimal means no (-1)-lines} that $\tilde{\mathcal{C}} = \mathcal{C}'$, and we get $\mathcal{C} \leq \mathcal{C}'$ as desired.
\end{proof}

\subsection{Galois covers and toric rank} \label{sec preliminaries geometric Galois}

If $C \to X$ is a Galois cover of curves over $K$ with Galois group $G$ such that the Galois automorphisms of $C$ are defined over $K$, then any model $\mathcal{X} / \mathcal{O}_K$ of $X$ gives rise to a model $\mathcal{C} / \mathcal{O}_K$ of $C$ on which $G$ acts by $\mathcal{O}_K$-scheme automorphisms, via letting $\mathcal{C}$ be the normalization of $\mathcal{X}$ in the function field $K(C)$ (see \cite[Definition 4.1.24]{liu2002algebraic}).  Conversely, given a model $\mathcal{C} / \mathcal{O}_K$ of $C$ on which $G$ acts, the quotient $\mathcal{C} / G$ with respect to that action is a model of $X$ whose normalization in $K(C)$ can be identified with $\mathcal{C}$.  Moreover, it is clear that $G$ acts on the set $\Irred(\mathcal{C}_s)$ of irreducible components of $\mathcal{C}_s$ and that the resulting quotient $\Irred(\mathcal{C}_s) / G$ can be identified with the set $\Irred(\mathcal{X}_s)$ of irreducible components of $X$.  Relatedly a dominant morphism $\mathcal{X}_1 \to \mathcal{X}_2$ between two models of $\mathcal{X}$ which contracts the components in a certain subset $\mathfrak{C} \subsetneq \Irred((\mathcal{X}_1)_s)$ lifts to a dominant morphism $\mathcal{C}_1 \to \mathcal{C}_2$ between their normalizations which contracts the components in the inverse image of $\mathfrak{C}$ in $\Irred((\mathcal{C}_1)_s)$.

It follows from \cite[Proposition 10.1.16]{liu2002algebraic} that the minimal regular model of a Galois cover $C \to X$ is acted on by the Galois group $G$: more precisely, the action of each element of $G$ on the generic fiber $C$ extends uniquely to an automorphism of the minimal regular model.

\begin{prop} \label{prop persistent nodes}

Let $C / K$ be a $G$-Galois cover of $X$, and let $K' / K$ be a finite extension over which $C$ achieves semistable reduction.  There is an extension $K'' / K'$ with $[K'' : K'] \leq 2$ such that, letting $\Cmini$ be the minimal regular model of $C$ over $\mathcal{O}_{K''}$ and defining $\Xmini$ to be the quotient $\Cmini / G$, a $k$-point of $\SF{\Cmini}$ is a node if and only if its image under the quotient map $\Cmini \to \Xmini$ is a node.  (In particular, the model $\Xmini$ is semistable.)

Moreover, if there exists a semistable model $\mathcal{C}$ of $C$ over $\mathcal{O}_{K'}$ such that the image of each node of $\SF{\mathcal{C}}$ in the quotient $\mathcal{C} / G$ is a node, then we may take $K'' = K'$.

\end{prop}

\begin{proof}
The fact that the image of a smooth point of a semistable model (over any extension of $\mathcal{O}_K$) is a smooth point in the quotient by $G$ is given by \cite[Proposition 10.3.48(a)]{liu2002algebraic}.

Let $\mathcal{C}'$ denote the minimal regular model of $C$ over $\mathcal{O}_{K'}$.  As discussed in \Cref{rmk minimal means no (-1)-lines}, the model $\mathcal{C}'$ is itself semistable as well as regular and therefore, every node of $(\mathcal{C}')_s$ has thickness $1$.  If each node of $(\mathcal{C}')_s$ lies over a node of $(\mathcal{X}')_s$, then, letting $K'' = K'$, we are done.  Otherwise, let $K'' / K'$ be a quadratic extension, which necessarily has ramification index $2$ because $K$ is assumed to have algebraically closed residue field.  Now the model $\mathcal{C}'' := \mathcal{C}' \otimes_{\mathcal{O}_{K''}} \mathcal{O}_{K'}$ of $C$ is still acted on by the Galois group $G$ of the cover $C \to X$ and has the property that the only singular points of $(\mathcal{C}'')_s$ (and thus the only singular points of the arithmetic surface $\mathcal{C}''$) are nodes of thickness $2$.

Now it is clear that $\Cmini$ can be obtained by desingularizing the arithmetic surface $\mathcal{C}''$ at each of its singular points.  Let $Q \in (\mathcal{C}'')_s$ be one of these singular points, a node of thickness $2$.  The inverse image of $Q$ under the blowing-up $\mathcal{C}''_Q \to \mathcal{C}''$ at $Q$ consists of a line $L \cong \proj_k^1$ which intersects the rest of the blown-up special fiber at $2$ nodes $Q_1, Q_2$ (each of thickness $1$) at the respective strict transforms $B_1, B_2$ of the two branches of $(\mathcal{C}'')_s$ which pass through $Q$.

Let $G_Q \leq G$ be the subgroup which fixes the node $Q$.  Then, as each automorphism in $G_Q$ extends to an automorphism of the desingularization $\mathcal{C}''_Q$, it is clear that $G_Q$ fixes the line $L$ as it is the inverse image of $Q$ under the blow-up.  Now let $G_{Q_1} \leq G_Q$ be the subgroup which fixes the node $Q_1$ in $(\mathcal{C}''_Q)_s$.  Since each element $g \in G_{Q_1}$ fixes both $L$ and $B_1$, it fixes each of the two branches of $(\mathcal{C}''_Q)_s$ passing through $Q_1$.  It then follows from an explicit local study of the quotient map $\mathcal{C}''_Q \to \mathcal{X}''_Q := \mathcal{C}''_Q / G$, which can be found in the proof of \cite[Proposition 10.3.48(b)]{liu2002algebraic}, that the image of $Q_1$ in $(\mathcal{X}''_Q)_s$ is a node.  By an identical argument, the image of $Q_2$ in ($\mathcal{X}''_Q)_s$ is a node.

Thus, desingularizing $\mathcal{C}''$ at each node $Q$ of its special fiber does not add any new nodes to the special fiber except for two nodes $Q_1, Q_2$, each of whose images modulo the action of $G$ is a node.  It follows that the minimal regular model $\Cmini$ of $C$ over $\mathcal{O}_{K''}$ satisfies the desired property that each node of $\SF{\Cmini}$ maps to a node of $\SF{\Xmini}$.

Now let $\mathcal{C} / \mathcal{O}_{K'}$ be a semistable model as in the second statement of the proposition.  The minimal regular model of $C$ over $\mathcal{O}_{K'}$ is obtained by repeatedly blowing up the special fiber of $\mathcal{C}$ at its nodes of thickness $\geq 2$ in order to get the minimal desingularization of $\mathcal{C}$ and then contracting the (-1)-lines.  The nodes created from each instance of blowing up each lie above a node in the quotient by the same argument as was given above, and contracting (-1)-lines only gets rid of nodes; therefore the minimal regular model of $C$ over $\mathcal{O}_{K'}$ satisfies the desired hypothesis about nodes, and we may take $K'' = K'$.
\end{proof}

\begin{cor} \label{cor inverse images are smooth}

With the set-up and definitions of \Cref{prop persistent nodes}, suppose that each component of $\SF{\Xmini}$ is a smooth $k$-curve.  Then the inverse image of any component of $\SF{\Xmini}$ in $\SF{\Cmini}$ is also a smooth (possible disconnected) $k$-curve.

\end{cor}

\begin{proof}
Choose a component $V \in \SF{\Xmini}$ and a point $Q \in \SF{\Cmini}(k)$ lying in the inverse image of $V$.  Let $P \in V(k)$ be the image of $Q$.  If the fiber $\SF{\Xmini}$ is smooth at $P$, then by \Cref{prop persistent nodes}, the fiber $\SF{\Cmini}$ is smooth at $Q$ and so $Q$ is a non-singular point of the inverse image of $V$.  If, on the other hand, the point $P$ is singular, then by \Cref{prop persistent nodes}, the point $Q$ must be as well; in fact, it must be a node as the model $\Cmini$ is semistable.  Since each component of $\SF{\Xmini}$ is smooth, the singular point $P$ must be the intersection of $V$ with some (unique) other component $V'$.  Then for some (distinct) components $W, W'$ of $\SF{\Cmini}$ respectively lying over $V, V'$, the node $Q$ must be the intersection of $W$ and $W'$.  As $Q$ is a node, only $2$ branches of $\SF{\Cmini}$ pass through it, so the irreducible curve $W$ must be smooth at $Q$, and no other component in the inverse image of $V$ passes through $Q$.  It follows that $Q$ is again a non-singular point of the inverse image of $V$.
\end{proof}

It follows from the above results that the minimal regular model $\Cmini$ of a $p$-cyclic Galois cover $C \to X := \proj_K^1$ is therefore determined by a model $\Xmini$ of $X = \proj_K^1$ (over some finite extension $\mathcal{O}_{K'} / \mathcal{O}_K$).  Fortunately, models of the projective line are easy to describe and classify -- see \S\ref{sec normalizations models of P^1} below for details.  The property of models $\mathcal{X}$ of the projective line that are pertinent to this subsection is that their special fibers consist of copies of $\proj_k^1$ meeting at ordinary singularities, each of which can itself be identified with the special fiber of a (smooth) model of $X$ which is dominated by $\mathcal{X}$.  If $\mathcal{X}$ is semistable, then these copies of $\proj_k^1$ are arranged in a tree-like fashion, \textit{i.e.} with the notation of \Cref{rmk abelian and toric}(c), their special fibers satisfy $H^1(\mathcal{X}_s, \zz) = 0$, or equivalently, we have 
\begin{equation} \label{eq toric rank of proj}
N_{\mathrm{nodes}}(\mathcal{X}_s) - N_{\mathrm{irred}}(\mathcal{X}_s) + 1 = 0.
\end{equation}

We will use the following proposition to compute the toric rank of a given $p$-cyclic cover $C \to X$.

\begin{prop} \label{prop toric rank formula}

With the set-up and definitions of \Cref{prop persistent nodes}, let $N_0$ denote the number of nodes of $\SF{\Xmini}$ which are not branch points of the cover $\Cmini \to \Xmini$, and let $N_1$ denote the number of components of $\SF{\Xmini}$ whose inverse images in $\SF{\Cmini}$ consist of multiple components.  Then we have the formula 
\begin{equation} \label{eq toric rank formula}
t(\SF{\Cmini}) = (p - 1)(N_0 - N_1).
\end{equation}

\end{prop}

\begin{proof}
By \Cref{rmk abelian and toric}(c), the toric rank $t(\SF{\Cmini})$ equals $N_{\mathrm{nodes}}(\SF{\Cmini}) - N_{\mathrm{irred}}(\SF{\Cmini}) + 1$.  Now, using (\ref{eq toric rank of proj}) (with $\mathcal{X} = \Xmini$), we may write 
\begin{equation} \label{eq first toric rank formula}
t(\SF{\Cmini}) = [N_{\mathrm{nodes}}(\SF{\Cmini}) - N_{\mathrm{nodes}}(\SF{\Xmini})] - [N_{\mathrm{irred}}(\SF{\Cmini}) - N_{\mathrm{irred}}(\SF{\Xmini})].
\end{equation}

If a node of $\SF{\Xmini}$ is not a branch point of the cover $\Cmini \to \Xmini$, then since this cover is Galois with cyclic order-$p$ Galois group, the inverse image of such a node consists of exactly $p$ points of $\SF{\Cmini}$; otherwise, the inverse image of that node consists of exactly $1$ point of $\SF{\Cmini}$.  By \Cref{prop persistent nodes}, every point in the inverse image of a node of $\SF{\Xmini}$ is a node of $\SF{\Cmini}$, and all nodes of $\SF{\Cmini}$ arise this way.  It follows that we have $N_{\mathrm{nodes}}(\SF{\Cmini}) - N_{\mathrm{nodes}}(\SF{\Xmini}) = (p - 1)N_0$.  Meanwhile, since the cover $\Cmini \to \Xmini$ is Galois with cyclic order-$p$ Galois group, the inverse image of any component of $\SF{\Xmini}$ consists either of a single component of $\SF{\Cmini}$ or of exactly $p$ components of $\SF{\Cmini}$, and therefore, we have $N_{\mathrm{irred}}(\SF{\Cmini}) - N_{\mathrm{irred}}(\SF{\Xmini}) = (p - 1)N_1$.  Plugging these expressions into the formula (\ref{eq first toric rank formula}) gives us the desired formula (\ref{eq toric rank formula}).
\end{proof}

\begin{rmk} \label{rmk genus divisible by p - 1}

As an immediate corollary of \Cref{prop toric rank formula}, we get that the toric rank of a $p$-cyclic cover of the projective line is divisible by $p - 1$, and that if that superelliptic curve is potentially degenerate, its genus is therefore also divisible by $p - 1$ (see \Cref{rmk alternate definition of split degenerate}).  In general, one sees from the genus formula given in \S\ref{sec intro} that the genus of such a curve is divisible by $(p - 1)/2$; the slightly stronger divisibility condition that we get in the potentially degenerate case is apparent from the form of the equation (\ref{eq superelliptic degenerate model}) which was already shown by \cite[Proposition 3.1(a)]{van1982galois} to characterize superelliptic curves with split degenerate reduction.

\end{rmk}

\Cref{prop toric rank formula} indicates that we will care particularly about which which nodes (resp. components) of $\SF{\Xmini}$ have $p$ nodes (resp. components) of $\SF{\Cmini}$ lying over them.  The following results will be useful in detecting where this happens.

\begin{prop} \label{prop nodes do not happen far away from branching}

With the set-up and definitions of \Cref{prop persistent nodes}, let $\mathcal{X}$ be a smooth model of the projective line $X$ dominated by $\Xmini$, and let $P$ be a $k$-point of $\mathcal{X}_s$ to which no element of the set $\mathcal{B}$ reduces.  Then the dominant morphism $\Xmini \to \mathcal{X}$ is an isomorphism above $P$, or in other words, does not contract any component to $P$.

\end{prop}

\begin{proof}
The special fiber of any smooth model consists of a single component which must be isomorphic to $\proj_k^1$; let us call it $V_0$.  Suppose that the morphism $\Xmini \to \mathcal{X}$ contracts another component of $\SF{\Xmini}$ to $P$, which implies that $P$ lies in the intersection of $V_0$ and another component $V_1$ (which must similarly be a copy of $\proj_k^1$) of $\SF{\Xmini}$.  Then no point in $\mathcal{B}$ reduces to a point in $V_1$.  If there is another component of $\SF{\Xmini}$ intersecting $V_1$ at a point $P_1 \neq P_0 := P$, then no point in $\mathcal{B}$ reduces to a point in $V_2$.  Continuing this process, it is clear from the tree-like structure of $\SF{\Xmini}$ that we can find a chain of components $V_0, V_1, \dots, V_r$ such that no point in $\mathcal{B}$ reduces to any point in $V_1, \dots, V_r \cong \proj_k^1$ and so that $V_r$ intersects the rest of $\SF{\Xmini}$ at a single point $P_{r-1}$.

It now follows from the Zariski-Nagata purity theorem, which says that the branch locus of the quotient map $\Cmini \to \Xmini$ is of pure dimension $1$, that if this map is branched over any point of the line $V_r$, the branch locus must contain the entire component $V_r$.  Let us assume for the moment that the branch locus does not contain $V_r$; then the map $\Cmini \to \Xmini$ is unbranched over $V_r$.  The only possible inverse image under an unbranched degree-$p$ cover of a projective line consists of $p$ copies of the projective line; we denote these lines forming the inverse image of $V_r$ by $\tilde{L}_1, \dots, \tilde{L}_p$ and note that they each intersect the rest of $\SF{\Cmini}$ at exactly $1$ point.  Now we have $V_r = (\mathcal{X}_r)_s$ for some smooth model $\mathcal{X}_r$ dominated by $\Xmini$; let $\mathcal{C}_r$ denote the normalization of $\mathcal{X}_r$ in $K''(C)$.  The dominant morphism $\Xmini \to \mathcal{X}_r$ lifts to a dominant morphism $\Cmini \to \mathcal{C}_r$.  Writing $L_1, \dots, L_p$ for the images in $\mathcal{C}_r$ of the lines $\tilde{L}_1, \dots, \tilde{L}_p$ under this composition of contractions, it is clear that these lines intersect each other at exactly $1$ point in $(\mathcal{C}_r)_s$.  We therefore have $w(L_1) = \dots = w(L_p) = 1$, and by \Cref{prop part of mini}, we do not have $\mathcal{C}_r \leq \Cmini$, a contradiction.  Meanwhile, if instead we assume that the whole component $V_r$ is contained in the branch locus, then its inverse image $\tilde{L}$ in $\SF{\Cmini}$ is an inseparable cover of $V_r \cong \proj_k^1$ and therefore is itself isomorphic to $\proj_k^1$; it moreover intersects the rest of $\SF{\Cmini}$ at exactly $1$ point lying over $P_{r-1}$.  By similar reasoning to that used in the unbranched case, analogously defining $L$ to be the image of $\tilde{L}$ under the dominant morphism $\Cmini \to \mathcal{C}_r$, we get $w(L) = 1$, which again thanks to \Cref{prop part of mini} produces a contradiction.
\end{proof}

\begin{cor} \label{cor crushed implies branch point}

With the set-up and definitions of \Cref{prop persistent nodes}, let $\mathcal{X}$ be a model of $X$ dominated by $\Xmini$ and satisfying that all elements of $\mathcal{B}$ reduce to the same point $P \in \mathcal{X}_s(k)$.  Then every node of $\SF{\Xmini}$ which lies in the strict transform $V$ of $\mathcal{X}_s$ in $\SF{\Xmini}$ is a branch point of the quotient map $\Cmini \to \Xmini$.

\end{cor}

\begin{proof}
By \Cref{prop nodes do not happen far away from branching}, the only possible node of $\SF{\Xmini}$ lying in $V$ is $P$.  Let us write $V' \neq V$ for the component of $\SF{\Xmini}$ that meets $V$ at the node $P$ and $\mathcal{X}' \leq \Xmini$ for the model of $X$ obtained by contracting all components of $\SF{\Xmini}$ other than $V'$.  It follows from \Cref{prop nodes do not happen far away from branching} that there is at least one point $z \in \mathcal{B}$ which reduces to (the image of) $P$ in the special fiber of the model $\mathcal{X}'$.  Then the branch point $z \in \mathcal{B}$ reduces to $P$ in the special fiber of $\Xmini$, which implies that the quotient map $\Cmini \to \Xmini$ is branched over $P$.
\end{proof}

\section{Relative normalizations of smooth models of the projective line} \label{sec normalizations}

The results of \S\ref{sec preliminaries geometric} show us that over some finite extension of $K$, there is a minimal regular model $\Cmini$ of our $p$-cyclic Galois cover $C \to X := \proj_K^1$ which is well-behaved with respect to the covering map in the sense that its corresponding quotient $\Xmini$ is a semistable model of $X$ satisfying that the inverse image of each smooth component of its special fiber $\SF{\Xmini}$ is smooth; the minimal regular model $\Cmini$ is the relative normalization of the model $\Xmini$ with respect to the covering map $C \to X$.  We note that as a result of this, given a smooth model $\mathcal{X}$ of $X$ which is dominated by $\Xmini$ and letting $\mathcal{C}$ denote its normalization in the function field of $C$, we have $\Cmini \geq \mathcal{C}$, and the strict transform of $\mathcal{C}_s$ in $\SF{\Cmini}$ is isomorphic to its desingularization (as it is smooth and has no (-1)-lines).  Keeping all of this in mind, we first provide an explicit description of smooth models of $X$ (in \S\ref{sec normalizations models of P^1}) and then, after establishing some important algebraic set-up (in \S\ref{sec normalizations part-pth-power}), proceed to explicitly construct their normalizations with respect to the covering map $C \to X$ (in \S\ref{sec normalizations p-cyclic covers}).

\subsection{Smooth models of the projective line} \label{sec normalizations models of P^1}

As before, let $X:=\mathbb{P}^1_K$ be the projective line, and let $x$ denote its standard coordinate.  Given $\alpha\in \bar{K}$ and  $\beta\in \bar{K}^\times$, one can define a smooth model $\mathcal{X}_{\alpha,\beta}$ of $X$ over the ring of integers $\mathcal{O}_{K'}$ of $K':=K(\alpha,\beta)\subseteq \bar{K}$) by declaring $\mathcal{X}_{\alpha,\beta}:=\mathbb{P}^1_{\mathcal{O}_{K'}}$, with coordinate $x_{\alpha,\beta} := \beta^{-1}(x - \alpha)$, as an $\mathcal{O}_{K'}$-scheme, and identifying the generic fiber $\mathcal{X}_\eta$ with $X$ via the affine transformation $x_{\alpha,\beta}=\beta^{-1}(x-\alpha)$. If $(\alpha_1,\beta_1)$ and $(\alpha_2,\beta_2)$ are such that $v(\alpha_1-\alpha_2)\ge v(\beta_1)=v(\beta_2)$, then $\mathcal{X}_{\alpha_1,\beta_1}$ and $\mathcal{X}_{\alpha_2,\beta_2}$ are isomorphic as models of $X$, the isomorphism being given by the change of variable $x_{\alpha_2,\beta_2}=u x_{\alpha_1,\beta_1}+ \epsilon$, where $u$ is the unit $\beta_1(\beta_2)^{-1}$ and $\epsilon$ is the integral element $\beta_2^{-1}(\alpha_1-\alpha_2)$.  The model $\mathcal{X}_{\alpha,\beta}$ therefore only depends (up to isomorphism) on the disc 
\begin{equation*}
D = D_{\alpha,b} := \{z \in \bar{K} \ | \ v(z - \alpha) \geq b\}
\end{equation*}
 of center $\alpha$ and depth $b := v(\beta)$; for this reason, we will often denote $\mathcal{X}_{\alpha,\beta}$ by $\mathcal{X}_D$.

If $\mathfrak{D} = \{D_1, \dots, D_n\}$ is a non-empty finite collection of discs of $\bar{K}$, one can form a corresponding model $\mathcal{X}_{\mathfrak{D}}$, which is defined as the minimal model dominating all the smooth models $\{\mathcal{X}_D: D\in \mathfrak{D}\}$.  Its special fiber is a reduced $k$-curve of arithmetic genus equal to $0$, \textit{i.e.} it consists of $n$ lines $L_1, \ldots, L_n$ corresponding to the discs $D_i$'s meeting each other at ordinary multiple points.  This scheme may be described algebraically as the affine space given by the coordinates $x_{\alpha, \beta}$ for $D_{\alpha, v(\beta)} \in \mathfrak{D}$ and cut out by the linear equations relating the $x_{\alpha, \beta}$'s.

\begin{prop} \label{prop models of projective line}

The construction $D \mapsto \mathcal{X}_D$ (resp. $\mathfrak{D} \mapsto \mathcal{X}_{\mathfrak{D}}$) described above defines a bijection between the discs in $\bar{K}$ (resp. finite collections of discs in $\bar{K}$) and the smooth models of $X$ (resp. the models of $X$ with reduced special fiber) defined over finite extensions of $\mathcal{O}_K$ considered up to isomorphism (where two models $\mathcal{X}_1/\mathcal{O}_{K'_1}$ and $\mathcal{X}_2/\mathcal{O}_{K'_2}$ are considered isomorphic if they become so over some common finite extension $\mathcal{O}_{K''} \supseteq \mathcal{O}_{K'_1}, \mathcal{O}_{K'_2}$).

\end{prop}

\begin{proof}
Given a smooth model of the line $\mathcal{X}$ over the ring of integers $\mathcal{O}_{K'}$ of some finite extension $K'/K$, one may prove that there exists an isomorphism of $\mathcal{O}_{K'}$-schemes $\mathcal{X}\cong \mathbb{P}^1_{\mathcal{O}_{K'}}$ (see \cite[Exercise 8.3.5]{liu2002algebraic}), which immediately implies that $\mathcal{X}$ is isomorphic, as a model, to $\mathcal{X}_D$ for some uniquely determined disc $D=D_{\alpha,b}$ with $\alpha\in K'$ and $b\in v((K')^\times)$.

Now suppose that $\mathcal{X}$ is a model of the line $X$ with reduced special fiber, and let $\lbrace L_1, \ldots, L_n \rbrace$ be the components of its special fiber $\mathcal{X}_s$.  For each $i$, let $\mathcal{X}_i$ be the model of obtained from $\mathcal{X}$ by contracting all lines in $\Irred(\mathcal{X}_s)$ except for $L_i$.  It is easy to prove that $\mathcal{X}_i$ is smooth (again, see \cite[Exercise 8.3.5]{liu2002algebraic}), so that $\mathcal{X}_i = \mathcal{X}_{D_i}$ for some uniquely determined disc $D_i \subset \bar{K}$.  Now the model $\mathcal{X}$ can be described as the minimal model dominating all the $\mathcal{X}_i$'s, \textit{i.e.} we have $\mathcal{X} \cong \mathcal{X}_{\mathfrak{D}}$ with $\mathfrak{D} = \{D_1, \ldots, D_n\}$.
\end{proof}

As the special fiber of the model $\mathcal{X}_D$ for any disc $D \subset \bar{K}$ is isomorphic to $\proj_k^1$, so is its strict transform in the special fiber of the model $\mathcal{X}_{\mathfrak{D}}$ for any collection $\mathfrak{D}$ of discs with $D \in \mathfrak{D}$; in this situation, we therefore identify each component of $(\mathcal{X}_{\mathfrak{D}})_s$ with the line $(\mathcal{X}_D)_s$ for the corresponding disc $D \in \mathfrak{D}$.  For any non-empty subset $\mathfrak{D}' \subsetneq \mathfrak{D}$, the dominant map $\mathcal{X}_{\mathfrak{D}} \to \mathcal{X}_{\mathfrak{D}'}$ is given by contracting the components $\SF{\mathcal{X}_D}$ of the special fiber $(\mathcal{X}_{\mathfrak{D}})_s$ corresponding to the discs $D \in \mathfrak{D} \smallsetminus \mathfrak{D}'$; this map can be defined algebraically in terms of ``forgetting" the coordinates $x_{\alpha, \beta}$ for $D_{\alpha, v(\beta)} \in \mathfrak{D} \smallsetminus \mathfrak{D}'$.  The following proposition describes the structure of the special fiber of $\mathcal{X}_{\mathfrak{D}}$ for a given collection $\mathfrak{D}$.

\begin{prop} \label{prop models of projective line special fiber}

Let $\mathfrak{D}$ be a collection of $\geq 2$ discs in $\bar{K}$; let $\mathcal{X}_{\mathfrak{D}}$ be defined as above; and let $K' / K$ be a finite extension over which $\mathcal{X}_{\mathfrak{D}}$ is defined.

Two components $\mathcal{X}_D$ and $\mathcal{X}_{D'}$ (with $D, D' \in \mathfrak{D}$) intersect at a singular point of $(\mathcal{X}_{\mathfrak{D}})_s$ if and only if the path $[\eta_D, \eta_{D'}] \subset \Berk$ does not contain the point $\eta_{D''}$ for any disc $D'' \subset \mathfrak{D}$.  If this singular point is a node, its thickness is equal to $\delta(\eta_D, \eta_{D'}) / v(\pi)$, where $\pi$ is a uniformizer of $K'$.

\end{prop}

\begin{proof}
This result is standard, and here we only describe the main ideas of the argument.  Let $K' / K$ be a finite extension containing centers and scalars that may be used to define each disc in $\mathfrak{D}$.  One may verify that given any discs $D = D_{\alpha, \beta} \neq D' = D_{\alpha', \beta'}$ and any collection $\mathfrak{D} \ni D, D'$, the dominant morphism $\mathcal{X}_{\mathfrak{D}} \to \mathcal{X}_D$ contracts the component $(\mathcal{X}_{D'})_s$) to the point $x_{\alpha, \beta} = \overline{\beta^{-1}(\alpha' - \alpha)}$ (resp. $x_{\alpha, \beta} = \infty$) of $(\mathcal{X}_D)_s$ if we have $\beta^{-1}(\alpha' - \alpha), \beta^{-1}\beta' \in \mathcal{O}_{K'}$ (resp. otherwise); the dominant morphism $\mathcal{X}_{\mathfrak{D}} \to \mathcal{X}_{D'}$ contracts the component $(\mathcal{X}_D)_s$ to the point on $(\mathcal{X}_{D'})_s$ whose coordinate $x_{\alpha', \beta'}$ is given by the analogous formula.  Noting that, thanks to flatness, the special fiber of any model the projective line is (geometrically) connected, it is clear that either the components $\mathcal{X}_D$ and $\mathcal{X}_{D'}$ meet at the point specified above or there is another component $(\mathcal{X}_{D''})_s$ (for some disc $D'' \in \mathfrak{D} \smallsetminus \{D, D'\}$) such that the dominant morphism $\mathcal{X}_{\mathfrak{D}} \to \mathcal{X}_{D''}$ contracts the components $(\mathcal{X}_D)_s$ and $(\mathcal{X}_{D'})_s$ to different $k$-points of $(\mathcal{X}_{D''})_s$.  One can then show, by applying the above formula to find the coordinate values of $k$-points on the component $(\mathcal{X}_{D''})_s$ to which other components contract that the latter happens if and only if we have $\eta_{D''} \in [\eta_D, \eta_{D'}]$.

To prove the statement about thickness, in the simple case that we have $\mathfrak{D} = \{D, D'\}$ and $D' = D_{\alpha, v(\beta')} \supsetneq D = D_{\alpha, v(\beta)}$, the idea is to check that a local defining equation at the node is given by $x_{\alpha, \beta} x_{\alpha, \beta'}^\vee = \beta' \beta^{-1}$, where $x_{\alpha, \beta'}^\vee = x_{\alpha, \beta'}^{-1}$ and noting that we have $v(\beta' \beta^{-1}) = \delta(\eta_D, \eta_{D'}) / v(\pi)$.  The general case follows from the fact (which can be shown using \cite[Lemma 10.3.21]{liu2002algebraic}, for instance) that contracting a component of the special fiber of any $\mathcal{O}_{K'}$-scheme which meets exactly $2$ other components at nodes of thickness $\theta_1$ and $\theta_2$ produces a node of thickness $\theta_1 + \theta_2$.
\end{proof}

\subsection{Approximations of a polynomial as a $p$th power} \label{sec normalizations part-pth-power}

We extend the valuation $v : \bar{K}^\times \to \qq$ to the \emph{Gauss valuation} $v: \bar{K}[z] \smallsetminus \{0\} \to \qq$; that is, for any non-zero polynomial $h(z) \in \bar{K}[z]$, we define $v(h)$ to be the minimum of the valuations of the coefficients of the terms appearing in $h$.

\begin{dfn} \label{dfn part-pth-power}

Given a nonzero polynomial $h(z) \in \bar{K}[z]$, a \emph{part-$p$th-power decomposition} (which we sometimes more simply call a \emph{decomposition}) of $h$ is a way of writing $h = q^p + \rho$ for some $q(z), \rho(z) \in \bar{K}[z]$, with $\deg(q) \leq \lfloor \deg(h)/p \rfloor$.  By convention, this definition allows $q = 0$.

\end{dfn}

Given a part-$p$th-power decomposition $h = q^p + \rho$, we define the rational number 
\begin{equation*}
t_{q, \rho} := v(\rho) - v(h) \in \qq \cup \{+\infty\}.
\end{equation*}

\begin{dfn} \label{dfn good and total}

We define the following properties of a part-$p$th-power decomposition $h = q^p + \rho$.

\begin{enumerate}[(a)]
\item The decomposition is said to be \emph{good} either if we have $t_{q,\rho} \geq \pfrac$ or if we have $t_{q,\rho} < \ptfrac$ and there is no decomposition $h = q_1^p + \rho_1$ such that $t_{q_1, \rho_1} > t_{q, \rho}$.
\item The decomposition is said to be \emph{total} if $\rho$ only consists of terms of prime-to-$p$ degree.
\end{enumerate}

\end{dfn}

\begin{rmk} \label{rmk good decompositions}

We make the following observations about good part-$p$th-power decompositions.

\begin{enumerate}[(a)]

\item The trivial part-$p$th-power decomposition $h = 0^p + h$ has $t_{0, h} = 0$; this immediately implies that all good decompositions $h = q^p + \rho$ satisfy $t_{q, \rho} \geq 0$. When $p$ is not the residue characteristic of $K$, the converse also holds because then we have $\pfrac=0$.

\item If $h = q^p + \rho$ is a part-$p$th-power decomposition with $t_{q, \rho} > 0$, then we immediately get $v(h) = p v(q)$.

\item If $h = q_1^p + \rho_1 = q_2^p + \rho_2$ are two good part-$p$th-power decompositions for the same polynomial $h \neq 0$, then we have $\truncate{t_{q_1, \rho_1}} = \truncate{t_{q_2, \rho_2}}$ directly from \Cref{dfn good and total}.

\end{enumerate}

\end{rmk}

In the following proposition and elsewhere below, it will be necessary to refer to a \textit{normalized reduction} of a polynomial over $\bar{K}$, defined as follows.

\begin{dfn} \label{dfn normalized reduction}

A \emph{normalized reduction} of a nonzero polynomial $h(z) \in \bar{K}[z]$ is the reduction in $k[z]$ of $\gamma^{-1}h$, where $\gamma \in \bar{K}^{\times}$ is some scalar satisfying $v(\gamma) = v(h)$.

\end{dfn}

\begin{rmk} \label{rmk normalized reduction}

Clearly a normalized reduction of a polynomial $h(z)$ is a nonzero polynomial in $k[x]$ and is unique up to scaling; thus, the degrees of the terms appearing in the normalized reduction (which is what we will be chiefly interested in for our purposes) do not depend on the particular choice of $\gamma \in K^{\times}$ in \Cref{dfn normalized reduction}.

\end{rmk}

\begin{prop} \label{prop good decomposition}

Let $h = q^p + \rho$ be a part-$p$th-power decomposition satisfying $t_{q,\rho} < \pfrac$.  Then the decomposition $h = q^p + \rho$ is good if and only if some (any) normalized reduction of $\rho$ is not the $p$th power of a polynomial with coefficients in $k$.

\end{prop}

\begin{proof}
If $t_{q,\rho} < 0$, the decomposition is not good by \Cref{rmk good decompositions}(a), and any normalized reduction of $\rho$ is a $p$th power, since it is a scalar multiple of a normalized reduction of $q^p$.  We therefore assume for the rest of the proof that $0 \leq t_{q,\rho} < \pfrac$.  We also note that, given a scalar $\gamma \in \bar{K}^\times$ with $v(\gamma) = v(h)$ and a decomposition $h = q^p + \rho$, we have a decomposition $\gamma^{-1}h = [\gamma^{-1/p}q]^p + \gamma^{-1}\rho$ of the polynomial $\gamma^{-1}h \in \bar{K}[z]$ of valuation $0$ satisfying $t_{\gamma^{-1/p}q, \gamma^{-1}\rho} = t_{q,\rho}$ and that $\gamma^{-1}\rho$ and $\rho$ have the same normalized reductions.  In light of this, we make the simplifying assumption that $v(h) = 0$.

Suppose that the part-$p$th-power decomposition $h = q^p + \rho$ is not good, so that there exists another decomposition $h = q_1^p + \rho_1$ with $v(\rho_1) = t_{q_1, \rho_1} > t_{q, \rho} = v(\rho) \geq 0$.  The fact that both $t_{q, \rho}$ and $t_{q_1, \rho_1}$ are non-negative implies that we have 
\begin{equation} \label{eq estimation of v(q)}
v(q) = \tfrac{1}{p}v(h - \rho) \geq 0; \ v(q_1) = \tfrac{1}{p}v(h - \rho_1) \geq 0.
\end{equation}

We claim that $\rho$ can be approximated (modulo polynomials of valuation $> v(\rho)$) as the $p$th power $(q_1 - q)^p$; this implies that a normalized reduction of $\rho$ is a normalized reduction of $(q_1 - q)^p$ and thus is a $p$th power.  To prove this claim, we first show that we have $v(q_1 - q) = \frac{1}{p}v(\rho)$.  Letting $\zeta_p$ be a primitive $p$th root of unity, note that for $1 \leq m \leq p - 1$, we have 
\begin{equation} \label{eq v(zeta_p q - q)}
v(\zeta_p^m q - q) = v(\zeta_p^m - 1) + v(q) \geq \tfrac{v(p)}{p-1} > \tfrac{1}{p} t_{q,\rho} = \tfrac{1}{p}v(\rho).
\end{equation}
First suppose that we have $v(q_1 - q) < \frac{1}{p}v(\rho)$.  Using (\ref{eq v(zeta_p q - q)}), we then see that $v(q_1 - \zeta_p^m q) = v(q_1 - q)$ for $1 \leq m \leq p - 1$, and therefore we have 
\begin{equation}
v(q_1^p - q^p) = \sum_{m = 0}^{p-1} v(q_1 - \zeta_p^m q) = p v(q_1 - q) < v(\rho) = v(q_1^p - q^p),
\end{equation}
a contradiction.  Now suppose that we have $v(q_1 - q) > \frac{1}{p}v(\rho)$.  Again using (\ref{eq v(zeta_p q - q)}), we then see that $v(q_1 - \zeta_p^m q) \geq \min\{v(q_1 - q), v(\zeta_p^m q - q)\} > \frac{1}{p}v(\rho)$ for $1 \leq m \leq p - 1$, and therefore we have 
\begin{equation}
v(q_1^p - q^p) = \sum_{m = 0}^{p-1} v(q_1 - \zeta_p^m q) > v(\rho) = v(q_1^p - q^p),
\end{equation}
again a contradiction.  We conclude that 
\begin{equation} \label{eq v(tilde q - q)}
v(q_1 - q) = \tfrac{1}{p}v(\rho), \text{ or equivalently, } v((q_1 - q)^p) = v(\rho).
\end{equation}

Meanwhile, it is easy to verify that we have 
\begin{equation}
p (q_1 - q) \mid (q_1^p - q^p) - (q_1 - q)^p,
\end{equation}
and furthermore, the polynomial $((q_1^p - q^p) - (q_1 - q)^p) / (p (q_1 - q))$ can be expressed as a homogeneous degree-$(p-1)$ polynomial with respect to $q$ and $q_1$ treated as variables; by (\ref{eq estimation of v(q)}), this quotient thus has non-negative valuation.  Using this and (\ref{eq v(tilde q - q)}), we deduce that 
\begin{equation} \label{eq approximation of ugly difference}
v((q_1^p - q^p) - (q_1 - q)^p) \geq v(p) + \tfrac{1}{p}v(\rho) > \tfrac{p-1}{p}t_{q,\rho} + \tfrac{1}{p}v(\rho) = \tfrac{p-1}{p}v(\rho) + \tfrac{1}{p}v(\rho) = v(\rho).
\end{equation}

Now we may write the polynomial $\rho$ as 
\begin{equation}
\rho = [\rho_1] + [(q_1^p - q^p) - (q_1 - q)^p] + [(q_1 - q)^p],
\end{equation}
where the first two expressions in square brackets on the right-hand side have valuation $> v(\rho)$ thanks to (\ref{eq approximation of ugly difference}) but the last such expression has valuation equal to $v(\rho)$ thanks to (\ref{eq v(tilde q - q)}).  This completes the proof that a normalized reduction of $\rho$ is a $p$th power.
        		
Let us now prove the converse.  Suppose that a normalized reduction of $\rho$ is a $p$th power, which is to say that we have a decomposition $\rho = s^p + \rho_0$, where $s$ and $\rho_0$ are polynomials such that $v(\rho_0) > v(\rho)$, which implies that we have $v(s) = \frac{1}{p}v(\rho)$.  Now let us now consider the part-$p$th-power decomposition $h = q_1^p + \rho_1$, where $q_1 = s + q$ and 
\begin{equation} \label{eq tilde rho}
\rho_1 = h - q_1^p = \rho + q^p - (s + q)^p = \rho - s^p - p s^{p-1} q - \dots - p s q^{p-1} = \rho_0 - p s^{p-1} q - \dots - p s q^{p-1}.
\end{equation}
As in (\ref{eq estimation of v(q)}), we have $v(q) \geq 0$, and we also have $v(s) = \frac{1}{p}v(\rho) \geq 0$.  The terms appearing after $\rho_0$ in (\ref{eq tilde rho}), each being an integral multiple of $p$ times $s$ times a homogeneous degree-$(p-1)$ polynomial with respect to $q$ and $s$ treated as variables, therefore each have valuation 
\begin{equation*}
\geq v(p) + \tfrac{1}{p}v(\rho) > \tfrac{p-1}{p}t_{q,\rho} + \tfrac{1}{p}v(\rho) = \tfrac{p-1}{p}v(\rho) + \tfrac{1}{p}v(\rho) = v(\rho).
\end{equation*}
As $\rho_1$ is thus the sum of polynomials of valuation $> v(\rho)$, we have $v(\rho_1) > v(\rho)$, implying that $t_{q_1, \rho_1} > t_{q, \rho}$, or in other words, that the part-$p$th-power decomposition $h = q^2 + \rho$ is not good.
\end{proof}

\begin{cor} \label{cor total is good}

Every total part-$p$th-power decomposition of a polynomial is good.

\end{cor}

\begin{proof}
Suppose that the decomposition $h = q^2 + \rho$ is total.  If $t_{q, \rho} \geq \ptfrac$, then we are already done, so assume that $t_{q, \rho} < \ptfrac$.  Then since $\rho$ consists only of prime-to-$p$-degree terms, the same is true of any normalized reduction of $\rho$, which consequently cannot be the $p$th power of any polynomial in $k[x]$.  Then \Cref{prop good decomposition}(a) implies that the decomposition is good.
\end{proof}

\begin{prop} \label{prop total existence}

Given a nonzero polynomial $h(z) \in \bar{K}[z]$, there always exists a total part-$p$th-power decomposition of $h$.  Therefore, thanks to \Cref{cor total is good}, there always exists a good part-$p$th-power decomposition of $h$.

\end{prop}

\begin{proof}
For $0 \leq j \leq m := \lfloor \deg(h)/p \rfloor$, let $a_j \in \bar{K}$ be the coefficient of the degree-$jp$ term of $h$.  The statement is equivalent to saying that there exists a polynomial $q(z) \in \bar{K}[z]$ of degree $m$ with the property that for $0 \leq j \leq m$, the degree-$jp$ coefficient of $q^p$ equals $a_j$.  Such a polynomial is constructed in the case of $p = 2$ in \cite[Proposition 3.15]{fiore2026clusters}.  For general $p$, our proof (which is non-constructive) uses a method suggested to the author by Robert Lemke Oliver, which is elementary but somewhat involved and far removed from the other ideas in this paper, so we relegate it to Appendix \ref{appendix b} below.
\end{proof}

We note that, if a nonzero polynomial $h \in \bar{K}[z]$ is written as a product of factors $h = \prod_{i=1}^s h_i$ with $h_i\in \bar{K}[z]$, then, given part-$p$th-power decompositions $h_i = q_i^p + \rho_i$, with $q_i, \rho_i \in \bar{K}[z]$, one can use them to form a part-$p$th-power decomposition $h = q^2 + \rho$, where $q = \prod_{i=1}^s q_i$ and $\rho = h - q^p$.  The following proposition dealing with this situation will be useful for our computations in \S\ref{sec t_f}.

\begin{prop} \label{prop product part-pth-power}

In the setting above, let $t = t_{q, \rho}$, and for $1 \leq i \leq s$, let $t_i = t_{q_i,\rho_i}$ and assume that the decompositions $h_i = q_i^p + \rho_i$ are each good.  Then we have $t \geq \min\{t_1,\dots, t_s\}$ with equality if the minimum is achieved for a unique $i \in \{1, \dots, s\}$.
\end{prop}

\begin{proof}
It clearly suffices to prove this for $s = 2$, so we assume that $s = 2$.  We have 
\begin{equation} \label{eq product part-pth-power}
\rho = h - q^p = h_1 h_2 - (q_1 q_2)^p = (q_1^p + \rho_1)(q_2^p + \rho_2) - (q_1 q_2)^p = \rho_1 q_2^p + \rho_2 q_1^p + \rho_1 \rho_2.
\end{equation}
\Cref{rmk good decompositions}(a) tells us that $t_1, t_2 \geq 0$, and therefore we have $v(\rho_i)\ge v(h_i)$ (with strict inequality when $t_i > 0$) and thus $v(q_i^p) = v(h_i - \rho_i) \geq v(h_i)$ (with equality when $t_i > 0$) for $i = 1, 2$.  We meanwhile have $v(\rho_i) = v(h_i) + t_i$ by definition of $t_i$ for $i = 1, 2$, which enables us to estimate the valuation of each term on the right-hand side of (\ref{eq product part-pth-power}): we compute $v(\rho_1 q_2^p) \geq v(h_1) + t_1 + v(h_2) = v(h) + t_1$ (with equality when $t_2 > 0$); symmetrically, we get $v(\rho_2 q_1^p) \geq v(h) + t_2$ (with equality when $t_1 > 0$); and finally, we get $v(\rho_1\rho_2) = v(h_1) + t_1 + v(h_2) + t_2 = v(h) + t_1 + t_2$.  In particular, each summand on the right-hand side of (\ref{eq product part-pth-power}) has valuation $\geq v(h) + \min\{t_1, t_2\}$ and so we have $v(\rho) \geq v(h) + \min\{t_1, t_2\}$, and we get $t = v(\rho) - v(h) \geq \min\{t_1, t_2\}$, as desired.

Now if we have $t_2 > t_1 \geq 0$, we observe from the above computations that then we have $v(\rho_1 q_2^p) = v(h) + t_1$ while the other two terms on the right-hand side of (\ref{eq product part-pth-power}) have valuation strictly greater than $v(h) + t_1$, implying that $t = v(\rho) - v(h) = t_1$.  By symmetry, we get $t = t_2$ if $t_1 > t_2$.  The assertion of equality, in the case where the minimum is achieved for a unique index, follows.
\end{proof}

\subsection{Normalizations of models of the projective line in $p$-cyclic function fields} \label{sec normalizations p-cyclic covers}

In this subsection, we compute the model of the superelliptic curve $C: y^p = f(x)$ arising (by relative normalization in the sense of \S\ref{sec preliminaries geometric Galois}) from a given smooth model of the projective line $X$.

More precisely, following \S\ref{sec normalizations models of P^1}, choose any disc $D = D_{\alpha, b} \subset \bar{K}$ with center $\alpha \in \bar{K}$ and logarithmic radius $b \in \qq$, and choose $\beta \in \bar{K}^\times$ with $v(\beta) = b$ so that $x_{\alpha, \beta} := \beta^{-1}(x - \alpha)$ is a coordinate of the corresponding smooth model $\mathcal{X}_D$ of $X$ defined over some extension $\mathcal{O}_{K'} / \mathcal{O}_K$.  We construct its normalization in $K'(C)$, which we denote by $\mathcal{C}_D$, as follows.  After possibly replacing $\mathcal{O}_{K'}$ with a further finite extension, the results of \S\ref{sec normalizations part-pth-power} (in particular \Cref{prop total existence} and \Cref{cor total is good}) tell us that $f_{\alpha,\beta}$ admits a good part-$p$th-power decomposition $f_{\alpha,\beta} = q_{\alpha,\beta}^p + \rho_{\alpha,\beta}$ over $K'$.

Now, defining $t_{q_{\alpha, \beta}, \rho_{\alpha, \beta}} := v(\rho_{\alpha, \beta}) - v(f_{\alpha, \beta})$ as in \S\ref{sec normalizations part-pth-power} and writing $t$ as abbreviated notation for this quantity, let $\gamma \in (K')^\times$ be a scalar satisfying $v(\gamma) = \min\{t, \pfrac\} + v(f_{\alpha, \beta})$; choose a $p$th root $\gamma^{1/p} \in \bar{K}$; and replace $K'$ by $K'(\gamma^{1/p})$.  Let us make the change of variable $y = \gamma^{1/p} y_{\alpha, \beta} + q_{\alpha, \beta}(x_{\alpha, \beta})$.  Now the equation (\ref{eq superelliptic}) which defines $C$ can be written as 
\begin{equation}
(\gamma^{1/p} y_{\alpha, \beta} + q_{\alpha, \beta}(x_{\alpha, \beta}))^p = q_{\alpha, \beta}^p(x_{\alpha, \beta}) + \rho_{\alpha, \beta}(x_{\alpha, \beta}).
\end{equation}
After expanding, simplifying, and dividing both sides by $\gamma$, we get the equation 
\begin{equation} \label{eq non-reduced normalization}
y_{\alpha, \beta}^p + \tbinom{p}{1} \gamma^{-1/p} q_{\alpha, \beta}(x_{\alpha, \beta}) y_{\alpha, \beta}^{p - 1} + \dots + \tbinom{p}{p - 1} \gamma^{-(p-1)/p} q_{\alpha, \beta}^{p - 1}(x_{\alpha, \beta}) y_{\alpha, \beta} = \gamma^{-1} \rho_{\alpha, \beta}(x_{\alpha, \beta}).
\end{equation}
After possibly replacing $K'$ with another finite extension, we assume that it contains a primitive $p$th root of unity $\zeta_p$, and we write $q_0 = (1 - \zeta_p) \gamma^{-1/p} q_{\alpha, \beta} \in K'[x_{\alpha, \beta}]$ and $\rho_0 = \gamma^{-1} \rho_{\alpha, \beta} \in K'[x_{\alpha, \beta}]$ (noting that $q_0$ and $\rho_0$ each depend on $\alpha$ and $\beta$ but that our notation suppresses $\alpha$ and $\beta$).

\begin{prop} \label{prop normalization}

With the above set-up, we have the following.
\begin{enumerate}[(a)]
\item The coefficients of the polynomials $\rho_0$ and $q_0$ are each integral (\textit{i.e.} have valuation $\geq 0$).
\item The $k$-curve given by the reduction of the equation in (\ref{eq non-reduced normalization}) is given by the equation 
\begin{equation} \label{eq reduced normalization}
y_{\alpha, \beta}^p + \bar{q}_0^{p-1}(x_{\alpha, \beta}) y_{\alpha, \beta} = \bar{\rho}_0(x_{\alpha, \beta}),
\end{equation}
 where $\bar{q}_0$ and $\bar{\rho}_0$ are the reductions of $q_0$ and $\rho_0$.  Moreover, if $k$ has characteristic $p$, then 
\begin{enumerate}[(i)]
\item the term $\bar{q}_0^{p-1}(x_{\alpha, \beta}) y_{\alpha, \beta}$ vanishes if and only if we have $t < \pfrac$; and 
\item the term $\bar{\rho}_0(x_{\alpha, \beta})$ vanishes if and only if we have $t > \pfrac$.
\end{enumerate}
\item The $k$-curve defined by the equation in (\ref{eq reduced normalization}) is reduced, and the map given by the coordinate $x_{\alpha, \beta}$ makes it a separable (resp. inseparable) cover of the affine $k$-line given by $x_{\alpha, \beta} \neq \infty$ in $(\mathcal{X}_D)_s \cong \proj_k^1$ if we have $t \geq \pfrac$ (resp. $t < \pfrac$).
\end{enumerate}

\end{prop}

\begin{proof}
We have $t \geq 0$ by \Cref{rmk good decompositions}(a), from which the inequalities $v(\rho_{\alpha, \beta}) \geq v(f_{\alpha, \beta})$ and $v(q_{\alpha, \beta}) = \frac{1}{p}v(f_{\alpha, \beta} - \rho_{\alpha, \beta}) \geq \frac{1}{p}v(f_{\alpha, \beta})$ (with equality if $t > 0$) follow.  Now we have the inequalities 
\begin{equation} \label{eq v(rho_0) and v(q_0)}
\begin{aligned}
v(\rho_0) = -v(\gamma) + t + v(f_{\alpha, \beta}&) \geq 0; \qquad \qquad \qquad \qquad v(q_0) = v(1 - \zeta_p) - \tfrac{1}{p}v(\gamma) + v(q_{\alpha, \beta}) \\
&\geq \tfrac{v(p)}{p-1} - \tfrac{1}{p}(\ptfrac + v(f_{\alpha, \beta})) + v(q_{\alpha, \beta}) = v(q_{\alpha, \beta}) - \tfrac{1}{p}v(f_{\alpha, \beta}) \geq 0.
\end{aligned}
\end{equation}
This proves part (a).  Moreover, one observes that the first (resp. second) inequality in (\ref{eq v(rho_0) and v(q_0)}) is strict if and only if we have $t > \pfrac$ (resp. $t < \pfrac$), which implies assertions (i) and (ii) of part (b).  The initial statement of (b) is immediate on verifying that all terms except for $y_{\alpha, \beta}^p$, $q_0^{p-1}(x_{\alpha, \beta}) y_{\alpha, \beta}$, and $\rho_0(x_{\alpha, \beta})$ vanish in the reduction of the equation in (\ref{eq non-reduced normalization}), or in other words, that the polynomial $\binom{p}{i} \gamma^{-i/p} q_{\alpha, \beta}^i(x_{\alpha, \beta}) y_{\alpha, \beta}^{p-i}$ has positive valuation for $1 \leq i \leq p - 2$: indeed, as we have $p \mid \binom{p}{i}$ for $1 \leq i \leq p - 2$, we get $v(\binom{p}{i} \gamma^{-i/p}) \geq v(p) - \frac{ipv(p)}{p-1} > 0$, and so part (b) is proved.

In order to prove part (c), we first note that the $k$-curve defined by the reduction of  (\ref{eq reduced normalization}) is not reduced if and only if the polynomial 
\begin{equation*}
F(x_{\alpha,\beta}, y) := y^p + \bar{q}_0^{p-1}(x_{\alpha, \beta}) y - \bar{\rho}_0(x_{\alpha, \beta}) \in k[x_{\alpha,\beta}, y]
\end{equation*}
 is a $p$th power.  Since $k$ has characteristic $p$, the polynomial $F$ is a $p$th power if and only if $\bar{\rho}_0$ is a $p$th power and the term $\bar{q}_0^{p-1}(x_{\alpha, \beta}) y$ vanishes (or in other words we have $\bar{q}_0 = 0$).  If $t \geq \pfrac$, then by (b)(i) we have $\bar{q}_0 \neq 0$.  Meanwhile, if $t < \pfrac$, the reduced polynomial $\bar{\rho}_0$ is a normalized reduction of $\rho_{\alpha,\beta}$, which is not a $p$th power by \Cref{prop good decomposition}.  We conclude that the $k$-curve given by $F(x_{\alpha,\beta}, y)=0$ is always reduced.

Now the coordinate $x_{\alpha, \beta}$ defines a degree-$p$ cover from the $k$-curve $F(x_{\alpha, \beta}, y) = 0$ to the affine $k$-line which is the part of $(\mathcal{X}_D)_s$ where $x_{\alpha, \beta} \neq \infty$, and it is immediate to realize that this cover is inseparable only when the term $\bar{q}_0^{p-1}(x_{\alpha, \beta})  y$ vanishes, which again happens if and only if $t < \pfrac$.  This completes the proof of part (c).
\end{proof}

We now set out to construct the desired model $\mathcal{C}_D$.  The equation in (\ref{eq non-reduced normalization}) defines a scheme $W$ over $\mathcal{O}_{K'}$ whose generic fiber is isomorphic to the affine chart $x_{\alpha, \beta} \neq \infty$ of $C$.  The coordinate $x_{\alpha, \beta}$ defines a map $W \to \mathcal{X}_D$ whose image is the affine chart $x_{\alpha,\beta} \neq \infty$ of $\mathcal{X}_D$.   Over the affine chart $x_{\alpha,\beta} \neq 0$ of $\mathcal{X}_D$, we can correspondingly form the $\mathcal{O}_{K'}$-scheme $W^\vee$ defined by the equation 
\begin{equation} \label{eq rear chart}
\begin{aligned}
\check{y}_{\alpha, \beta}^p + \tbinom{p}{1} p^{-1/(p-1)} q_0^\vee(\check{x}_{\alpha, \beta}) y_{\alpha, \beta}^{p - 1} + \dots + (q_0^\vee)^{p - 1}&(\check{x}_{\alpha, \beta}) y_{\alpha, \beta} = \rho_0^\vee(\check{x}_{\alpha, \beta}), \\
\text{where } \check{x}_{\alpha, \beta} = x_{\alpha, \beta}^{-1}; \ \check{y}_{\alpha, \beta} = x_{\alpha, \beta}^{-\lceil \deg(f)/p \rceil} y_{\alpha, \beta}; \ \ \ &q_0^\vee(\check{x}_{\alpha, \beta}) = x_{\alpha, \beta}^{-\lceil \deg(f)/d \rceil} q_0(x_{\alpha, \beta}); \\
\text{and } &\rho_0^\vee(\check{x}_{\alpha, \beta}) = x_{\alpha, \beta}^{-p \lceil \deg(f)/d \rceil} \rho_0(x_{\alpha, \beta}).
\end{aligned}
\end{equation}

\begin{rmk} \label{rmk normalization rear view}

As by \Cref{dfn part-pth-power} we have $\deg(q) \leq \lfloor \deg(f)/p \rfloor$, we have $\deg(\rho) \leq \deg(f)$.  As the polynomials $q_0, \rho_0$ are just respective scalings of $q, \rho$, their degrees are respectively equal; and it follows by construction that the rational functions $q_0^\vee(\check{x}_{\alpha, \beta}), \rho_0^\vee(\check{x}_{\alpha, \beta})$ are polynomials and that moreover, we have $v(q_0^\vee) = v(q_0)$ and $v(\rho_0^\vee) = v(\rho_0)$.  It directly follows that we get the analogs of parts (a) and (b) of \Cref{prop normalization} where the equation in (\ref{eq non-reduced normalization}), the coordinates $x_{\alpha, \beta}, y_{\alpha, \beta}$, the polynomials $q_0, \rho_0$, and their respective reductions $\bar{q}_0, \bar{\rho}_0$ are replaced by the equation in (\ref{eq rear chart}), the coordinates $\check{x}_{\alpha, \beta}, \check{y}_{\alpha, \beta}$, the polynomials $q_0^\vee, \rho_0^\vee$, and their respective reductions $\bar{q}_0^\vee, \bar{\rho}_0^\vee$.

\end{rmk}

We now define $\mathcal{C}_D$ to be the scheme obtained by gluing the affine charts $W$ and $W^\vee$ together in the obvious way: it is endowed with a degree-$p$ covering map $\mathcal{C}_D \to \mathcal{X}_D$, and its generic fiber is identified with the superelliptic curve $C$.

\begin{prop} \label{prop equations for normalization}

The scheme $\mathcal{C}_D / \mathcal{O}_{K'}$ constructed above coincides with the normalization of $\mathcal{X}_D / \mathcal{O}_{K'}$ in the function field $K'(C)$ of the superelliptic curve $C$, and it is a model of $C$ whose special fiber is reduced.

\end{prop}

\begin{proof}
We have to show that the scheme $\mathcal{C}_D$ we have constructed is normal.  The $\mathcal{O}_{K'}$-schemes $W$ and $W^\vee$ are complete intersections, and hence they are Cohen-Macaulay.  As a consequence, to check that $\mathcal{C}_D$ is normal, it is enough to prove that it is regular at its codimension-$1$ points.  Since the generic fiber of $\mathcal{C}_D$ coincides with $C$, it is certainly regular; hence, all that is left is to check that $\mathcal{C}_D$ is regular at the generic point of each irreducible component $V_i$ of the special fiber $(\mathcal{C}_D)_s$.  Since the $k$-curve $W_s$ is reduced by \Cref{prop normalization}(c), the $k$-curve $(\mathcal{C}_D)_s$ is also clearly reduced, which implies that $\mathcal{C}_D$ is certainly regular at the generic points of its irreducible fibers.  Thus, the scheme $\mathcal{C}_D$ is normal, and it is clear that $\mathcal{C}_D$ is the model of $C$ obtained by normalizing $\mathcal{C}_D / \mathcal{O}_{K'}$ in the function field $K'(C)$.
\end{proof}

The results of \S\ref{sec preliminaries geometric} provide (in \Cref{prop part of mini}) criteria (in terms of the invariants $m$, $a$, and $w$ as defined at the beginning of \S\ref{sec preliminaries geometric components}) for which smooth models $\mathcal{X}_D$ are dominated by $\Xmini$ and a formula (in \Cref{prop toric rank formula}) for computing the toric rank of $\SF{\Cmini}$ using geometric information about the strict transforms of the normalizations in $K''(C)$ of the components of $\SF{\Xmini}$ (in particular, the number of connected components and the number of points in the inverse image of a node of $\SF{\Xmini}$).  We therefore want a result that allows us to compute these quantities using the geometry of $(\mathcal{C}_D)_s$ for any disc $D \subset \bar{K}$.

\begin{prop} \label{prop normalization properties}

For any center $\alpha \in \bar{K}$ and any scalar $\beta \in \bar{K}^\times$, write $D = D_{\alpha, v(\beta)}$; let $f_{\alpha, \beta} = q_{\alpha, \beta}^p + \rho_{\alpha, \beta}$ be a good part-$p$th-power decomposition; for an appropriate finite extension $K' / K$ define the polynomials $q_0(x_{\alpha, \beta}), \rho_0(x_{\alpha, \beta}) \in K'[x_{\alpha, \beta}]$ as above; and define the normalization of $\mathcal{C}_D$ of $\mathcal{X}_D$ as above.  We follow the convention that $\deg(0) = -\infty$, and we write $\ord : \bar{K}[x_{\alpha, \beta}] \to \zz_{\geq 0} \cup \{+\infty\}$ for the function assigning to a polynomial the maximal power of $x_{\alpha, \beta}$ dividing it, with the convention that $\ord(0) = +\infty$.

If the characteristic of $k$ is not $p$, then assume that we have $q_{\alpha, \beta} = 0$.  If the characteristic of $k$ is $p$, then assume that we have $p \nmid \ord(\rho_{\alpha, \beta})$ (which in particular is guaranteed if we take $f_{\alpha, \beta} = q_{\alpha, \beta}^p + \rho_{\alpha, \beta}$ to be a total decomposition, as is possible thanks to \Cref{prop total existence}).

\begin{enumerate}[(a)]

\item There is a unique $k$-point of $(\mathcal{C}_D)_s$ with $x_{\alpha, \beta} = \infty$, which is singular if and only if we have $\deg(\bar{\rho}_0) \leq p \lceil \deg(f)/p \rceil - 2$.

If the center $\alpha$ is chosen so that $v(\alpha - z) > v(\beta)$ where $z$ is a root of $f$, then there is a unique $k$-point of $(\mathcal{C}_D)_s$ with $x_{\alpha, \beta} = 0$, which is singular if and only if we have $\ord(\bar{\rho}_0) \geq 2$.

\item \textbf{(tame case)} Suppose that the characteristic of $k$ is not $p$.

\begin{enumerate}[(i)]
\item Assume that the unique $k$-point of $(\mathcal{C}_D)_s$ with $x_{\alpha, \beta} = \infty$ is singular.  There are exactly $p$ branches passing through it if we have $p \mid \deg(\bar{\rho}_0)$, and there is a unique branch passing through it otherwise.
\item Assume that $\alpha$ satisfies the hypothesis of part (a) and that the point $(x_{\alpha, \beta}, y_{\alpha, \beta}) = (0, 0)$ is singular.  There are exactly $p$ branches passing through it if we have $p \mid \ord(\bar{\rho}_0)$, and there is a unique branch passing through it otherwise.
\end{enumerate}

\item \textbf{(wild case)} Suppose that the characteristic of $k$ is $p$.

\begin{enumerate}[(i)]
\item Assume that the unique $k$-point of $(\mathcal{C}_D)_s$ with $x_{\alpha, \beta} = \infty$ is singular.  There are exactly $p$ branches passing through it if we have $\deg(\bar{\rho}_0) < p\deg(\bar{q}_0)$, and there is a unique branch passing through it otherwise.
\item Assume that $\alpha$ satisfies the hypothesis of part (a) and that the point $(x_{\alpha, \beta}, y_{\alpha, \beta}) = (0, 0)$ is singular.  There are exactly $p$ branches passing through it if we have $\ord(\bar{\rho}_0) > p\,\ord(\bar{q}_0)$, and there is a unique branch passing through it otherwise.
\end{enumerate}

\end{enumerate}

\end{prop}

\begin{proof}
Define the coordinates $\check{x}_{\alpha, \beta}, \check{y}_{\alpha, \beta}$ and the polynomials $q_0^\vee, \rho_0^\vee$ as above, and note that (defining $\ord$ for $q_0^\vee(\check{x}_{\alpha, \beta}), \rho_0^\vee(\check{x}_{\alpha, \beta})$ as the maximal power of $\check{x}_{\alpha, \beta}$ dividing the polynomial) we have 
\begin{equation} \label{eq orders and degrees}
\begin{aligned}
\ord(\bar{q}_0^\vee) = \lceil \deg(f)/p \rceil - \deg(\bar{q}_0); \ \ &\deg(\bar{q}_0^\vee) = \lceil \deg(f)/p \rceil - \ord(\bar{q}_0); \\
\ord(\bar{\rho}_0^\vee) = p \lceil \deg(f)/p \rceil - \deg(\bar{\rho}_0); \ \ &\deg(\bar{\rho}_0^\vee) = p \lceil \deg(f)/p \rceil - \ord(\bar{\rho}_0).
\end{aligned}
\end{equation}
In particular, as we have $\deg(\bar{\rho}_0) \leq p \lceil \deg(f)/p \rceil - 1$ and $\deg(\bar{q}_0) \leq \lceil \deg(f)/p \rceil - 1$, we get $\ord(\bar{\rho}_0^\vee), \ord(\bar{q}_0^\vee) \geq 1$.  The condition that $x_{\alpha, \beta} = \infty$ is the same as saying that $\check{x}_{\alpha, \beta} = 0$, and so plugging that into the equation for the reduction of $W^\vee$ (see \Cref{rmk normalization rear view}) given by 
\begin{equation} \label{eq reduced rear view}
\check{y}_{\alpha, \beta}^p + (\bar{q}_0^\vee)^{p-1}(\check{x}_{\alpha, \beta}) \check{y}_{\alpha, \beta} = \bar{\rho}_0^\vee(\check{x}_{\alpha, \beta})
\end{equation}
yields the equation $\check{y}_{\alpha, \beta}^p + 0 \check{y}_{\alpha, \beta} = 0$, for which the only solution is $\check{y}_{\alpha, \beta} = 0$.  Now the singularity of the point $(\check{x}_{\alpha, \beta}, \check{y}_{\alpha, \beta}) = (0, 0)$ is equivalent to the simultaneous vanishing of the partial derivatives, which is easily checked to hold if and only if we have $\ord(\bar{\rho}_0^\vee) \geq 2$; by (\ref{eq orders and degrees}), this is equivalent to $\deg(\bar{\rho}_0) \leq p \lceil \deg(f)/p \rceil - 2$.  This proves the claims of part (a) about points lying above $x_{\alpha, \beta} = \infty$.

Given $\alpha_1, \alpha_2 \in D$ such that $v(\alpha_2 - \alpha_1) > v(\beta)$, the coordinates $x_{\alpha_1, \beta}$ and $x_{\alpha_2, \beta}$ have the same value at $k$-points in the reduction.  Therefore, we are free to assume that a center $\alpha$ satisfying the assumption of part (a) is itself a root of $f$.  This implies that the polynomial $f_{\alpha, \beta}$ has no constant term.  By assumption, we have $q_{\alpha, \beta} = 0$ and hence $\bar{q}_0 = 0$ or we have $p \nmid \ord(\rho_{\alpha, \beta})$ so that $x_{\alpha, \beta} \mid \ord(\rho_{\alpha, \beta})$; as we already have $x_{\alpha, \beta} \mid f_{\alpha, \beta}$, this forces $x_{\alpha, \beta} \mid q_{\alpha, \beta}$.  In either case, we get $\ord(\bar{\rho}_0^\vee), \ord(\bar{q}_0^\vee) \geq 1$, and the proof of the claims of part (a) about points lying above $x_{\alpha, \beta} = 0$ follow from analogous arguments to those used for the first statements of (a).

Now suppose that the characteristic of $k$ is not $p$, so that, by our assumption that $q_{\alpha, \beta} = 0$, the equation in (\ref{eq reduced rear view}) becomes $\check{y}_{\alpha, \beta}^p = \bar{\rho}_0^\vee(\check{x}_{\alpha, \beta})$.  Now if the point $(\check{x}_{\alpha, \beta}, \check{y}_{\alpha, \beta}) = (0, 0)$ is singular, we make the change of variables $\check{y}_{\alpha, \beta} = \check{x}_{\alpha, \beta}^{\lfloor \ord(\bar{\rho}_0^\vee)/p \rfloor} \check{y}_{\alpha, \beta}'$.  This yields the equation 
\begin{equation} \label{eq desingularized rear view tame}
(\check{y}_{\alpha, \beta}')^p = \check{x}_{\alpha, \beta}^{-p \lfloor \ord(\bar{\rho}_0^\vee)/p \rfloor} \bar{\rho}_0^\vee(\check{x}_{\alpha, \beta}).
\end{equation}
If $\check{x}_{\alpha, \beta}$ does not divide the right-hand side of the equation in (\ref{eq desingularized rear view tame}), then one verifies straightforwardly there are exactly $p$ solutions to the equation with $\check{x}_{\alpha, \beta} = 0$ and that the partial derivatives do not simultaneously vanish for any of the solutions; therefore, we have desingularized the curve given by (\ref{eq reduced rear view}) at the point $\check{x}_{\alpha, \beta} = 0$, over which there are $p$ points in the desingularization.  If, on the other hand, $\check{x}_{\alpha, \beta}$ exactly divides (resp. non-exactly divides) the right-hand side of (\ref{eq desingularized rear view tame}), then we have a unique point with $x_{\alpha, \beta} = 0$ which is non-singular (resp. a unibranch singularity: see the proof of \Cref{prop superelliptic model}(c)).

The condition that the right-hand side is not divisible by $\check{x}_{\alpha, \beta}$ is equivalent to $p \mid \ord(\bar{\rho}_0^\vee)$.  Meanwhile, it follows from (\ref{eq orders and degrees}) that we have $p \mid \ord(\bar{\rho}_0^\vee)$ if and only if $p \mid \deg(\bar{\rho}_0)$, and the claim of part (b)(i) follows.  Part (b)(ii) follows from an analogous argument.

Now suppose that the characteristic of $k$ is $p$ and that the point $(\check{x}_{\alpha, \beta}, \check{y}_{\alpha, \beta}) = (0, 0)$ is singular.  In this case we make the change of variables $\check{y}_{\alpha, \beta} = \check{x}_{\alpha, \beta}^m \check{y}_{\alpha, \beta}'$, where $m = \min\{\lfloor \ord(\bar{\rho}_0^\vee)/p \rfloor, \ord(\bar{q}_0^\vee)\}$.  This yields the equation 
\begin{equation} \label{eq desingularized rear view wild}
(\check{y}_{\alpha, \beta}')^p + [\check{x}_{\alpha, \beta}^{-m} \bar{q}_0^\vee(\check{x}_{\alpha, \beta})]^{p-1} \check{y}_{\alpha, \beta}' = \check{x}_{\alpha, \beta}^{-pm} \bar{\rho}_0^\vee(\check{x}_{\alpha, \beta}).
\end{equation}
If we have $\check{x}_{\alpha, \beta} \nmid \check{x}_{\alpha, \beta}^{-m} \bar{q}_0^\vee$, then putting $\check{x}_{\alpha, \beta} = 0$ into (\ref{eq desingularized rear view wild}) yields an Artin-Schreier equation with exactly $p$ solutions, and it is easy to check that the partial derivatives do not simultaneously vanish for any of the solutions; therefore, we have desingularized the curve given by (\ref{eq reduced rear view}) at the point $\check{x}_{\alpha, \beta} = 0$, over which there are $p$ points in the desingularization.  If, on the other hand, we have $\check{x}_{\alpha, \beta} \mid \check{x}_{\alpha, \beta}^{-m} \bar{q}_0^\vee$, then we have a unique point with $\check{x}_{\alpha, \beta} = 0$ which is non-singular (resp. a unibranch singularity) if $\check{x}_{\alpha, \beta}$ exactly divides (resp. non-exactly divides) the right-hand side of (\ref{eq desingularized rear view wild}): the unibranch singularity claim can be seen by verifying that the point is the reduction in $\SF{\mathcal{C}_D}$ of the (unique point) of $C$ with $\check{x}_{\alpha, \beta} = 0$, which is treated by \Cref{prop superelliptic model}(c); see \Cref{rmk why we ignore this type of singularity} below.

The condition $\check{x}_{\alpha, \beta} \nmid \check{x}_{\alpha, \beta}^{-m} \bar{q}_0^\vee$ is equivalent to $\ord(\bar{q}_0^\vee) = m$ and in turn to $\ord(\bar{\rho}_0^\vee) \geq p\,\ord(\bar{q}_0^\vee)$.  By (\ref{eq orders and degrees}), this is the same as saying that $\deg(\bar{\rho}_0) < p\deg(\bar{q}_0)$ (where the inequality is strict because of our assumption that $p \nmid \ord(\rho_{\alpha, \beta})$).  The claim of part (c)(i) follows, and part (c)(ii) follows from an analogous argument.
\end{proof}

\section{The $p$th-power closeness function} \label{sec t_f}

The purpose of this section is to study a function denoted $\mathfrak{t}_h$ (for a given polynomial $h$ over a finite extension of $K$) that we call the \emph{$p$th-power closeness function} and which is defined as follows.

\begin{dfn} \label{dfn t_f}

Let $h(z) \in \bar{K}[z]$ be a polynomial; let $\eta = \eta_D \in \Hyp$ be a point such that there is a center $\alpha \in D$ and scalar $\beta \in \bar{K}^\times$ such that $D = D_{\alpha, v(\beta)} = \{z \in \bar{K} \ | \ v(z - \alpha) \geq v(\beta)\}$.  Let $h_{\alpha, \beta} = q_{\alpha, \beta}^p + \rho_{\alpha, \beta}$ be a good part-$p$th-power decomposition, and define $t_{q_{\alpha, \beta}, \rho_{\alpha, \beta}}$ as in \S\ref{sec normalizations part-pth-power}.  We define the function $\mathfrak{t}_h$ on the Type II points of $\Berk$ to be given by the assignment 
\begin{equation} \label{eq t_f}
\mathfrak{t}_h(\eta) = \truncate{t_{q_{\alpha, \beta}, \rho_{\alpha, \beta}}}.
\end{equation}
The fact that $\mathfrak{t}_h$ is continuous on the space of points of Type II by \Cref{cor t_f is continuous} below will allow us to extend $\mathfrak{t}_h$ as defined above to the function $\mathfrak{t}_h : \Hyp \to \rr$ on all of $\Hyp$ by continuity.

\end{dfn}

\begin{prop} \label{prop t_f well defined}

Let $h(z) \in \bar{K}[z]$ be a polynomial.

\begin{enumerate}[(a)]

\item If for some $\alpha, \alpha' \in \bar{K}$ and $\beta, \beta' \in \bar{K}^\times$, we have $D_{\alpha, v(\beta)} = D_{\alpha', v(\beta')}$, then we have $v(h_{\alpha, \beta}) = v(h_{\alpha', \beta'})$.

\item The function $\mathfrak{t}_h$ given by \Cref{dfn t_f} is well defined, \textit{i.e.} the formula in (\ref{eq t_f}) does not depend on the choice of center $\alpha$, scalar $\beta$, or good part-$p$th-power decomposition $h_{\alpha, \beta} = q_{\alpha, \beta}^p + \rho_{\alpha, \beta}$.

\end{enumerate}

\end{prop}

\begin{proof}
It is immediately clear from definitions that replacing $\beta$ by a scalar of the same valuation does not affect the valuation of the polynomial $h_{\alpha, \beta}$.  Now if $\alpha' \in D$ is another center, we have $v(\alpha' - \alpha) \geq v(\beta)$, or equivalently $v(\beta^{-1}(\alpha' - \alpha)) \geq 0$.  Now given that 
\begin{equation}
h_{\alpha, \beta}(x_{\alpha, \beta}) = h_{\alpha', \beta}(x_{\alpha', \beta}) = h_{\alpha', \beta}(x_{\alpha, \beta} + \beta^{-1}(\alpha' - \alpha)),
\end{equation}
 it is elementary to verify that this polynomial in $x_{\alpha, \beta}$ has the same valuation as the polynomial $h_{\alpha', \beta}(x_{\alpha', \beta})$.  Thus, we have $v(h_{\alpha, \beta}) = v(h_{\alpha', \beta})$, proving part (a).

By applying this result to the polynomial $\rho$ as well, we get that the quantity $t_{q_{\alpha, \beta}, \rho_{\alpha, \beta}}$ is not affected by replacing $\alpha, \beta$ with $\alpha', \beta'$ such that $D_{\alpha, v(\beta)} = D_{\alpha', v(\beta')}$.  Finally, \Cref{rmk good decompositions}(c) says that the quantity $t_{q_{\alpha, \beta}, \rho_{\alpha, \beta}}$ does not depend on the choice of good part-$p$th-power decomposition, and part (b) is proved.
\end{proof}

Note that as we have $t_{q_{\alpha, \beta}, \rho_{\alpha, \beta}} \geq 0$ for any good part-$p$th-power decomposition $h_{\alpha, \beta} = q_{\alpha, \beta}^p + \rho_{\alpha, \beta}$ (as in \Cref{rmk good decompositions}(a)) and as we have $\pfrac = 0$ if the residue characteristic of $K$ is not $p$, the function $\mathfrak{t}_h$ is identically $0$ on $\Hyp$ in the tame setting.  Therefore, there is nothing interesting to prove in this section if one assumes the tame setting, and so instead \textit{for the rest of this section, we assume that we are in the wild setting, \textit{i.e.} that the residue characteristic of $K$ is $p$.}

The rest of this section is devoted to studying the function $\mathfrak{t}_f$, culminating in \Cref{thm t_f formula one component} below.

\begin{rmk} \label{rmk ctt different}

In \cite[\S4]{cohen2016morphisms}, given a wild cover $Y \to X$ of curves whose associated cover of Berkovich spaces will be denoted $\mathfrak{Y} \to \mathfrak{X}$, the \emph{(Berkovich) different function} $\delta_{\mathfrak{Y} / \mathfrak{X}}$ is defined on the points of $\mathfrak{Y}$ not of Type I and taking values in $(0, 1]$; this function takes the value $1$ outside of the wild ramification locus.  A number of properties of this different function, as well as formulas in some cases, have appeared in the literature, and it has been used to analyze morphisms of Berkovich curves and their ramification and branch loci, as in \cite{obus2016wild} for instance (where it is also compared to Kato's \emph{depth Swan conductor}).  One should be able to show that, in the case of our $p$-cover $Y = C \to X = \proj_K^1$ over residue characteristic $p$, this function $\delta_{\mathfrak{Y} / \mathfrak{X}}$ is essentially a lift of the function $\mathfrak{t}_f: \Hyp \to [0, \ptfrac]$ to $\mathfrak{Y}$, where the relationship is given by $\mathfrak{t}_f = \ptfrac(1 - \log_{|p|}(\delta_{\mathfrak{Y} / \mathfrak{X}}))$, with $|\cdot|$ denoting the absolute value for the non-archimedean field $\cc_p$.  This connection may be useful in providing a different perspective on the behavior of $\mathfrak{t}_f$ as well as providing explicit formulas.

\end{rmk}

\subsection{Families of polynomials and their valuations} \label{sec t_f families}

We begin with an elementary lemma.

\begin{lemma} \label{lemma degrees of terms appearing}

Let $h(x) \in \bar{K}[x]$ be a monic polynomial, and choose elements $\alpha \in \bar{K}$ and $\beta \in \bar{K}^\times$.  Let $m$ (resp. $n$) denote the lowest (resp. highest) degree appearing among the terms of some (any) normalized reduction of the polynomial $h_{\alpha, \beta}$.  Let $\underline{\mathcal{R}} \subset \bar{K}$ denote the multiset of roots of $h$, and let $\underline{\mathfrak{c}}$ (resp. $\underline{\mathfrak{s}}$) be the multiset of roots $z$ of $h$ satisfying $v(z - \alpha) > v(\beta)$ (resp. $v(z - \alpha) \geq v(\beta)$), where each element of each multiset appears with its multiplicity as a root of $h$.  Then we have $m = \#\underline{\mathfrak{c}}$ and $n = \#\underline{\mathfrak{s}}$.

\end{lemma}

\begin{proof}
We may write the polynomial $h_{\alpha, \beta}(x_{\alpha, \beta})$ as 
\begin{equation} \label{eq scaled and translated polynomial}
\prod_{z \in \underline{\mathcal{R}}} (\beta x_{\alpha, \beta} - (z - \alpha)) = \beta^{\#\underline{\mathfrak{s}}} \prod_{z \in \underline{\mathcal{R}} \smallsetminus \underline{\mathfrak{s}}} (z - \alpha) \prod_{z \in \underline{\mathcal{R}} \smallsetminus \underline{\mathfrak{s}}} (\beta (z - \alpha)^{-1} x_{\alpha, \beta} - 1) \prod_{z \in \underline{\mathfrak{s}}} (x_{\alpha, \beta} - \beta^{-1}(z - \alpha)),
\end{equation}
where products are taken with multiplicity.  One verifies easily that the right-most two products above multiply to a polynomial whose coefficients have non-negative valuation and whose degree-$\#\underline{\mathfrak{s}}$ term has coefficient equal to a power of $-1$.  Any normalized reduction of $h_{\alpha, \beta}$ is therefore equal to a non-zero scalar times the polynomial $\prod_{z \in \underline{\mathfrak{s}}} (x_{\alpha, \beta} - \overline{\beta^{-1}(z - \alpha)})$.  This is clearly a monic polynomial whose degree equals $\#\underline{\mathfrak{s}}$ and whose lowest-degree term has degree equal to the number (counted with multiplicity) of elements $z \in \underline{\mathfrak{s}}$ satisfying $\overline{\beta^{-1}(z - \alpha)} = 0$, which equals $\#\underline{\mathfrak{c}}$.
\end{proof}

Given a polynomial $h(x) \in \bar{K}[x]$ and a point $\eta \in \Hyp$ of Type II, we will consequently denote by $\vfun_h(\eta)$ the valuation of $h_{\alpha,\beta}$ for any $\alpha$ and $\beta$ such that $\eta = \eta_{D_{\alpha, v(\beta)}}$; this is a well defined function by \Cref{prop t_f well defined}(a).  When a center $\alpha \in \bar{K}$ is fixed, we may consider the function $b \mapsto \vfun_h(\eta_{D_{\alpha, b}}) \in \qq$ defined for all $b \in \qq = v(\bar{K}^\times)$.

\begin{lemma} \label{lemma t_f difference formula}

Choose polynomial $h(z) \in \bar{K}[z]$ and a center $\alpha \in \bar{K}$, and let $h = q^p + \rho$ be a part-$p$th-power decomposition such that the decomposition $h_{\alpha, 1} = q_{\alpha, 1}^p + \rho_{\alpha, 1}$ is total.  Then for $\eta \in \Hyp$ of Type II with $\eta > \eta_\alpha$, we have 
\begin{equation}
\mathfrak{t}_h(\eta) = \min\{\vfun_\rho(\eta) - \vfun_h(\eta), \ptfrac\}.
\end{equation}

\end{lemma}

\begin{proof}
It is immediate from definitions that if $h_{\alpha, 1} = q_{\alpha, 1}^p + \rho_{\alpha, 1}$ is a total decomposition, then so is $h_{\alpha, \beta} = q_{\alpha, \beta}^p + \rho_{\alpha, \beta}$ for any $\beta \in \bar{K}^\times$; thus, by \Cref{cor total is good}, each decomposition $h_{\alpha, \beta} = q_{\alpha, \beta}^p + \rho_{\alpha, \beta}$ is good.  We thus have $\mathfrak{t}_h(\eta) = \min\{t_{q_{\alpha, \beta}, \rho_{\alpha, \beta}}, \pfrac\}$ if $\alpha \in \bar{K}$ and $\beta \in \bar{K}^\times$ are such that $\eta$ corresponds to the disc $D_{\alpha, v(\beta)}$ (a condition implying $\eta > \eta_\alpha$).  The desired formula follows.
\end{proof}

\begin{lemma} \label{lemma mathfrakv}

Choose a polynomial $h(z) \in \bar{K}[z]$, and fix a center $\alpha \in \bar{K}$.  With the above set-up, we have the following.

\begin{enumerate}[(a)]
\item The function on $\qq$ given by $b \mapsto \vfun_h(\eta_{D_{\alpha,b}}) \in \qq$ is a continuous piecewise linear function with decreasing positive integer slopes each $\leq \deg(h)$.
\item For any $b \in \qq$ and $\beta \in \bar{K}^{\times}$ such that $v(\beta) = b$, the left (resp.\ right) derivative of the function $c \mapsto \vfun_h(\eta_{D_{\alpha,c}}) \in \qq$ at $c = b$ coincides with the highest (resp.\ lowest) degree of the variable $x_{\alpha,\beta}$ appearing in some (any) normalized reduction of $h_{\alpha,\beta}$.
\end{enumerate}
\end{lemma}

\begin{proof}
Write $H_i$ for the $z^i$-coefficient of $h_{\alpha,1}$, and note that $\beta^i H_i$ is the $z^i$-coefficient of $h_{\alpha,\beta}$ for any scalar $\beta$.  Now given any $b \in \qq$ and $\beta \in \bar{K}^{\times}$ with $v(\beta) = b$, by definition we have 
\begin{equation} \label{eq mathfrakv}
\vfun_h(\eta_{D_{\alpha,b}}) = \min_{0 \leq i \leq \deg(h)} \{v(\beta^i H_i)\} = \min_{0 \leq i \leq \deg(h)} \{v(H_i) + ib\}.
\end{equation}
All the properties of the function $b \mapsto \vfun_h(\eta_{D_{\alpha,b}})$ stated in the lemma immediately follow from the explicit expression given above.
\end{proof}

\begin{cor} \label{cor t_f is continuous}

For any $h(z) \in \bar{K}[z]$, the function $\mathfrak{t}_h$ on the points of Type II is continuous.

\end{cor}

\begin{proof}
Choosing an element $\alpha \in \bar{K}$ and applying \Cref{lemma mathfrakv} to both the polynomials $h$ and $\rho$, we get that the functions on $\qq$ given by $b \mapsto \vfun_\rho(\eta_{D_{\alpha, b}})$ and $b \mapsto \vfun_h(\eta_{D_{\alpha, b}})$ are each continuous.  This is the same as saying that the functions $\vfun_\rho$ and $\vfun_h$ are continuous on the (dense) subspace of the open path $(\eta_\alpha, \eta_\infty)$ consisting of points of Type II.  It now follows from \Cref{lemma t_f difference formula} that the function $\mathfrak{t}_h$ is continuous on this subspace.  By construction, the space $\Hyp$ is the union of the open paths $(\eta_\alpha, \eta_\infty)$ over all $\alpha \in \bar{K}$; moreover, any two distinct such paths $(\eta_\alpha, \eta_\infty)$ and $(\eta_{\alpha'}, \eta_\infty)$ intersect (their intersection being the half-open path $[\eta_\alpha \vee \eta_{\alpha'}, \eta_\infty)$).  The corollary follows.
\end{proof}

In light of the above corollary, as in \Cref{dfn t_f}, from now on we consider $\mathfrak{t}_h$ and $\vfun_h$ as functions on all of $\Hyp$ for any $h(x) \in \bar{K}[x]$.

For any subset $\mathfrak{r} \subseteq \mathcal{B} \smallsetminus \{\infty\}$ of roots of $f$, let us write $f_{\mathfrak{r}}$ for the monic polynomial of degree $\#\underline{\mathfrak{r}}$ whose multiset of roots coincides with $\underline{\mathfrak{r}} \subset \underline{\mathcal{B}}$.  In order to avoid cumbersome notation, given such a set $\mathfrak{r}$, we write $\mathfrak{t}_{\mathfrak{r}}$ for the function $\mathfrak{t}_{f_{\mathfrak{r}}} : \Hyp \to [0, \pfrac]$.

\begin{prop} \label{prop mathfrakt minimum}

Given a disjoint union $\mathfrak{r} = \mathfrak{r}_1 \sqcup \ldots \sqcup \mathfrak{r}_s \subset \mathcal{B} \smallsetminus \{\infty\}$, for all $\eta \in \Hyp$, we have 
\begin{equation}
\mathfrak{t}_{\mathfrak{r}}(\eta) \geq \min\{\mathfrak{t}_{\mathfrak{r}_1}(\eta), \ldots, \mathfrak{t}_{\mathfrak{r}_s}(\eta)\},
\end{equation}
with equality if the minimum is achieved for a unique $i \in \{1, \dots, s\}$.

\end{prop}

\begin{proof}
Fixing a point $\eta = \eta_{D_{\alpha, v(\beta)}}$ (of Type II) for some $\alpha \in \bar{K}$ and $\beta \in \bar{K}^\times$, for $1 \leq i \leq N$, let $f_{\mathfrak{r}_i} = q_i^p + \rho_i$ be part-$p$th-power decompositions such that the decompositions $(f_{\mathfrak{r}_i})_{\alpha, v(\beta)} = (q_i)_{\alpha, v(\beta)}^p + (\rho_i)_{\alpha, v(\beta)}$ are good.  Then, by setting $q = \prod_i q_i$ and $\rho = f_{\mathfrak{r}} - q^p$, we obtain a part-$p$th-power decomposition for $f_{\mathfrak{r}}$ which according to \Cref{prop product part-pth-power} satisfies 
\begin{equation}
t_{q_{\alpha, \beta}, \rho_{\alpha, \beta}} \geq \min_{1 \leq i \leq s} \{t_{(q_i)_{\alpha, \beta}, (\rho_i)_{\alpha, \beta}}\}
\end{equation}
with equality when the minimum on the right-hand side is achieved for a unique index $i$.  From this (and by continuity thanks to \Cref{cor t_f is continuous}) it follows that we have $\mathfrak{t}_{\mathfrak{r}}(\eta) \geq \min\{\mathfrak{t}_{\mathfrak{r}_i}(\eta)\}_{1 \leq i \leq s}$ for all $\eta \in \Hyp$ with equality when the minimum is achieved for a unique $i$.
\end{proof}

\subsection{Our main result on the behavior of the $p$th-power closeness function}

We now present the main result of this section, followed by an important corollary.

\begin{thm} \label{thm t_f formula one component}

Let $\mathcal{B} = \mathfrak{r}_0 \sqcup \dots \sqcup \mathfrak{r}_h$ be the partition such that the decomposition of $\Sigma_{\mathcal{B}}^*$ as the disjoint union of its connected components can be written as $\Sigma_{\mathfrak{r}_0} \sqcup \dots \sqcup \Sigma_{\mathfrak{r}_h}$ as in \Cref{prop connected components Sigma}.  Fix $i \in \{0, \dots, h\}$; write $\mathfrak{r}_i' = \mathfrak{r}_i \smallsetminus \{\infty\}$; and let $\mathfrak{s}_i$ be the smallest cluster containing $\mathfrak{r}_i$ if $i \neq 0$.

\begin{enumerate}[(a)]

\item For $\eta \in \Hyp$, we have $\mathfrak{t}_{\mathfrak{r}_i'}(\eta) = 0$ if and only if $\eta \in \Sigma_{\mathfrak{r}_i}$.

\item Let $\mathfrak{r}$ be the (disjoint) union of some collection of the subsets $\mathfrak{r}_1, \dots, \mathfrak{r}_h$ (so that $\infty \notin \mathfrak{r}$), and let $\mathfrak{s}$ be the smallest cluster containing $\mathfrak{r}$.  For $\alpha \in \mathfrak{r}$, the function $b \mapsto \mathfrak{t}_{\mathfrak{r}}(\eta_{D_{\alpha, b}})$, when not taking the value $\pfrac$, has decreasing negative integer slopes on $(-\infty, d(\mathfrak{s}))$.  In particular, this claim holds for the function $b \mapsto \mathfrak{t}_{\mathfrak{r}_i'}(\eta_{D_{\alpha, b}})$ over $(-\infty, d(\mathfrak{s}_i)$ for any index $i \neq 0$.

\item Let $\mathfrak{r}$ be the (disjoint) union of some collection of the subsets $\mathfrak{r}_0, \dots, \mathfrak{r}_h$ (so that $\mathfrak{r}$ may contain $\infty$); write $\mathfrak{r}' = \mathfrak{r} \smallsetminus \{\infty\}$; and let $\mathfrak{s}$ be the smallest cluster containing $\mathfrak{r}$.  For $\alpha \in \bar{K} \smallsetminus \mathfrak{r}$, the function $b \mapsto \mathfrak{t}_{\mathfrak{r}'}(\eta_{D_{\alpha, b}})$ has decreasing positive integer slopes on $(c, \infty)$ until it attains the value $\pfrac$, where $c \in \rr$ is such that we have $D_{\alpha, b} \cap \mathfrak{r} = \varnothing$ for $b > c$.  In particular, this claim holds for the function $b \mapsto \mathfrak{t}_{\mathfrak{r}_i'}(\eta_{D_{\alpha, b}})$ over $(c, \infty)$ for any index $i$ and $c$ such that $D_{\alpha, b} \cap \mathfrak{r}_i = \varnothing$ for $b > c$.

\item Assume that $\underline{\mathcal{B}}$ is clustered in pairs.  Then for $\eta \in \Hyp$, we have the formula 
\begin{equation} \label{eq t_r_i formula}
\mathfrak{t}_{\mathfrak{r}_i'}(\eta) = \min\{\delta(\eta, \Sigma_{\mathfrak{r}_i}), \ptfrac\}.
\end{equation}

\end{enumerate}

\end{thm}

\begin{cor} \label{cor t_f formula}

Assume that the multiset $\underline{\mathcal{B}}$ of branch points is clustered in pairs.  Then the function $\mathfrak{t}_f$ satisfies $\mathfrak{t}_f(\eta) = 0$ if and only if we have $\eta \in \Sigma_{\mathfrak{r}_i}$ for some $i \in \{0, \dots h\}$.  If moreover $\underline{\mathcal{B}}$ is clustered in $\ptfrac$-separated pairs, a formula for $\mathfrak{t}_f : \Hyp \to [0, \pfrac]$ is given by 

\begin{equation} \label{eq t_f formula}
\mathfrak{t}_f(\eta) = \min\big\{\delta(\eta, \Sigma_{\mathfrak{r}_0}), \dots, \delta(\eta, \Sigma_{\mathfrak{r}_h}), \ptfrac\big\} = \min\big\{\mathfrak{t}_{\mathfrak{r}_0'}(\eta), \dots, \mathfrak{t}_{\mathfrak{r}_h'}(\eta), \ptfrac\big\}.
\end{equation}

\end{cor}

\begin{proof}
This is immediate from combining \Cref{thm t_f formula one component}(a)(d) with \Cref{prop mathfrakt minimum} and noting that the hypothesis of being clustered in pairs (resp. $\ptfrac$-separated pairs) implies that each $\eta \in \Hyp$ can lie in at most one subspace $\Sigma_{\mathfrak{r}_i} \subset \Hyp$ (resp. for each $\eta \in \Hyp$, at most one of the distances $\delta(\eta, \Sigma_{\mathfrak{r}_i})$ can be $< \ptfrac$).
\end{proof}

We devote the rest of this section to proving \Cref{thm t_f formula one component}.  First we present the following lemma, which shows in particular that parts (a)--(c) of \Cref{thm t_f formula one component} describe the behavior of the function $\mathfrak{t}_{\mathfrak{r}_i'}$ (for any $i \in \{0, \dots, h\}$) on the whole domain $\Hyp$.

\begin{lemma} \label{lemma all of hyp covered}

With the set-up of \Cref{thm t_f formula one component}, each $\eta \in \Hyp$ satisfies the hypothesis of (at least) one of parts (a)--(c) of \Cref{thm t_f formula one component}, \textit{i.e.} satisfies $\eta \in \Sigma_{\mathfrak{r}_i}$ (as in part (a)), $\eta \geq \eta_{\alpha, d(\mathfrak{s}_i)} = \eta_{\mathfrak{s}_i}$ for some $\alpha \in \mathfrak{r}_i$ (as in part (b)), or $\eta = \eta_{D_{\alpha, b}}$ for some $\alpha \in \bar{K} \smallsetminus \mathfrak{r}_i$ and $b > c$ where $c \in \rr$ satisfies the condition in part (c).

\end{lemma}

\begin{proof}
Choose any point $\eta = \eta_D \in \Hyp$.  In the case that $D \cap \mathfrak{r}_i = \varnothing$, we have $D = D_{\alpha, b}$ for some $\alpha \in \bar{K}$ and $b \in \rr$ satisfying $b > c := \max\{z - \alpha\}_{\zeta \in \mathfrak{r}_i}$; clearly $c$ satisfies the condition in \Cref{thm t_f formula one component}(c).  In the case that $D \supseteq \mathfrak{r}_i$, we have $D \supseteq D_{\mathfrak{s}_i} = D_{\alpha, d(\mathfrak{s}_i)}$ for some (any) $\alpha \in \mathfrak{r}_i$.  If this containment is strict, then we get $D = D_{\alpha, b}$ with $b < d(\mathfrak{s}_i)$ as in \Cref{thm t_f formula one component}(b).  If instead we have the equality $D = D_{\mathfrak{s}_i}$, then we have $\eta \in \Sigma_{\mathfrak{r}_i}$ by \Cref{prop highest point of Lambda}(a).  Finally, if $D \cap \mathfrak{r}_i \neq \varnothing$ and $D \not\supset \mathfrak{r}_i$, \Cref{prop highest point of Lambda}(b) implies that we have $\eta \in \Sigma_{\mathfrak{r}_i}$ as in \Cref{thm t_f formula one component}(a).
\end{proof}

\subsection{Proof of \Cref{thm t_f formula one component}(a)(b)(c)}

By continuity (guaranteed by \Cref{cor t_f is continuous}), it suffices to assume that any given $\eta \in \Hyp$ that we are dealing with is of Type II.

Choose any point $\eta = \eta_D \in \Sigma_{\mathfrak{r}_i}$ which is not the maximal point of $\Sigma_{\mathfrak{r}_i}$, and let $\mathfrak{s} = D \cap \mathcal{B}$.  Applying \Cref{prop highest point of Lambda}(b) and keeping in mind that $\eta \notin \Sigma_{\mathfrak{r}_j}$ for $j \neq i$, we get $\mathfrak{s} \cap \mathfrak{r}_i \neq \varnothing$ and $\mathfrak{s} \not\supset \mathfrak{r}_i$ as well as either $\mathfrak{s} \supset \mathfrak{r}_j$ or $\mathfrak{s} \cap \mathfrak{r}_j = \varnothing$ for $j \neq i$.  Now \Cref{cor cardinality-p clusters} implies that we have $p \nmid \#\underline{\mathfrak{s}}$.  Meanwhile, the fact that $p \mid \#\underline{\mathfrak{r}}_j$ for $0 \leq j \leq h$ by \Cref{prop connected components Sigma} now also implies $p \mid \#(\underline{\mathfrak{s}} \smallsetminus \bigcup_{j \neq i} \underline{\mathfrak{r}}_j)$, so we have $p \nmid \#(\underline{\mathfrak{s}} \cap \underline{\mathfrak{r}}_i)$.

For simplicity of notation, we write $h = f_{\mathfrak{r}_i}$.  Choose a root $\alpha \in \mathfrak{s} \cap \mathfrak{r}_i$ of the polynomial $h$, and choose a scalar $\beta \in \bar{K}^\times$ such that we have $D = D_{\alpha, v(\beta)}$.  Now the multiset $\underline{\mathfrak{s}} \cap \underline{\mathfrak{r}}_i$, whose cardinality we have shown is not divisible by $p$, is precisely the subset of roots $z$ of $h$ (counted with multiplicity) satisfying $v(z - \alpha) \geq v(\beta)$.  Then by \Cref{lemma degrees of terms appearing}, the highest degree among terms appearing in any normalized reduction of $h_{\alpha, \beta}$ is not divisible by $p$.  A normalized reduction of $h_{\alpha, \beta}$ thus cannot be a $p$th power, so by \Cref{prop good decomposition}, the decomposition $h_{\alpha, \beta} = 0^p + h_{\alpha, \beta}$ is good.  We thus get $\mathfrak{t}_h(\eta) = 0$.  As this holds for all non-maximal points $\eta \in \Sigma_{\mathfrak{r}_i}$ and the maximal point of $\Sigma_{\mathfrak{r}_i}$ is a limit point of this component, the function $\mathfrak{t}_{\mathfrak{r}_i'}$ (which is continuous thanks to \Cref{cor t_f is continuous}) is identically $0$ on $\Sigma_{\mathfrak{r}_i}$, and the backward direction of \Cref{thm t_f formula one component}(a) is proved.

With $\mathfrak{r}$ and $\mathfrak{s}$ as in the statement of part (b), let us now investigate the behavior of $\mathfrak{t}_{\mathfrak{r}}$ on an open path beginning at the point $\eta_{\alpha, d(\mathfrak{s})} = \eta_{\mathfrak{s}}$ for some (any) $\alpha \in \mathfrak{r}$ (which, by \Cref{prop highest point of Lambda}, is the maximal point of $\Sigma_{\mathfrak{r}}$) and ``going upward" in the direction of $\eta_\infty$.  We now write $h = f_{\mathfrak{r}}$  Now for any $b < d(\mathfrak{s})$, we have $v(z - \alpha) > b$ for each root $z \in \mathfrak{r}$.  Then given $\beta \in \bar{K}^\times$ with $v(\beta) = b$, by \Cref{lemma degrees of terms appearing}, the only degree appearing among the terms of a normalized reduction of $h_{\alpha, \beta}$ is $\#\underline{\mathfrak{r}}_i$ (which we deduce from \Cref{prop connected components Sigma} to be divisible by $p$).  \Cref{lemma mathfrakv}(b) then tells us that the function $b \mapsto \vfun_h(\eta_{D_{\alpha, b}})$, as $b$ increases along $(-\infty, d(\mathfrak{s}))$, is a linear function with slope equal to $\#\underline{\mathfrak{r}}$.  Letting $h = q^p + \rho$ be a part-$p$th-power decomposition such that the associated decomposition $h_{\alpha, 1} = q_{\alpha, 1}^p + \rho_{\alpha, 1}$ is total, we have $\deg(\rho_{\alpha, b}) \leq \deg(h) - 1 = \#\underline{\mathfrak{r}} - 1$ for $b \in (-\infty, d(\mathfrak{s})]$, so \Cref{lemma mathfrakv}(b) tells us that the function $b \mapsto \vfun_\rho(\eta_{D_{\alpha, b}})$, as $b$ increases along $(-\infty, d(\mathfrak{s}))$, is piecewise linear with decreasing positive integer slopes each $\leq \#\underline{\mathfrak{r}} - 1$.  Now applying \Cref{lemma t_f difference formula}, we get that the function $b \mapsto \mathfrak{t}_{\mathfrak{r}}(\eta_{D_{\alpha, b}})$, as $b$ increases along $(-\infty, d(\mathfrak{s}))$, is piecewise linear with decreasing slopes which are negative wherever the function takes a value $< \pfrac$ (which in the case of $\mathfrak{r} = \mathfrak{r}_i$ happens eventually as it approaches the value $0$ at $b = d(\mathfrak{s}_i)$, which corresponds to the point $\eta_{\mathfrak{s}_i} \in \Sigma_{\mathfrak{r}_i}$).  This not only proves \Cref{thm t_f formula one component}(b) but shows that we have $\mathfrak{t}_{\mathfrak{r}_i}(\eta_{D_{\alpha, b}}) > 0$ for $b < d(\mathfrak{s}_i)$, thus proving the forward direction of part (a) of the theorem (its contrapositive) for $\eta \in (\eta_{\mathfrak{s}_i}, \eta_\infty)$.

The proof of \Cref{thm t_f formula one component}(c), together with the proof of the forward direction of part (a) of the theorem (its contrapositive) for $\eta \in (\eta_{D_{\alpha, c}}, \eta_\alpha)$ with $\alpha \in \bar{K}$ and $c \in \rr$ satisfying the hypotheses of part (c), is similar and is only sketched here.  Let $\mathfrak{r}$ and $\mathfrak{s}$ be as in the statement of part (c); again write $h = f_{\mathfrak{r}}$; and fix $\alpha \in \bar{K} \smallsetminus \mathfrak{r}$ and $c \in \rr$ so that $D_{\alpha, b} \cap \mathfrak{r} = \varnothing$ for $b > c$.  This clearly implies the inequality $b > v(z - \alpha)$ for all $b > c$ and all $z \in \mathfrak{r}$.  An application of \Cref{lemma degrees of terms appearing} shows that any normalized reduction of $h_{\alpha, \beta}$ with $b := v(\beta) > c$ is a constant, and then \Cref{lemma mathfrakv}(b) says that the function $b \mapsto \vfun_h(\eta_{D_{\alpha, b}})$ is constant on $(c, \infty)$.  Given a decomposition $h = q^p + \rho$ such that the decomposition $h_{\alpha, 1} = q_{\alpha, 1}^p + \rho_{\alpha, 1}$ is total, \Cref{lemma mathfrakv}(b) also says that $b \mapsto \vfun_\rho(\eta_{D_{\alpha, b}})$ has decreasing positive integer slopes, and the desired results (including this case of the forward direction of part (a)) again follow from applying \Cref{lemma t_f difference formula}; in particular, the function $b \mapsto \mathfrak{t}_{\mathfrak{r}'}(\eta_{D_{\alpha, b}})$ has decreasing positive integer slopes.

\subsection{Proof of \Cref{thm t_f formula one component}(d): the behavior of $\mathfrak{t}_{\mathfrak{r}_i'}$ in the case of clustered in pairs}

We have now proved parts (a)--(c) of \Cref{thm t_f formula one component}, and so our final task in this section is to prove part (d).  To this end, we assume that $\underline{\mathcal{B}}$ is clustered in pairs, so that we have $\#\underline{\mathfrak{r}}_i = p$ and $\#\mathfrak{r}_i = 2$, and we may write $\mathfrak{r}_i = \{z, w\}$, where the root $z$ has multiplicity $m \in \{1, \dots, p - 1\}$, which means that $w$ has multiplicity $p - m$.  

Let us assume for now that we have $i \neq 0$, so that $w \neq \infty$ and $\mathfrak{r}_i' = \mathfrak{r}_i$.  Then our formula for $h = f_{\mathfrak{r}_i'}$ is $h(x) = (x - z)^m (x - w)^{p-m}$.  Note also that the cluster $\mathfrak{s}_i$ corresponding to $\mathfrak{r}_i$ is the smallest to contain $z$ and $w$ and so satisfies $d(\mathfrak{s}_i) = v(w - z)$.

First we consider the situation in \Cref{thm t_f formula one component}(b), where $\alpha$ is chosen to be a root in $\mathfrak{r}_i$ and $\beta \in \bar{K}^\times$ satisfies $v(\beta) < d(\mathfrak{s}_i) = v(w - z)$.  Without loss of generality, we set $\alpha = z$ and get 
\begin{equation}
h_{\alpha, \beta}(x_{\alpha, \beta}) = \beta^m x^m (\beta x - (w - z))^{p-m}.
\end{equation}
A total part-$p$th-power decomposition $h_{\alpha, 1} = q_{\alpha, 1}^p + \rho_{\alpha, 1}$ is given by setting $q_{\alpha, 1}(x_{\alpha, 1}) = x_{\alpha, 1}$ and $\rho_{\alpha, 1} = h_{\alpha, 1} - x_{\alpha, 1}^p$.  For $m \leq i \leq p - 1$, the degree-$i$ coefficient of $\rho_{\alpha, \beta}$ equals ${p - m \choose p - i} \beta^i (w - z)^{p - i}$.  Noting that $p \nmid {p - m \choose p - i}$ for $m \leq i \leq p - 1$, one immediately verifies using the inequality $v(\beta) < v(w - z)$ that the degree-$(p-1)$ coefficient has minimal valuation among the coefficients of $\rho_{\alpha, \beta}$, which means that the highest degree appearing in a normalized valuation of $\rho_{\alpha, \beta}$ is $p - 1$.

Applying \Cref{lemma mathfrakv}(b), we get that the function $b \mapsto \vfun_\rho(\eta_{D_{\alpha, b}})$ is linear with slope $p - 1$ as $b$ increases along $(-\infty, d(\mathfrak{s}_i))$.  We saw in the proof of \Cref{thm t_f formula one component}(b) that $b \mapsto \vfun_h(\eta_{D_{\alpha, b}})$ is linear with slope $\#\underline{\mathfrak{r}}_i = p$ on $(-\infty, d(\mathfrak{s}_i))$; applying \Cref{lemma t_f difference formula}, we get that the piecewise linear function $b \mapsto \mathfrak{t}_{\mathfrak{r}_i}(\eta_{D_{\alpha, b}})$ on $(-\infty, d(\mathfrak{s}_i))$ is identically $\pfrac$ and then has slope $-1$, and we know from \Cref{thm t_f formula one component}(a) that it then reaches $0$ at $b = d(\mathfrak{s}_i)$.  This implies that $\mathfrak{t}_{\mathfrak{r}_i}(\eta_{D_{\alpha, b}}) = \min\{d(\mathfrak{s}_i) - b, \pfrac\}$ for points $\eta_{D_{\alpha, b}} > \eta_{\mathfrak{s}_i} \in \Sigma_{\mathfrak{r}_i}$; as $\eta_{\mathfrak{s}_i}$ is the maximal point in $\Sigma_{\mathfrak{r}_i}$, this gives the formula in (\ref{eq t_r_i formula}) for $\mathfrak{t}_{\mathfrak{r}_i'}$.

Now we consider the situation for $\eta \in (\eta_{D_{\alpha, c}}, \eta_\alpha)$ for some choice of $\alpha \in \bar{K} \smallsetminus \mathfrak{r}_i$ and $c \in \rr$ satisfying the hypothesis of \Cref{thm t_f formula one component}(c), which is equivalent to saying that $c \geq v(z - \alpha), v(w - \alpha)$.  It clearly suffices to assume that we have $c = \max\{v(z - \alpha), v(w - \alpha)\}$; note that this implies $D_{\alpha, c} \cap \mathfrak{r}_i \neq \varnothing$.  Let us also assume without loss of generality that we have $v(w - \alpha) \geq v(z - \alpha)$, which implies $c = v(w - \alpha)$.

Applying \Cref{prop highest point of Lambda} tells us that we have $\eta_{D_{\alpha, c}} \in \Sigma_{\mathfrak{r}_i}$ or $\eta_{D_{\alpha, c}} > \eta_{\mathfrak{s}_i}$ (and in the latter case with $\delta(\eta_{D_{\alpha, c}}, \Sigma_{\mathfrak{r}_i}) = \delta(\eta_{D_{\alpha, c}}, \eta_{\mathfrak{s}_i}) = d(\mathfrak{s}_i) - c$) and thus already know that $\mathfrak{t}_{\mathfrak{r}_i'}(\eta_{D_{\alpha, c}})$ is given by the formula in (\ref{eq t_r_i formula}).  Note that for $b > c$, we have 
\begin{equation} \label{eq sum of distances}
\delta(\eta_{D_{\alpha, b}}, \Sigma_{\mathfrak{r}_i}) = \delta(\eta_{D_{\alpha, b}}, \eta_{D_{\alpha, c}}) + \delta(\eta_{D_{\alpha, c}}, \Sigma_{\mathfrak{r}_i});
\end{equation}
indeed, if $\eta_{D_{\alpha, c}} \in \Sigma_{\mathfrak{r}_i}$ (which is equivalent to $c = v(w - \alpha) \geq v(z - \alpha) = v(z - w)$), this is clear as the path $(\eta_{D_{\alpha, c}}, \eta_{D_{\alpha, b}})$ does not intersect $\Sigma_{\mathfrak{r}_i}$, whereas otherwise, as the closest point to $\eta_{D_{\alpha, c}}$ in $\Sigma_{\mathfrak{r}_i}$ is its maximal point $\eta_{\mathfrak{s}_i}$ and as $D_{\alpha, b} \cap \mathfrak{r}_i = \varnothing$, we have $\eta_{D_{\alpha, c}} = \eta_{D_{\alpha, b}} \vee \eta_{\mathfrak{s}_i}$.  From (\ref{eq sum of distances}) we get the formula 
\begin{equation} \label{eq formula for distance}
\delta(\eta_{D_{\alpha, b}}, \Sigma_{\mathfrak{r}_i}) = 
\begin{cases}
b - c &\text{if } c = v(w - \alpha) \geq v(z - \alpha) = v(z - w) \\
v(z - w) + b - 2c &\text{if } c = v(w - \alpha) = v(z - \alpha) \leq v(z - w)
\end{cases}.
\end{equation}

Now, fixing a scalar $\beta \in \bar{K}^\times$ with $b := v(\beta) > c$, we have 
\begin{equation} \label{eq h_(alpha', beta)}
h_{\alpha, \beta}(x_{\alpha, \beta}) = (\beta x_{\alpha, \beta} - (z - \alpha))^m (\beta x_{\alpha, \beta} - (w - \alpha))^{p-m}.
\end{equation}
A total part-$p$th-power decomposition $h_{\alpha, 1} = q_{\alpha, 1}^p + \rho_{\alpha, 1}$ is given by setting $q_{\alpha, 1}(x_{\alpha, 1}) = x_{\alpha, 1} + (-1)^p \mu^m \nu^{p-m}$, where $\mu$ and $\nu$ are respectively $p$th roots of $z - \alpha$ and of $w - \alpha$, and setting $\rho_{\alpha, 1} = h_{\alpha, 1} - (x_{\alpha, 1} + (-1)^p \mu^m \nu^{-m})^p$.

\begin{lemma} \label{lemma t_f formula one component}

Given the situation directly above, we have the following.

\begin{enumerate}[(a)]

\item The constant term of $h_{\alpha, \beta}$ has valuation equal to $m v(z - \alpha) + (p - m) v(w - \alpha)$.

\item Suppose that we have $v(w - z) - v(z - \alpha) < \ptfrac$.  Then 
the linear coefficient of $\rho_{\alpha, \beta}$ has valuation equal to $b + (m - 1) v(z - \alpha) + (p - m - 1) v(w - \alpha) + v(z - w)$.

\end{enumerate}

\end{lemma}

\begin{proof}

On noting that the constant term of $h_{\alpha, \beta}$ is $(-1)^p (z - \alpha)^m (w - \alpha)^{p-m}$, part (a) is immediate.

To prove part (b), we observe that we may express the linear term of $h_{\alpha, \beta}$ as 
\begin{equation}
\begin{aligned}
\beta (\alpha - z)^{m-1} (\alpha - w)^{p-m-1} &(m(\alpha - w) + (p - m)(\alpha - z)) \\
&= \beta (\alpha - z)^{m-1} (\alpha - w)^{p-m-1} \big(p (\alpha - z) + m (z - w)\big).
\end{aligned}
\end{equation}
Since the linear coefficient of $q_{\alpha, \beta}$ is $p \beta \mu^{m(p-1)} \nu^{(p-m)(p-1)}$, we may write the linear coefficient of $\rho_{\alpha, \beta}$ as 
\begin{equation} \label{eq linear term of rho}
\begin{aligned}
\beta (\alpha - z)^{m-1} (\alpha - w)&^{p-m-1} \big(p (\alpha - z) + m (z - w) - p \mu^{m(p-1) - p(m-1)} \nu^{(p-m)(p-1) - p(p-m-1)} \big) 
\\ &= \beta (\alpha - z)^{m-1} (\alpha - w)^{p-m-1} \big(p (\alpha - z)(1 - \mu^{-m} \nu^m) + m(z - w)).
\end{aligned}
\end{equation}

Suppose we have $v(w - z) = v(z - \alpha) < v(w - \alpha)$.  Then we have $\mu = \frac{1}{p} v(z - \alpha) < \frac{1}{p} v(w - \alpha) = \nu$ and thus $v(1 + \mu^{-m} \nu^m) = 0$.  From this we get the inequality 
\begin{equation}
v(p(\alpha - z)(1 + \mu^{-m} \nu^m)) = v(p) + v(\alpha - z) > v(\alpha - z) = v(m(z - w)).
\end{equation}
It follows from this and the formula in (\ref{eq linear term of rho}) that the linear term of $\rho_{\alpha, \beta}$ has valuation equal to that of $\beta (\alpha - z)^{m-1} (\alpha - w)^{p-m-1} (m(z - w))$, which equals $b + m v(z - \alpha) + (p - m - 1) v(w - \alpha) = m v(z - \alpha) + (p - m) v(w - \alpha) + (b - c)$, thus proving part (b) when $v(z - \alpha) < v(w - \alpha)$.

Now suppose instead that we have $v(w - z) \geq v(z - \alpha) = v(w - \alpha)$.  Note that $\mu^{-m} \nu^m$ is a $p$th root of $1 + \frac{w - z}{z - \alpha}$.  We thus get the product 
\begin{equation} \label{eq product in proof of one component}
\prod_{n = 0}^{p-1} (1 - \zeta_p^n \mu^{-m}\nu^m) = \frac{z - w}{z - \alpha}.
\end{equation}
Assume by way of contradiction that we have $v(1 - \mu^{-m}\nu^m) \geq \tfrac{v(p)}{p-1}$.  Then each other term of the above product also has valuation $\geq \tfrac{v(p)}{p-1}$, as it differs from $1 - \mu^m\nu^m$ by an element of the form $(1 - \zeta_p^n) \mu^m \nu^{-m}$, which has valuation $\tfrac{v(p)}{p-1}$.  It follows that the product has valuation $\geq \ptfrac$.  This contradicts the fact that, by hypothesis, we have $v(\frac{z - w}{z - \alpha}) = v(w - z) - v(z - \alpha) < \ptfrac$.  Therefore, we have $v(1 - \mu^{-m}\nu^m) < \tfrac{v(p)}{p-1}$.  Then each other term of the product in (\ref{eq product in proof of one component}) has the same valuation as $v(1 - \mu^{-m}\nu^m)$, and from (\ref{eq product in proof of one component}) we get 
\begin{equation}
p v(1 - \mu^{-m}\nu^m) = v(w - z) - v(z - \alpha).
\end{equation}
We then observe, again using our hypothesis that $v(w - z) - v(z - \alpha) < \ptfrac$, the inequality 
\begin{equation}
\begin{aligned}
v(p&(\alpha - z)(1 + \mu^{-m}\nu^m)) = v(p) + v(\alpha - z) + \tfrac{1}{p}(v(w - z) - v(z - \alpha)) \\
&> \tfrac{p-1}{p}(v(w - z) - v(z - \alpha)) + v(\alpha - z) + \tfrac{1}{p}(v(w - z) - v(z - \alpha)) = v(m(z - w)).
\end{aligned}
\end{equation}
Using the formula in (\ref{eq linear term of rho}), it now follows that the linear term of $\rho_{\alpha, \beta}$ again has valuation equal to that of $\beta (\alpha - z)^{m-1} (\alpha - w)^{p-m-1} (m(z - w))$, and so part (b) is proved when $v(z - \alpha) = v(w - \alpha)$ as well.
\end{proof}

\begin{cor} \label{cor t_f formula one component}

Suppose in the above situation that we have $v(w - z) - v(z - \alpha) < \ptfrac$.  Then the valuation of the linear term of $\rho_{\alpha, \beta}$ minus the valuation of the constant term of $h_{\alpha, \beta}$ equals $\delta(\eta_{D_{\alpha, b}}, \Sigma_{\mathfrak{r}_i})$.

\end{cor}

\begin{proof}

This follows from \Cref{lemma t_f formula one component} via a straightforward calculation using the formula (\ref{eq formula for distance}).
\end{proof}

Using the fact that $b = v(\beta) > v(z - \alpha), v(w - \alpha)$ and combining \Cref{lemma degrees of terms appearing} with \Cref{lemma mathfrakv}(b), we deduce that $b \mapsto \vfun_h(\eta_{D_{\alpha, b}})$ is constant on $(c, \infty)$ and equal to the valuation of the constant term of $h_{\alpha, \beta}$.  \Cref{cor t_f formula one component} then tells us that, under the hypothesis $v(w - z) - v(z - \alpha) < \ptfrac$, the valuation of the linear term of $\rho_{\alpha, \beta}$ minus $\vfun_h(D_{\alpha, b})$ equals $\delta(\eta_{D_{\alpha, b}}, \Sigma_{\mathfrak{r}_i})$.  It now follows from \Cref{lemma t_f difference formula} that 
\begin{equation} \label{eq t_h one component estimation}
\mathfrak{t}_h(\eta_{D_{\alpha, b}}) \leq \min\{\delta(\eta_{D_{\alpha, b}}, \Sigma_{\mathfrak{r}_i}), \ptfrac\}
\end{equation}
 for $b \in (c, \infty)$, as long as $v(w - z) - v(z - \alpha) < \ptfrac$.

As the function $b \mapsto \vfun_h(\eta_{D_{\alpha, b}})$ is constant on $(c, \infty)$, combining \Cref{lemma degrees of terms appearing} with \Cref{lemma mathfrakv}(a), we get that the function $b \mapsto \mathfrak{t}_h(\eta_{D_{\alpha, b}})$ is a continuous piecewise linear function with decreasing positive linear slopes on $[c, \infty)$ until it reaches the value $\ptfrac$, at which point it becomes a constant function.  The function $b \mapsto \min\{\delta(\eta_{D_{\alpha, b}}, \Sigma_{\mathfrak{r}_i}), \ptfrac\}$ has slope $1$ on $[c, \infty)$ until it reaches the constant value $\ptfrac$, and (\ref{eq t_h one component estimation}) says that this function bounds $b \mapsto \mathfrak{t}_h(\eta_{D_{\alpha, b}})$ from above over $(c, \infty)$ (and thus over $[c, \infty)$ by continuity given by \Cref{cor t_f is continuous}) under the hypothesis $v(w - z) - v(z - \alpha) < \ptfrac$.  If the inequality in (\ref{eq t_h one component estimation}) were strict for some $b \in (c, \infty)$, then our observation about decreasing positive integer slopes would imply that $\mathfrak{t}_h(\eta_{D_{\alpha, c}}) < \min\{\delta(\eta_{D_{\alpha, c}}, \Sigma_{\mathfrak{r}_i}), \ptfrac\}$.  But this would contradict our earlier observation that $\mathfrak{t}_h(\eta_{D_{\alpha, c}})$ is indeed given by the formula in (\ref{eq t_f formula}).  It follows that the inequality in (\ref{eq t_h one component estimation}) is in fact an equality, and so we have proved the formula in (\ref{eq t_f formula}) for $\eta = \eta_{D_{\alpha, b}}$ with $b \in (c, \infty)$ under the hypothesis $v(w - z) - v(z - \alpha) < \ptfrac$.

Now assume instead that we have $v(w - z) - v(z - \alpha) \geq \ptfrac$.  Then we get 
\begin{equation}
\delta(\eta_{D_{\alpha, c}}, \Sigma_{\mathfrak{r}_i}) = \delta(\eta_{D_{\alpha, c}}, \eta_{\mathfrak{s}_i}) = d(\mathfrak{s}_i) - c \geq \ptfrac,
\end{equation}
 and so (\ref{eq t_r_i formula}) yields $\mathfrak{t}_{\mathfrak{r}_i'}(\eta_{D_{\alpha, c}}) = \pfrac$.  By \Cref{lemma mathfrakv}(a), the function $b \mapsto \vfun_\rho(\eta_{D_{\alpha, b}})$ on $[c, \infty)$ is non-decreasing and, as $b \mapsto \vfun_h(\eta_{D_{\alpha, b}})$ is constant, by \Cref{lemma t_f difference formula}, the function $b \mapsto \mathfrak{t}_{\mathfrak{r}_i'}(\eta_{D_{\alpha, b}})$ is non-decreasing on $[c, \infty)$; as it takes the value $\pfrac$ at $b = c$, it must be identically $\pfrac$ on $[c, \infty)$ and therefore still satisfies the formula in (\ref{eq t_r_i formula}).

In light of \Cref{lemma all of hyp covered}, we have thus entirely proved that the formula claimed by \Cref{thm t_f formula one component}(d) for all $\eta \in \Hyp$ in the case that $\infty \notin \mathfrak{r}_i$.  If $i = 0$, so that $\mathfrak{r}_i = \{z, w := \infty\}$, the proof is similar but much easier: the polynomial $h$ has the simpler form $h(x) = (x - z)^m$; given an element $\alpha \in \bar{K} \smallsetminus \{z\}$, we have the total part-$p$th-power decomposition $h_{\alpha, 1} = q_{\alpha, 1}^p + \rho_{\alpha, 1}$ where $q_{\alpha, 1}$ is a constant polynomial whose $p$th power is $(\alpha - z)^m$ and $\rho_{\alpha, 1} = h_{\alpha, 1} - q_{\alpha, 1}^p$; and \Cref{lemma t_f formula one component} is replaced by the following slightly analogous (and much easier) lemma.

\begin{lemma} \label{lemma t_h one component for i = 0}

Given the situation directly above and fixing a scalar $\beta \in \bar{K}^\times$, we have the following.

\begin{enumerate}[(a)]

\item Each coefficient of $h_{\alpha, \beta}$ has valuation $\geq m v(z - \alpha)$, while the constant term has precisely this valuation.

\item The linear coefficient of $\rho_{\alpha, \beta}$ has minimal valuation among coefficents of $\rho_{\alpha, \beta}$.

\end{enumerate}

\end{lemma}

\begin{proof}
For $0 \leq j \leq m \leq p - 1$, the degree-$j$ coefficient of $h_{\alpha, \beta}(x_{\alpha, \beta}) = (\beta x_{\alpha, \beta} - (z - \alpha))^m$ is given by ${m \choose j} \beta^j (\alpha - z)^{m-j}$.  Note that we have $p \nmid {m \choose j}$, so the valuation of the degree-$j$ coefficient equals $j v(\beta) + (m - j) v(z - \alpha)$.  As we have $b = v(\beta) > v(\alpha - z)$, the valuations of the coefficients increase as the degree $j$ increases.  This proves part (a), and in light of the fact that $\rho_{\alpha, \beta}$ is just $h_{\alpha, \beta}$ minus its constant term, it also proves part (b).
\end{proof}

\Cref{lemma t_h one component for i = 0} implies, thanks to a combination of Lemmas \ref{lemma degrees of terms appearing}, \ref{lemma t_f difference formula}, and \ref{lemma mathfrakv}(b), that the function $b \mapsto \mathfrak{t}_h(\eta_{D_{\alpha, b}})$ for $b \in (v(z - \alpha), \infty)$ is linear with slope $1$ until it reaches the constant value $\ptfrac$ (and again by continuity of $\mathfrak{t}_h$, this holds over $[v(z - \alpha), \infty)$ as well).  The fact that $z \in D_{\alpha, v(z - \alpha)}$ implies that $\eta_{D_{\alpha, v(z - \alpha)}} \in \Sigma_{\mathfrak{r}_0'}$; \Cref{thm t_f formula one component}(a) then implies that 
\begin{equation}
\mathfrak{t}_h(\eta_{D_{\alpha, v(z - \alpha)}}) = 0 = \delta(\eta_{D_{\alpha, v(z - \alpha)}}, \Sigma_{\mathfrak{r}_0'}).
\end{equation}
As the function $b \mapsto \delta(\eta_{D_{\alpha, b}}, \Sigma_{\mathfrak{r}_0'})$ is linear with slope $1$ on $[v(z - \alpha), \infty)$, the formula in (\ref{eq t_f formula}) for $\eta = \eta_{D_{\alpha, b}}$ again follows.

Again, thanks to \Cref{lemma all of hyp covered}, we have now shown that the claimed formula in (\ref{eq t_r_i formula}) for $\mathfrak{t}_{\mathfrak{r}_0'} = \mathfrak{t}_h$ holds for all $\eta \in \Hyp$, and \Cref{thm t_f formula one component}(d) is proved when $i = 0$ as well.

\section{Proof of the main theorem} \label{sec proof of main}

In this section we prove \Cref{thm main}.  Throughout this section, we assume that our ground field $K$ contains the roots of $f$ as well as the $p$th roots of unity, and we let $K' / K$ be a finite extension over which $C$ attains semistable reduction, and such that the regular model $\Cmini / \mathcal{O}_{K'}$ satisfies the property given by \Cref{prop persistent nodes}: under the quotient map $\Cmini \to \Xmini$, the image of every node of $\SF{\Cmini}$ is a node of $\SF{\Xmini}$.

Define $\Upsilon$ to be the subspace of $\Sigma_{\mathcal{B}}^0$ given by $\mathfrak{t}_f^{-1}(\pfrac) \cap \Sigma_{\mathcal{B}}^0$ (noting that if $p$ is not the residue characteristic of $K$, then we have $\Upsilon = \Sigma_{\mathcal{B}}^0$).  For any $\alpha \in \bar{K}$ and $\beta \in \bar{K}^\times$, recall the smooth model $\mathcal{X}_{\alpha, \beta}$ of $X := \proj_K^1$ in \S\ref{sec normalizations models of P^1} and the construction in \S\ref{sec normalizations p-cyclic covers} of the normalization $\mathcal{C}_{\alpha, \beta}$ of $\mathcal{X}_{\alpha, \beta}$ in $K'(C)$.  Let $\Upsilon_{K'} \subset \Upsilon$ denote the subset of points $\eta_{D_{\alpha, v(\beta)}}$ such that the model $\mathcal{C}_{\alpha, \beta}$ is defined over $K'$.

\begin{prop} \label{prop relating nodes to Upsilon}

Choose any point $\eta = \eta_{D_{\alpha, v(\beta)}} \in \Upsilon_{K'}$, and let $U$ be the connected component of $\Upsilon$ that $\eta$ lies in.  Choose a point $z \in \mathcal{B}$, and let $Q$ be the unique $k$-point of the special fiber $(\mathcal{C}_{\alpha, \beta})_s$ lying over the point $(\mathcal{X}_{\alpha, \beta})_s$ to which $z$ reduces.

We have $(\eta, \eta_z) \cap U \neq \varnothing$ if and only if $Q$ is a singular point with exactly $p$ branches passing through it.

\end{prop}

\begin{proof}
First of all, the uniqueness of the point $Q$ is implied by \Cref{prop normalization properties}(a) and also results from the fact that the $\mathcal{O}_{K'}$-point of $\Xmini$ extending $z \in X(K)$ is clearly a branch point of the morphism $\Cmini \to \Xmini$ (as $z$ is a branch point of the morphism $C \to X$).

Let $\mathcal{B} = \mathfrak{c}_1 \sqcup \dots \sqcup \mathfrak{c}_s \sqcup \mathfrak{c}_\infty$ be the decomposition determined by $\eta$ given by \Cref{prop removing a point from convex hull} so that, according to that proposition, we have $\mathfrak{s} := \mathfrak{c}_1 \sqcup \dots \sqcup \mathfrak{c}_s = D_{\alpha, v(\beta)} \cap \mathcal{B}$.

Fix a part-$p$th-power decomposition $f = q^p + \rho$ with $q = 0$ if the residue characteristic of $K$ is not $p$ and such that the decomposition $f_{\alpha, 1} = q_{\alpha, 1}^p + \rho_{\alpha, 1}$ is total (and thus, the decomposition $f_{\alpha, \beta'} = q_{\alpha, \beta'}^p + \rho_{\alpha, \beta'}$ for each $\beta' \in \bar{K}$ is good by \Cref{cor total is good}) if the residue characteristic of $K$ is $p$.  Define the reduced polynomials $\bar{q}_0(x_{\alpha, \beta}), \bar{\rho}_0(x_{\alpha, \beta}) \in k[x]$ as in the discussion preceding \Cref{prop normalization}, so that an affine equation for the special fiber $(\mathcal{C}_{\alpha, \beta})_s$, as in \Cref{prop normalization}(b), is given by 
\begin{equation}
y_{\alpha, \beta}^p + \bar{q}_0^{p-1}(x_{\alpha, \beta}) y_{\alpha, \beta} = \bar{\rho}_0(x_{\alpha, \beta}).
\end{equation}
Note that if $\bar{\rho}_0 \neq 0$, then $\bar{\rho}_0$ is a normalized reduction of $\rho$, and the integer $\ord(\bar{\rho}_0)$ is the least degree appearing among terms in any normalized reduction of $\rho_{\alpha, \beta}$; note also that a similar statement holds for $\bar{q}_0$ and $q_{\alpha, \beta}$.

\textbf{Tame case:} Assume that $k$ has characteristic $\neq p$, so that we have $\Upsilon = \Sigma_{\mathcal{B}}^0$.  Suppose first that $Q \in \SF{\mathcal{C}_{\alpha, \beta}}(k)$ is the point with $x_{\alpha, \beta} = \infty$, which is equivalent to $v(z - \alpha) < v(\beta)$ and thus $z \notin D_{\alpha, v(\beta)}$.  By \Cref{prop boundary points}(b), we have $(\eta, \eta_z) \cap U = \varnothing$ if and only if we have $p \nmid \#\underline{\mathfrak{s}}$, which by \Cref{lemma degrees of terms appearing} is equivalent to $p \nmid \deg(\bar{\rho}_0)$.  \Cref{prop normalization properties}(b)(i) says that this in turn is equivalent to $Q$ not having $p$ branches passing through it.

Now suppose that $Q \in \SF{\mathcal{C}_D}(k)$ is the point with $x_{\alpha, \beta} \neq \infty$, which is equivalent to $v(z - \alpha) \geq v(\beta)$ and thus $z \in D_{\alpha, v(\beta)}$; it is clear that translating $\alpha$ by any element of valuation $\geq v(\beta)$ does not affect the disc $D_{\alpha, v(\beta)}$ or any object determined by it, so, after possibly making such a translation, we assume that $x_{\alpha, \beta} = 0$.  By \Cref{prop boundary points}(c), we have $(\eta, \eta_z) \cap U = \varnothing$ if and only if we have $p \nmid \#\underline{\mathfrak{c}}$ where $\mathfrak{c} \in \{\mathfrak{c}_1, \dots, \mathfrak{c}_s\}$ is the cluster containing $z$ whose parent cluster $\mathfrak{c}'$ coincides with $\mathfrak{s}$ (and thus its elements $w$ may be characterized as satisfying $v(w - \alpha) > v(\beta)$), which by \Cref{lemma degrees of terms appearing} is equivalent to $p \nmid \ord(\bar{\rho}_0)$.  \Cref{prop normalization properties}(b)(ii) says that this in turn is equivalent to $Q$ not having $p$ branches passing through it.

\textbf{Wild case:} Assume that $k$ has characteristic $p$.  Define the functions $\vfun_h, \vfun_q, \vfun_\rho : \Hyp \to \rr_{\geq 0}$ as in \S\ref{sec t_f families}.  Suppose first that $Q \in \SF{\mathcal{C}_{\alpha, \beta}}(k)$ is the point with $x_{\alpha, \beta} = \infty$, which again is equivalent to $z \notin D_{\alpha, v(\beta)}$.  This implies the existence of a subpath $(\eta, \eta') \subset (\eta, \eta_z)$ with $\eta' > \eta$ (as is obvious if $z = \infty$; otherwise, one may take $\eta' = \eta \vee \eta_z$ for instance).  Let us choose $\eta'$ close enough to $\eta$ that (by continuity thanks to \Cref{cor t_f is continuous}) the function $\mathfrak{t}_f$ is positive-valued over the path $(\eta', \eta)$; it then follows from \Cref{rmk good decompositions}(b) that we have $\vfun_f = p \vfun_q$ on $(\eta', \eta]$.

Now the condition $(\eta, \eta_z) \cap U = \varnothing$ is equivalent to the assertion that $\mathfrak{t}_f$ is monotonically increasing on the open path $(\eta', \eta)$ (as one goes in the direction of $\eta$).  It follows from \Cref{lemma t_f difference formula} that on $(\eta', \eta]$ that this is equivalent to 
\begin{equation} \label{eq t_f = v_rho - p v_q}
\mathfrak{t}_f = \vfun_\rho - \vfun_f = \vfun_\rho - p \vfun_q.
\end{equation}
In particular, this tells us that we have $v(\rho_{\alpha, \beta}) - p v(q_{\alpha, \beta}) = \mathfrak{t}_f(\eta) = \pfrac$, and one checks from the construction of $\rho_0$ in \S\ref{sec normalizations p-cyclic covers} that $\bar{\rho}_0$ is a normalized reduction of $\rho$ (\textit{i.e.} $\bar{\rho}_0$ does not vanish).  The equations in (\ref{eq t_f = v_rho - p v_q}) and the fact that $\mathfrak{t}_f$ is increasing on $(\eta', \eta]$ are moreover equivalent to saying that the left derivative of $\vfun_\rho$ at $\eta$ is greater than $p$ times the left derivative of $\vfun_q$ at $\eta$.  By \Cref{lemma mathfrakv}, this is equivalent to $\deg(\bar{\rho}_0) > p \deg(\bar{q}_0)$.  \Cref{prop normalization properties}(c)(i) says that this in turn is equivalent to $Q$ not having $p$ branches passing through it.

For the case that the $Q \in \SF{\mathcal{C}_{\alpha, \beta}}(k)$ is the point with $x_{\alpha, \beta} \neq \infty$ (in which we may assume, as in the tame case, that $x_{\alpha, \beta} = 0$ at $Q$), the argument is precisely analogous: we instead choose $\eta' \in (\eta, \eta_z)$ (so that $\eta > \eta'$), get ``monotonically decreasing" in place of ``monotonically increasing" and ``left derivatives" in place of ``right derivatives", get (\ref{eq t_f = v_rho - p v_q}) on $[\eta, \eta')$, and find that $(\eta, \eta_z) \cap U = \varnothing$ is equivalent to $\ord(\bar{\rho}_0) < p \,\ord(\bar{q}_0)$, which by \Cref{prop normalization properties}(c)(ii) is equivalent to $Q$ not having $p$ branches passing through it.
\end{proof}

\begin{cor} \label{cor relating nodes to Upsilon}

For any point $\eta = \eta_{D_{\alpha, v(\beta)}} \in \Upsilon_{K'}$ which is not isolated in $\Upsilon$, we have $\mathcal{X}_{\alpha, \beta} \leq \Xmini$.  Moreover, the special fiber $(\mathcal{C}_{\alpha, \beta})_s$ has a singular point with $p$ branches passing through it, and it has $p$ components if and only if $\eta$ is an interior point of the subspace $\Upsilon \subset \Sigma_{\mathcal{B}}$.

\end{cor}

\begin{proof}
Given such a point $\eta = \eta_{D_{\alpha, v(\beta)}} \in \Upsilon_{K'}$, let $U \subseteq \Upsilon$ be the connected component containing $\eta$, and let $A \subseteq \mathcal{B}$ be the set of branch points $z$ satisfying $(\eta, \eta_z) \cap U \neq \varnothing$.  As $\eta$ is not isolated in $\Upsilon$, there is a point $\eta' \in \Upsilon$ with $[\eta, \eta'] \subset \Upsilon$.  If $\eta' > \eta$, then we have $(\eta, \eta') \subset (\eta, \eta_\infty)$; otherwise, as we have $\eta' \in \Sigma_{\mathcal{B}}$, it follows from \Cref{prop highest point of Lambda}(a)(b) that the disc corresponding to $\eta'$ contains some root $z \in \mathcal{B} \smallsetminus \{\infty\}$, and one can check that we must have $\eta' \neq \eta \vee \eta' \in [\eta, \eta_z]$ so that we have $(\eta, \eta') \subset (\eta, \eta_z)$.  It follows that we have $A \neq \varnothing$.  From similar reasoning, one sees that we have $A = \mathcal{B}$ if and only if $\eta$ lies in the interior of $\Upsilon$.

By the Zariski-Nagata purity theorem, the normalization $\widetilde{(\mathcal{C}_{\alpha, \beta})_s}$ of $(\mathcal{C}_{\alpha, \beta})_s$ can only be branched over $\SF{\mathcal{X}_{\alpha, \beta}}$ above the points which are reductions of points in $\mathcal{B}$.  Meanwhile, it follows from \Cref{prop relating nodes to Upsilon} that we have $z \in A$ if and only if the point $Q \in (\mathcal{C}_{\alpha, \beta})_s(k)$ lying over the point $P \in \SF{\mathcal{X}_{\alpha, \beta}}(k)$ to which $z$ reduces is a singular point with $p$ branches passing through it, \textit{i.e.} $P$ is not a branch point of the normalization $\widetilde{\SF{\mathcal{C}_{\alpha, \beta}}} \to \SF{\mathcal{X}_{\alpha, \beta}}$.  Therefore, we have $A = \mathcal{B}$ if and only if the $p$-cyclic cover $\widetilde{(\mathcal{C}_{\alpha, \beta})_s} \to \SF{\mathcal{X}_{\alpha, \beta}}$ is not branched anywhere, which in turn is equivalent to saying that $\widetilde{(\mathcal{C}_{\alpha, \beta})_s}$ (and equivalently, $(\mathcal{C}_{\alpha, \beta})_s$) has $p$ components which are each a copy of $\proj_k^1$.  This proves the second statement of the corollary.

To get the first statement, we apply the criterion given in \Cref{prop part of mini} as follows.  For each component $V$ of the special fiber $(\mathcal{C}_{\alpha, \beta})_s$, with $w(V)$ defined as in \S\ref{sec preliminaries geometric components}, we will show that we have $w(V) \geq 2$, which by \Cref{prop part of mini} guarantees that we have $\mathcal{X}_{\alpha, \beta} \leq \Xmini$.  Suppose that the fiber $(\mathcal{C}_{\alpha, \beta})_s$ has only a single component $V$.  Then as $A \neq \varnothing$, it has a singular point $Q$ with $p$ branches passing through it, which implies $w(V) \geq p \geq 2$.  Now suppose instead that the fiber $(\mathcal{C}_{\alpha, \beta})_s$ has $p$ components, and let $V \cong \proj_k^1$ be one of them.  As argued above, there is a point $z \in \mathcal{B} \smallsetminus \{\infty\}$ such that we have $\eta > \eta_z$.  Let $P_0$ (resp. $P_\infty$) be the $k$-point of $\SF{\mathcal{X}_{\alpha, \beta}}$ to which $z$ (resp. $\infty$) reduces, and note that we have $P_0 \neq P_\infty$.  As $P_0, P_\infty$ are reductions of branch points of the morphism $C \to X$, they are themselves branch points of the morphism $(\mathcal{C}_{\alpha, \beta})_s \to \SF{\mathcal{X}_{\alpha, \beta}}$ of special fibers.  Then the points $Q_0, Q_\infty \in V(k)$ which respectively lie over them intersect with the other components of $(\mathcal{C}_{\alpha, \beta})_s$ and so are singular points of this special fiber, again implying that $w(V) \geq 2$.
\end{proof}

The following lemma is more or less a converse to \Cref{cor relating nodes to Upsilon}.

\begin{lemma} \label{lemma relating nodes to Upsilon converse}

Let $\eta = \eta_{\alpha, v(\beta)} \in \Hyp$, and suppose that we have $\mathcal{X}_{\alpha, \beta} \leq \Xmini$ and that there is a singular $k$-point of $(\mathcal{C}_{\alpha, \beta})_s$ with $p$ branches passing through it.  Then we have $\eta \in \Upsilon_{K'}$.

\end{lemma}

\begin{proof}
Any model $\leq \Xmini$ is defined over $K'$, so by definition of $\Upsilon_{K'}$, it suffices to show that $\eta \in \Upsilon$.

The hypotheses of this lemma imply, by \Cref{cor crushed implies branch point}, that it cannot be the case that all points in $\mathcal{B}$ reduce to a single point in the special fiber $\SF{\mathcal{X}_{\alpha, \beta}}$; therefore, there are at least $2$ distinct $k$-points of $\SF{\mathcal{X}_{\alpha, \beta}}$ to which the points in disjoint subsets of $\mathcal{B}$ reduce.  Letting $P_\infty$ be the point to which $\infty \in \mathcal{B}$ reduces, there is another point $P_0$ to which a point $z \in \mathcal{B} \smallsetminus \{\infty\}$ reduces.  Clearly the point $P_\infty$ is given by $x_{\alpha, \beta} = \infty$; the points $x = w \in \mathcal{B}$ which reduce to this are those that satisfy $v(w - \alpha) < v(\beta)$.  Therefore, we must have $v(z - \alpha) \geq v(\beta)$, or equivalently, $z \in D_{\alpha, v(\beta)}$.  Then it follows from \Cref{prop highest point of Lambda}(b) that we have $\eta \in \Sigma_{\mathcal{B}}$.  We may now assume (without affecting the hypothesis on $\SF{\mathcal{C}_{\alpha, \beta}}$) that $\alpha \in D_{\alpha, v(\beta)} \cap \mathcal{B}$.

\textbf{Tame case:} Suppose that $k$ has characteristic different from $p$.  Let $\bar{\rho}_0$ be as in the discussion preceding \Cref{prop normalization}, so that it is a normalized reduction of $\rho$.  Write $\mathfrak{s} = D_{\alpha, v(\beta)} \cap \mathcal{B}$.  Our hypothesis on $\eta$ implies, thanks to \Cref{prop normalization properties}(b), that we have $p \mid \deg(\bar{\rho}_0)$ or $p \mid \ord(\bar{\rho}_0)$.  In the former case, \Cref{lemma degrees of terms appearing} implies that we have $p \mid \#\underline{\mathfrak{s}}$, and in the latter case, \Cref{lemma degrees of terms appearing} implies that we have $p \mid \#\underline{\mathfrak{c}}$, where $\mathfrak{c} \ni \alpha$ is a cluster whose parent coincides with $\mathfrak{s}$.  Now by \Cref{prop boundary points}, we have $\eta \in \Sigma_{\mathcal{B}}^0 = \Upsilon$.

\textbf{Wild case:} Suppose now that $k$ has characteristic $p$.  As the function $\mathfrak{t}_f$ is identically $0$ on $\Sigma_{\mathcal{B}}^*$ by \Cref{thm t_f formula one component}(a) combined with \Cref{prop mathfrakt minimum}, we have $\eta \in \Sigma_{\mathcal{B}} \smallsetminus \Sigma_{\mathcal{B}}^* \subsetneq \Sigma_{\mathcal{B}}^0$.  Now suppose that we have $\mathfrak{t}_f(\eta) < \pfrac$.  Then \Cref{prop normalization}(c) combined with \Cref{prop equations for normalization} implies that the special fiber $(\mathcal{C}_{\alpha, \beta})_s$ of the normalization of $\mathcal{X}_{\alpha, \beta}$ in $K(C)$ is an inseparable cover of $\SF{\mathcal{X}_{\alpha, \beta}} \cong \proj_k^1$.  Therefore, it is itself isomorphic to $\proj_k^1$ and thus cannot have a singular point, which contradicts our hypothesis.  It follows that we have $\mathfrak{t}_f(\eta) = \pfrac$, which is to say that $\eta \in \Upsilon$.
\end{proof}

\begin{lemma} \label{lemma non-maximal boundary points}

The toric rank $t(\SF{\Cmini})$ of (the special fiber of) the semistable model $\Cmini$ is equal to $p - 1$ times the number of non-maximal boundary points of connected components $U$ of $\Upsilon$ (with $U$ considered as a subspace of $\Sigma_{\mathcal{B}}$).

\end{lemma}

\begin{proof}
After possibly replacing $K'$ by a finite extension, we assume each non-isolated boundary point of $\Upsilon$ lies in $\Upsilon_{K'}$.  Now by \Cref{cor relating nodes to Upsilon} and \Cref{lemma relating nodes to Upsilon converse} combined, the points $\eta_{D_{\alpha, v(\beta)}} \in \Upsilon_{K'}$ are precisely the points $\eta_{D_{\alpha, v(\beta)}} \in \Hyp$ with $\mathcal{X}_{\alpha, \beta} \leq \Xmini$ and such that the special fiber $(\mathcal{C}_{\alpha, \beta})_s$ has a singular point with $p$ branches passing through it.  The latter condition is equivalent to saying that the strict transform of $\SF{\mathcal{C}_{\alpha, \beta}}$ in $\Cmini$ has $p$ points lying over a node in (the strict transform of $\SF{\mathcal{X}_{\alpha, \beta}}$ in) $\SF{\Xmini}$ (see the top of \S\ref{sec normalizations}).

Let us now put a graph structure on the discrete set $\Upsilon_{K'}$ by viewing its points as vertices and by defining an edge between a pair $\eta, \eta'$ of vertices to be a (necessarily unique) path $[\eta, \eta'] \subset \Upsilon$ with no other vertex of $\Upsilon_{K'}$ lying in its interior.  With this structure, each component of the graph $\Upsilon_{K'}$ is clearly a tree, and the leaves (\textit{i.e.} points with exactly $1$ edge) of $\Upsilon_{K'}$ are the boundary points of the subspace $\Upsilon \subset \Sigma_{\mathcal{B}}$.  It now follows from applying \Cref{prop models of projective line special fiber} that the nodes of $\SF{\Xmini}$ which are not branch points of $\Cmini \to \Xmini$ (or equivalently, whose inverse images in $\SF{\Cmini}$ consist of $p$ points) correspond to the edges of the graph $\Upsilon_{K'}$.  Therefore, in the notation of \Cref{prop toric rank formula}, the integer $N_0$ equals the number of edges of $\Upsilon_{K'}$.  Since $\Upsilon_{K'}$ is a disjoint union of trees, this equals the number of its vertices minus the number of its connected components.

Meanwhile, \Cref{cor relating nodes to Upsilon} implies that the number $N_1$ (in the notation of \Cref{prop toric rank formula}) of components of $\SF{\Xmini}$ whose inverse images in $\SF{\Cmini}$ consist of multiple components equals the number of non-leaf vertices of $\Upsilon_{K'}$.  The quantity $N_0 - N_1$ therefore equals the number of boundary points of $\Upsilon$ minus the number of connected components of $\Upsilon$.  Each connected component $U$ of $\Upsilon$ has a unique maximal point $\xi$ which is clearly a boundary point (as we have $\xi' \subset (\xi, \eta_\infty) \subset \Sigma_{\mathcal{B}}$ and $(\xi, \xi') \cap U = \varnothing$ for some $\xi' > \xi$ close enough to $\xi$); therefore, we get that $N_0 - N_1$ equals the number of non-maximal boundary points of connected components of $\Upsilon$.  By \Cref{prop toric rank formula}, the desired formula follows.
\end{proof}

\begin{lemma} \label{lemma constructing Phi}

Suppose that the residue characteristic of $K$ is $p$.  Let the partition $\mathcal{B} = \mathfrak{r}_0 \sqcup \dots \sqcup \mathfrak{r}_h$ and the associated clusters $\mathfrak{s}_1, \dots \mathfrak{s}_h$ be as in \Cref{prop connected components Sigma}.  Let $\eta \in \Upsilon$ be a point such that there exists a point $\eta' < \eta$ with $(\eta, \eta'] \subset \Sigma_{\mathcal{B}}^0 \smallsetminus \Upsilon$ (for instance, we may choose $\eta \in \Upsilon$ to be a non-maximal boundary point).

There is an index $i \in \{1, \dots, h\}$ such that we have $\eta > \eta_{\mathfrak{s}_i}$ and such that along the path from $\eta$ to $\eta_{\mathfrak{s}_i}$, the function $t_f$ is (strictly) decreasing from $\ptfrac$ to $0$.

\end{lemma}

\begin{proof}

Let $D \subset \cc_K$ be the disc corresponding to the non-maximal boundary point $\eta = \eta_D \in \Upsilon \subset \Sigma_{\mathcal{B}}^0$, and write $\mathfrak{s} = D \cap \mathcal{B}$.  We prove the lemma by constructing a sequence of inclusions of clusters $\mathfrak{s} \supseteq \mathfrak{s}^{(1)} \supsetneq \dots \supsetneq \mathfrak{s}^{(m)} = \mathfrak{s}_i$ for some index $i \in \{1, \dots, h\}$, where the parent of each cluster $\mathfrak{s}^{(j)}$ is $\mathfrak{s}^{(j-1)}$ and the function $\mathfrak{t}_f$ is decreasing on each of the paths $[\eta, \eta_{\mathfrak{s}^{(1)}}], [\eta_{\mathfrak{s}^{(j-1)}}, \eta_{\mathfrak{s}^{(j)}}]$ for $2 \leq j \leq m$; its ending value of $0$ is given by \Cref{thm t_f formula one component}(a).

As we are in the wild case (where $\ptfrac \neq 0$) and have $\eta \in \Upsilon$, \Cref{thm t_f formula one component}(a) combined with \Cref{prop mathfrakt minimum} implies that $\eta \notin \Sigma_{\mathcal{B}}^*$, so that $\eta$ is not a boundary point of $\Sigma_{\mathcal{B}}^0$ by definition of $\Sigma_{\mathcal{B}}^0$.  Write $\mathfrak{s} = \mathfrak{c}_1 \sqcup \dots \sqcup \mathfrak{c}_s$ for the partition into clusters corresponding to $\eta$ given by \Cref{prop removing a point from convex hull}.  Now one sees from \Cref{prop boundary points}(a)(b)(c) that we have $p \mid \#\underline{\mathfrak{c}}_1, \dots, \#\underline{\mathfrak{c}}_r$; that proposition then also implies that we have $[\eta, \eta_{\mathfrak{c}_j}] \subset \Sigma_{\mathcal{B}}^0$ for all $j$.  Let $\eta'$ be as in the statement of the lemma, so that its corresponding disc contains a root lying in $\mathfrak{s}$ (by \Cref{prop highest point of Lambda}(b)) and thus lying in one of the clusters $\mathfrak{c}_j$.  If we choose $\eta'$ close enough to $\eta$, we then get $[\eta, \eta'] \subseteq [\eta, \eta_{\mathfrak{c}_j}]$.  After possibly choosing $\eta'$ even closer to $\eta$, the function $\mathfrak{t}_f$ takes values $< \ptfrac$ on the open path $(\eta, \eta')$.  Now it follows from \Cref{prop mathfrakt minimum} applied to the disjoint union $\mathcal{B} = \mathfrak{c}_1 \sqcup \dots \sqcup \mathfrak{c}_r \sqcup (\mathcal{B} \smallsetminus \mathfrak{s})$ that one of the functions $\mathfrak{t}_{\mathfrak{c}_j}$ takes values $< \ptfrac$ on $(\eta, \eta')$.

For $1 \leq j \leq r$, choose a root $z_j \in \mathfrak{c}_j$.  Now given any $\varepsilon > 0$, applying \Cref{lemma mathfrakv}(a) tells us that the functions $\mathfrak{t}_{\mathfrak{c}_j}$ are each continuous piecewise linear on the respective paths $[\eta = \eta_{D_{z_j, c}}, \eta'_j := \eta_{D_{z_j, c + \varepsilon}}]$ (where $c$ is the logarithmic radius of the disc corresponding to $\eta$).  Fix an index $s$ such that $\mathfrak{t}_{\mathfrak{c}_j}(\eta)$ is minimized and such that, among the indices $j$ with that value of $\mathfrak{t}_{\mathfrak{c}_j}(\eta)$, the right derivative of the piecewise linear function $b \mapsto \mathfrak{t}_{\mathfrak{c}_j}(\eta_{D_{z_j, b}})$ is minimized.  The conclusion of the last paragraph implies that, if we choose $\varepsilon$ small enough, the function $\mathfrak{t}_{\mathfrak{c}_s}$ takes values $< \ptfrac$ on $[\eta, \eta'_s]$.  Moreover, our choice of $s$ implies that we have $\mathfrak{t}_{\mathfrak{c}_s}(\eta_{D_{z_s, b}}) \leq \mathfrak{t}_{\mathfrak{c}_j}(\eta_{D_{z_j, b}})$ for all $b \in (c, c + \varepsilon)$ and each index $j \neq s$.

Note that $\underline{\mathfrak{c}}_1, \dots, \underline{\mathfrak{c}}_r, \underline{\mathcal{B} \smallsetminus \mathfrak{s}}$, having cardinalities divisible by $p$, are each the disjoint union of some of the $\mathfrak{r}_i$'s thanks to \Cref{cor cardinality-p clusters}, allowing us to apply \Cref{thm t_f formula one component}(b)(c).  Now \Cref{thm t_f formula one component}(b) implies that the function $\mathfrak{t}_{\mathfrak{c}_s}$ is decreasing along the open path from $\eta$ to $\eta_{\mathfrak{c}_s}$, and meanwhile, for any $j \neq s$, we may apply \Cref{thm t_f formula one component}(c) (noting that points in $(\eta, \eta_{\mathfrak{c}_s})$ correspond to discs centered at $z_s \notin \mathfrak{c}_j$ which do not intersect $\mathfrak{c}_j$) to get that the function $\mathfrak{t}_{\mathfrak{c}_j}$ is always increasing along that same path until it reaches $\ptfrac$ and becomes constant; since we have $\mathfrak{t}_{\mathfrak{c}_s}(\eta) \leq \mathfrak{t}_{\mathfrak{c}_j}(\eta)$, this implies the inequality $\mathfrak{t}_{\mathfrak{c}_s} < \mathfrak{t}_{\mathfrak{c}_j}$ over the open path $(\eta, \eta_{\mathfrak{c}_s})$.  In fact, we also have $\mathfrak{t}_{\mathfrak{c}_s}(\eta_{D_{z_s, b}}) < \mathfrak{t}_{\mathcal{B} \smallsetminus \mathfrak{s}}(\eta_{D_{z_s, b}})$ for all $b \in (c, c + \varepsilon)$ if $\varepsilon$ is chosen small enough.  Indeed to assume otherwise would imply $\mathfrak{t}_{\mathcal{B} \smallsetminus \mathfrak{s}}(\eta_{D_{z_s, b}}) < \mathfrak{t}_{\mathfrak{c}_s}(\eta_{D_{z_s, b}})$ for $b \in (c, c + \varepsilon)$; then similar application of \Cref{thm t_f formula one component}(b)(c) show that $\mathfrak{t}_{\mathcal{B} \smallsetminus \mathfrak{s}}$ takes strictly lesser values than each $\mathfrak{t}_{\mathfrak{c}_j}$ over the path $(\eta'_0, \eta) \subset \Sigma_{\mathcal{B}}^0$ for a point $\eta'_0 > \eta$ close enough to $\eta$ and that these values are $< \ptfrac$, and thus so does $\mathfrak{t}_f$ on that path thanks to \Cref{prop mathfrakt minimum}, contradicting the fact that $\eta$ is non-maximal as a boundary point of $\Upsilon$.  We may now apply \Cref{prop mathfrakt minimum} to get that $\mathfrak{t}_f = \mathfrak{t}_{\mathfrak{c}_s}$ on $(\eta, \eta_{\mathfrak{c}_s})$ (and, by continuity given by \Cref{cor t_f is continuous}, on its closure), and so in particular, the function $\mathfrak{t}_f$ is decreasing on the path from $\eta$ to $\eta_{\mathfrak{c}_s}$.

If we have $\mathfrak{s}^{(1)} := \mathfrak{c}_s = \mathfrak{s}_i$ for some $i$, then we are done.  Otherwise, by \Cref{prop highest point of Lambda}, the point $\eta_{\mathfrak{s}^{(1)}}$ is not a maximal boundary point of $\Sigma_{\mathcal{B}}^*$.  As the points in the open path $(\eta, \eta_{\mathfrak{s}^{(1)}})$ all correspond to discs whose intersection with $\mathcal{B}$ coincides with $\mathfrak{s}^{(1)}$, an application of \Cref{prop boundary points}(a)(b)(c) implies $(\eta, \eta_{\mathfrak{s}^{(1)}}) \subset \Sigma_{\mathcal{B}}^0$, which means that $\eta_{\mathfrak{s}^{(1)}}$ lies in the interior of $\Sigma_{\mathcal{B}}^0$.  Now we may repeat the same argument as was used for $\mathfrak{s}$ in the above paragraphs to get a cluster $\mathfrak{s}^{(2)}$ whose parent is $\mathfrak{s}^{(1)}$ such that the function $\mathfrak{t}_f$ is decreasing on the path from $\eta_{\mathfrak{s}^{(1)}}$ to $\eta_{\mathfrak{s}^{(2)}}$.  Repeating this process, we get the desired sequence of clusters $\mathfrak{s} \supseteq \mathfrak{s}^{(1)} \supsetneq \mathfrak{s}^{(2)} \supsetneq \dots$, which must terminate at $\mathfrak{s}_i$ for some $i$.
\end{proof}

\begin{prop} \label{prop bound on toric rank}

If the connected components of the space $\Sigma_{\mathcal{B}}^*$ are written as $\Sigma_{\mathfrak{r}_0}, \dots, \Sigma_{\mathfrak{r}_h}$ as in \Cref{prop connected components Sigma}, we have the inequality $t(\SF{\Cmini}) \leq (p - 1)h$, and equality holds if the residue characteristic of $K$ is not $p$.

\end{prop}

\begin{proof}
We prove this by furnishing a one-to-one function $\Phi$, which is also onto in the tame case, assigning to each boundary point of $\Upsilon$ an index $i \in \{1, \dots, h\}$, the existence of which implies that the number of non-maximal boundary points of $\Upsilon$ is $\leq h$ with equality in the tame case; \Cref{lemma non-maximal boundary points} then implies the assertion of the corollary.

\textbf{Tame case:} Suppose that the residue characteristic of $K$ is different from $p$.  Then the non-maximal boundary points of $\Upsilon = \Sigma_{\mathcal{B}}^0$ are precisely the maximal points of the connected components $\Sigma_{\mathfrak{r}_i}$ of $\Sigma_{\mathcal{B}}^*$ for $1 \leq i \leq h$ by \Cref{prop boundary points}(a)(c)(d); we then define the desired bijection $\Phi$ as assigning to each non-maximal boundary point of $\Upsilon$ the corresponding index $i$.

\textbf{Wild case:} Suppose now that the residue characteristic of $K$ is $p$.  Let $\eta \in \Upsilon$ be a non-maximal boundary point.  Then \Cref{lemma constructing Phi} provides an index $i \neq 0$ such that we have $\eta > \eta_{\mathfrak{s}_i}$ (where $\mathfrak{s}_i$ is the smallest cluster containing $\mathfrak{r}_i$) and such that $\mathfrak{t}_f$ decreases along the path $[\eta, \eta_{\mathfrak{s}_i}]$ from $\ptfrac$ to $0$.  We define our desired function $\Phi$ by letting $\Phi(\eta)$ be this index $i$.  Now if there are non-maximal boundary points $\eta, \eta' \in \Upsilon$ with $\Phi(\eta) = \Phi(\eta') =: i$, then we have $\eta, \eta' > \eta_{\mathfrak{s}_i}$ (so that $\eta \geq \eta'$ or $\eta' \geq \eta$) and that $\mathfrak{t}_f$ begins at $\pfrac$ and (strictly) decreases to $0$ on both paths $[\eta, \eta_{\mathfrak{s}_i}]$ and $[\eta', \eta_{\mathfrak{s}_i}]$.  This implies $\eta' = \eta$, so that the function $\Phi$ is one-to-one, as desired.
\end{proof}

\begin{rmk} \label{rmk dokchitsers toric rank}

In the case that $p = 2$ and the residue characteristic of $K$ is not $2$, \Cref{prop bound on toric rank}, combined with \Cref{prop cardinality-p clusters}, recovers the result of Dokchitser-Dokchitser-Maistret-Morgan (given as \cite[Theorem 1.8(7)]{dokchitser2022arithmetic}) which says in this case that the potential toric rank equals the number of even-cardinality clusters which are not themselves a disjoint union of $\geq 2$ even-cardinality clusters.

\end{rmk}

\begin{proof}[Proof (of \Cref{thm main})]
In light of \Cref{prop properties (i) and (ii) of main theorem}, our task is to show that the curve $C$ is potentially degenerate if and only if the branch multiset $\underline{\mathcal{B}}$ is clustered in $\ptfrac$-separated pairs.

Recall that the genus of $C$ is equal to $\frac{1}{2}(p - 1)(\#\mathcal{B} - 2)$.  Then if $C$ is potentially degenerate, we have (as in \Cref{rmk alternate definition of split degenerate}) the equation 
\begin{equation} \label{eq toric rank equals genus}
t(\SF{\Cmini}) = \tfrac{1}{2}(p - 1)(\#\mathcal{B} - 2) = (p - 1)(\#\mathcal{B}/2 - 1).
\end{equation}
Meanwhile, letting $\mathcal{B} = \mathfrak{r}_0 \sqcup \dots \sqcup \mathfrak{r}_h$ be the partition given by \Cref{prop connected components Sigma}, as we have $p \mid \#\underline{\mathfrak{r}}_i$ for $0 \leq i \leq h$ but each element of $\mathcal{B}$ appears with multiplicity $\leq p - 1$, it must be the case that $\#\mathfrak{r}_i \geq 2$ for $0 \leq i \leq h$.  Therefore, we have $h + 1 \leq \#\mathcal{B}/2$, with equality if and only if each set $\mathfrak{r}_i$ has cardinality $2$.  Combining this with (\ref{eq toric rank equals genus}) and with \Cref{prop bound on toric rank}, we get the inequalities 
\begin{equation} \label{eq squeezed between inequalities}
(p - 1)(\#\mathcal{B}/2 - 1) = t(\SF{\Cmini}) \leq (p - 1)h \leq (p - 1)(\#\mathcal{B}/2 - 1).
\end{equation}
The above inequalities are forced to be equalities, implying that we have $h + 1 = \#\mathcal{B}/2$, which means that we have $\#\mathfrak{r}_0 = \dots = \#\mathfrak{r}_i = 2$.  Now as multiplicities of elements in $\mathcal{B}$ are $\leq p - 1$, this forces $\#\underline{\mathfrak{r}}_0 = \dots = \#\underline{\mathfrak{r}}_h = p$.  Thus, we have $h + 1 = \#\underline{\mathcal{B}} / p$, which according to \Cref{dfn clustered in pairs berk}, means that $\underline{\mathcal{B}}$ is clustered in $p$-multisets and in fact clustered in pairs as each $\mathfrak{r}_i$ has cardinality $2$ as a set.

We have thus shown that being potentially degenerate implies the clustered in pairs property; we now assume for the moment that $K$ has residue characteristic not $p$ and show under this condition that the converse holds.  \Cref{prop bound on toric rank} tells us that we have $t(\SF{\Cmini}) = (p - 1)h$, while the hypothesis that $\underline{\mathcal{B}}$ is clustered in pairs implies that $h + 1 = \#\mathcal{B}/2$.  Putting this together, we get 
\begin{equation}
t(\SF{\Cmini}) = (p - 1)h = (p - 1)(\#\mathcal{B}/2 - 1).
\end{equation}
As the expression on the right is the formula for the genus, the curve $C$ is potentially degenerate.  This completes the proof of \Cref{thm main} in the tame case.

Now suppose that we are in the wild case, \textit{i.e.} that the residue characteristic of $K$ is $p$.  Assume that $\underline{\mathcal{B}}$ is clustered in pairs (which we have shown to be implied by potential degneracy); this gives us $h + 1 = \#\mathcal{B}/2$.  For $1 \leq i \leq h$, let $\xi_i = \eta_{\mathfrak{s}_i}$ (where $\mathfrak{s}_i$ is the smallest cluster containing $\mathfrak{r}_i$) be the maximal point of the connected component $\Sigma_{\mathfrak{r}_i} \subset \Sigma_{\mathcal{B}}^*$ (see \Cref{prop highest point of Lambda}(a)), and let $\hat{\xi}_i$ be the closest point to $\xi_i$ satisfying $\hat{\xi}_i \geq \xi_i$ and $\mathfrak{t}_f(\hat{\xi}_i) = \ptfrac$, if such a point exists.

From \Cref{prop bound on toric rank}, we get 
\begin{equation}
t(\SF{\Cmini}) \leq (p - 1)h = (p - 1)(\#\mathcal{B}/2 - 1).
\end{equation}
As the expression on the right is the formula for the genus, to complete the proof of \Cref{thm main} in the wild case, our task is to show that we have $t(\SF{\Cmini}) = (p - 1)h$ if and only if we have the $\pfrac$-separateness property (\textit{i.e.} $B(\Sigma_{\mathfrak{r}_i}, \pfrac) \cap B(\Sigma_{\mathfrak{r}_j}, \pfrac) = \varnothing$ for $i \neq j$).  As in the proof of \Cref{prop bound on toric rank}, the equality $t(\SF{\Cmini}) = (p - 1)h$ is equivalent to the one-to-one function $\Phi$ also being onto, where $\Phi$ takes a non-maximal boundary point $\eta \in \Upsilon$ to the index $i$ such that we have $\eta > \xi_i$ and that $\mathfrak{t}_f$ takes values beginning at $\pfrac$ and decreasing to $0$ on $[\eta, \xi_i]$.  It is clear from definitions that these non-maximal boundary points are precisely those points $\hat{\xi}_i \in \Upsilon$ which are defined and which are not maximal in their connected component of $\Upsilon$, and that $i \mapsto \hat{\xi}_i$ is a one-sided inverse of $\Phi$.  Thus, the property of the one-to-one function $\Phi$ being a bijection is equivalent to saying that the one-sided inverse $i \mapsto \hat{\xi}_i$ is defined for all indices $i \in \{1, \dots, h\}$, takes each $i$ to a non-maximal point in a connected component of $\Upsilon$, and is one-to-one (\textit{i.e.} that we have $\hat{\xi}_j \neq \hat{\xi}_l$ for indices $j \neq l$).

Suppose first that $\underline{\mathcal{B}}$ is clustered in $\pfrac$-separated pairs.  In this case, \Cref{cor t_f formula} provides us with a precise formula for the function $\mathfrak{t}_f$ on $\Hyp$.  As the condition of being clustered in $\ptfrac$-separated pairs tells us that only one of the distances $\delta(\eta, \Sigma_{\mathfrak{r}_i})$ can be $< \ptfrac$ at a given $\eta \in \Hyp$, our formula for $\mathfrak{t}_f$ implies that for $1 \leq i \leq h$, the point $\hat{\xi}_i$ is defined and is the (unique) point $> \xi_i$ satisfying $\delta(\hat{\xi}_i, \xi_i) = \ptfrac$.  For distinct indices $j, l \in \{1, \dots, h\}$, we then have $\hat{\xi}_j \in B(\Sigma_{\mathfrak{r}_j}, \ptfrac)$, which implies $\hat{\xi}_j \notin B(\Sigma_{\mathfrak{r}_l}, \ptfrac) \ni \hat{\xi}_l$.  It follows that $\hat{\xi}_j \neq \hat{\xi}_l$; moreover, for $\eta > \hat{\xi}_j$ close enough to $\hat{\xi}_j$, we have $\delta(\eta, \Sigma_{\mathfrak{r}_i}) > \ptfrac$ for $0 \leq i \leq h$ and thus $\mathfrak{t}_f(\eta) = \ptfrac$ and $\eta \in \Upsilon$, so that $\hat{\xi}_j$ is not a maximal boundary point of its connected component of $\Upsilon$.  As the indices $j \neq l$ were arbitrarily chosen, we have thus proved the desired properties of the map $i \mapsto \hat{\xi}_i$.

Now suppose instead that $\underline{\mathcal{B}}$ is not clustered in $\pfrac$-separated pairs.  We first assume that there exists a point $\xi \in \Sigma_{\mathcal{B}}$ such that $\mathfrak{t}_f(\xi) > \delta(\xi, \Sigma_{\mathfrak{r}_i})$ for some index $i$ and proceed to find either indices $j \neq l$ such that $\hat{\xi}_j = \hat{\xi}_l$ or an index $l \neq 0$ with $\hat{\xi}_l$ undefined.  For $0 \leq j \leq h$, write $\nu_j$ for the closest point in $\Sigma_{\mathfrak{r}_j}$ to $\xi$, so that for some index $i$, we have $\mathfrak{t}_f(\xi) > \delta(\xi, \nu_i) = \mathfrak{t}_{\mathfrak{r}_i'}(\xi) < \ptfrac$, where the equality comes from \Cref{thm t_f formula one component}(d).  \Cref{prop mathfrakt minimum} now implies that the minimum of the values $\mathfrak{t}_{\mathfrak{r}_j'}(\xi)$ for $0 \leq j \leq h$, which by \Cref{thm t_f formula one component}(d) are each equal to $\min\{\delta(\xi, \Sigma_{\mathfrak{r}_j}), \ptfrac\}$, is achieved by distinct indices $j, l$ and is $< \ptfrac$.  Writing $\nu_{j, l}$ for the midpoint of the path $[\nu_j, \nu_l]$, the fact that $\nu_j$ and $\nu_l$ are equidistant from $\xi$ implies $\nu_{j, l} \in [\nu_j, \xi] \cap [\nu_l, \xi]$.  We thus get 
\begin{equation} \label{eq inequality with points nu}
\tfrac{1}{2}\delta(\nu_j, \nu_l) \leq \tfrac{1}{2}(\delta(\nu_j, \xi) + \delta(\nu_l, \xi)) = \delta(\nu_j, \xi) = \delta(\nu_l, \xi) < \ptfrac.
\end{equation}
After possibly replacing $j, l$ by another pair of distinct indices (and defining $\nu_{j, l}$ the same way), we may assume that the distance $\delta(\nu_j, \nu_l)$ is minimized among choices of distinct indices $j, l \in \{0, \dots, h\}$ and observe that (\ref{eq inequality with points nu}) still holds.

Assume that there exists a point $\eta \in (\nu_j, \nu_{j, l})$ and an index $m \neq j$ with $\delta(\eta, \nu_m) < \delta(\nu_j, \nu_{j, l})$.  Then we have $\delta(\nu_j, \nu_m) \leq \delta(\nu_j, \eta) + \delta(\eta, \nu_m) < 2\delta(\nu_j, \nu_{j, l}) = \delta(\nu_j, \nu_l)$, which contradicts our minimality condition on $j, l$.  Therefore, we have $\delta(\eta, \nu_m) > \delta(\eta, \nu_j)$ for all indices $m \neq j$, and symmetrically, we have $\delta(\eta, \nu_m) > \delta(\eta, \nu_l)$ for all $m \neq l$.  Applying \Cref{thm t_f formula one component}(d) and \Cref{prop mathfrakt minimum} now implies $\mathfrak{t}_f(\eta) = \min\{\delta(\eta, \nu_j), \delta(\eta, \nu_l)\}$ for $\eta \in (\nu_j, \nu_{j, l}) \cup (\nu_l, \nu_{j, l})$.  By continuity of $\mathfrak{t}_f$ given by \Cref{cor t_f is continuous} and using the inequality in (\ref{eq inequality with points nu}), for $\eta \in [\nu_j, \nu_l]$ we have
\begin{equation} \label{eq t_f small on [nu_j, nu_l]}
\mathfrak{t}_f(\eta) = \min\{\delta(\eta, \nu_j), \delta(\eta, \nu_l)\} \leq \tfrac{1}{2}\delta(\nu_j, \nu_l) < \ptfrac.
\end{equation}

Assume that $l \neq 0$, so that the point $\xi_l$ is defined.  If $j \neq 0$ and $\hat{\xi}_j$ is not defined, then we are done, so let us assume that $j \neq 0$ and $\hat{\xi}_j$ is defined (resp. let us assume that $j = 0$).  As $\xi_j$ is the maximal point of $\Sigma_{\mathfrak{r}_j}$, it is the closest point in $\Sigma_{\mathfrak{r}_j}$ to $\hat{\xi}_j > \xi_j$, so we get $\hat{\xi}_j > \xi_j \geq \nu_j$ (resp. we have $(\nu_0, \eta_\infty) \subset \Sigma_{\mathfrak{r}_0}$).  
Note that the function $\mathfrak{t}_f$ is identically $0$ on $[\nu_j, \xi_j] \subset \Sigma_{\mathfrak{r}_j}$ (resp. on $[\nu_0, \eta_\infty) \subset \Sigma_{\mathfrak{r}_0}$) and on $[\nu_l, \xi_l] \subset \Sigma_{\mathfrak{r}_l}$ thanks to \Cref{cor t_f formula}.  In particular, this implies that $\mathfrak{t}_f$ takes values $< \ptfrac$ on $[\nu_j, \hat{\xi}_j]$ except at the endpoint $\hat{\xi}_j$ (resp. on $[\nu_0, \eta_\infty)$).  Meanwhile, the inequality in (\ref{eq t_f small on [nu_j, nu_l]}) implies that $\mathfrak{t}_f$ takes values $< \ptfrac$ on the paths $[\nu_j, \nu_j \vee \nu_l], [\nu_l, \nu_j \vee \nu_l] \subseteq [\nu_j, \nu_l]$, and so we get $\hat{\xi}_j > \nu_j \vee \nu_l \geq \nu_j, \nu_l$ if $j \neq 0$.  Now since $\mathfrak{t}_f$ takes values $< \ptfrac$ on the path $[\nu_l, \xi_l]$ as well as on the paths $[\nu_l, \nu_j \vee \nu_l]$ and $[\nu_j \vee \nu_l, \hat{\xi}_j] \subseteq [\nu_j, \hat{\xi}_j]$ except at the endpoint $\hat{\xi}_j$ (resp. on the paths $[\nu_l, \nu_j \vee \nu_l]$ and $[\nu_j \vee \nu_l, \eta_\infty)$), we must have $\xi_l \in [\nu_l, \hat{\xi}_j)$ (resp. $\xi_l \in [\nu_l, \eta_\infty)$) and $\hat{\xi}_j = \hat{\xi}_l$ (resp. that $\hat{\xi}_l$ is not defined).  We have thus proved that the function $i \mapsto \hat{\xi}_i$ is not one-to-one (resp. is not defined for all indices in $\{1, \dots, h\}$) if $j \neq 0$ and $\hat{\xi}_j$ is defined (resp. if $j = 0$).

Having now dealt with the case where there exists a point $\xi \in \Sigma_{\mathcal{B}}$ and an index $i$ such that $\mathfrak{t}_f(\xi) > \delta(\xi, \Sigma_{\mathfrak{r}_i})$, for the rest of the proof in the case that $\underline{\mathcal{B}}$ is not clustered in $\ptfrac$-separated pairs, we assume that such a $\xi$ does not exist, or in other words, that we have $\mathfrak{t}_f(\xi) \leq \delta(\xi, \Sigma_{\mathfrak{r}_i})$ for $0 \leq i \leq h$ and all $\eta \in \Sigma_{\mathcal{B}}$.


For $1 \leq i \leq h$, let $\tilde{\xi}_i$ denote the (unique) point $> \xi_i$ satisfying $\delta(\tilde{\xi}_i, \xi_i) = \ptfrac$.  Note that as long as $\hat{\xi}_i$ is defined, we have $\hat{\xi}_i \geq \tilde{\xi}_i$, as implied by the fact that $\mathfrak{t}_f(\eta) \leq \delta(\eta, \xi_i) \leq \ptfrac$ for $\eta \in [\xi_i, \tilde{\xi}_i]$ with equality only possible at $\eta = \tilde{\xi}_i$.  We observe that for each $i$, the point $\tilde{\xi}_i$ is the (unique) maximal point of $B(\Sigma_{\mathfrak{r}_i}, \pfrac)$: indeed, for any point $\eta \in B(\Sigma_{\mathfrak{r}_i}, \pfrac)$, we have $\eta \vee \xi_i \in [\xi_i, \tilde{\xi}_i]$ (we have $\eta \vee \xi_i  \geq \xi_i$ and $\eta \vee \xi_i \in [\eta, \xi_i]$, a path in the connected space $B(\Sigma_{\mathfrak{r}_i}, \pfrac)$).

Now as $\underline{\mathcal{B}}$ is not clustered in $\pfrac$-separated pairs, for some indices $j \neq l$ there exists a point $\eta \in B(\Sigma_{\mathfrak{r}_j}, \pfrac) \cap B(\Sigma_{\mathfrak{r}_l}, \pfrac)$.  Assume that $l \neq 0$, so that the points $\xi_l, \tilde{\xi}_l$ are defined.  Note that each point of the half-open path $[\xi_i, \tilde{\xi}_i)$ has distance $< \ptfrac$ from $\Sigma_{\mathfrak{r}_i}$, implying that the function $\mathfrak{t}_f$ takes values $< \ptfrac$ on $[\xi_i, \tilde{\xi}_i)$.

Assume that we are in the $j \neq 0$ case.  Since $\tilde{\xi}_j$ and $\tilde{\xi}_l$ are the respective maximal points of $B(\Sigma_{\mathfrak{r}_j}, \pfrac)$ and $B(\Sigma_{\mathfrak{r}_l}, \pfrac)$, we have $\tilde{\xi}_j, \tilde{\xi}_l \geq \eta$, from which it follows that $\tilde{\xi}_j \in [\eta, \tilde{\xi}_l]$ or $\tilde{\xi}_l \in [\eta, \tilde{\xi}_j]$; without loss of generality, we assume the latter.  Then as $[\eta, \tilde{\xi}_j]$ is a path connecting points which both lie in the connected space $B(\Sigma_{\mathfrak{r}_j}, \ptfrac)$, we get $[\tilde{\xi}_l, \tilde{\xi}_j] \subseteq [\eta, \tilde{\xi}_j] \subset B(\Sigma_{\mathfrak{r}_j}, \ptfrac)$.  Suppose that there is some interior point $\xi \in (\tilde{\xi}_l, \tilde{\xi}_j)$ with $\delta(\xi, \Sigma_{\mathfrak{r}_j}) = \ptfrac$, and let $\xi'$ be the closest point in $[\tilde{\xi}_l, \tilde{\xi}_j]$ to $\Sigma_{\mathfrak{r}_j}$.  If we have $\xi' \in [\tilde{\xi}_l, \xi]$ (resp. $\xi' \in [\xi, \tilde{\xi}_j]$), then we have $\delta(\tilde{\xi}_j, \Sigma_{\mathfrak{r}_j}) > \ptfrac$ (resp. $\delta(\tilde{\xi}_l, \Sigma_{\mathfrak{r}_j}) > \ptfrac$), a contradiction.  Therefore, every point in the open path $(\tilde{\eta}_l, \tilde{\xi}_j)$ has distance $< \ptfrac$ from $\Sigma_{\mathfrak{r}_j}$; it follows that the function $\mathfrak{t}_f$ takes values $< \ptfrac$ on $(\tilde{\eta}_l, \tilde{\eta}_j)$.  Therefore, if we have $\mathfrak{t}_f(\tilde{\eta}_l) = \ptfrac$, then we get $\tilde{\eta}_l = \hat{\eta}_l \in \Upsilon$ and moreover that $\hat{\eta}_l$ is a maximal boundary point of $\Upsilon$.  If instead we have $\mathfrak{t}_f(\tilde{\eta}_l) < \ptfrac$, then the function $\mathfrak{t}_f$ takes values $< \ptfrac$ on $[\xi_l, \tilde{\xi}_l) \cup [\tilde{\xi}_l, \tilde{\xi}_j) \cup [\tilde{\xi}_j, \hat{\xi}_j] = [\xi_l, \hat{\xi}_j]$ except at the endpoint $\hat{\xi}_j$, proving that $\hat{\xi}_l = \hat{\xi}_j$ and the map $i \mapsto \hat{\xi}_i$ is not one-to-one.

In the $j = 0$ case, the argument is slightly simpler.  As the half-open path $[\eta, \eta_\infty)$ has its endpoint in $B(\Sigma_{\mathfrak{r}_0}, \ptfrac)$ and its (unique) limit point a limit point of the connected space $\Sigma_{\mathfrak{r}_0}$, the whole path is contained in $B(\Sigma_{\mathfrak{r}_0}, \ptfrac)$.  We have $\tilde{\xi}_l \geq \eta$ so that $\tilde{\xi}_l \in [\eta, \eta_\infty) \subset B(\Sigma_{\mathfrak{r}_0}, \ptfrac)$.  As the subspace $\Sigma_{\mathfrak{r}_0} \subset \Sigma_{\mathcal{B}}$ is an upwards path having $\eta_\infty$ as a limit point, the closest point $\nu_0$ in $\Sigma_{\mathfrak{r}_0}$ to $\tilde{\xi}_l$ must satisfy $\nu_0 \geq \tilde{\xi}_l$.  We get either $[\nu_0 = \tilde{\xi}_l, \eta_\infty) \subset \Sigma_{\mathfrak{r}_0}$ or $\delta(\xi, \Sigma_{\mathfrak{r}_0}) < \delta(\tilde{\xi}_l, \Sigma_{\mathfrak{r}_0}) \leq \ptfrac$ for any interior point $\xi \in (\tilde{\xi}_l, \eta_\infty)$.  In either case, we see that every point in open path $(\hat{\xi}_l, \eta_\infty)$ has distance $< \ptfrac$ from $\Sigma_{\mathfrak{r}_0}$; it follows that the function $\mathfrak{t}_f$ takes values $< \ptfrac$ on $(\tilde{\xi}_l, \eta_\infty)$.  Therefore, if we have $\mathfrak{t}_f(\tilde{\xi}_l) = \ptfrac$, as in the $j \neq 0$ case we get that $\tilde{\xi}_l = \hat{\xi}_l$ is a maximal boundary point of $\Upsilon$.  If instead we have $\mathfrak{t}_f(\tilde{\xi}_l) < \ptfrac$, then the function $\mathfrak{t}_f$ takes values $< \ptfrac$ on the whole path $[\xi_l, \tilde{\xi}_l) \cup [\tilde{\xi}_l, \eta_\infty) = [\xi_l, \eta_\infty)$, so that $\hat{\xi}_l$ is not defined.
\end{proof}

\begin{rmk} \label{rmk no non-isolated boundary points}

If $C$ is potentially degenerate, then there are no isolated (boundary) points of $\Upsilon \subset \Sigma_{\mathcal{B}}$.  In the tame case, this follows from the fact that $\Upsilon = \Sigma_{\mathcal{B}}^0$ is the complement of a subspace of $\Sigma_{\mathcal{B}}$ whose components are closed.  In the wild case, note that the definition of the function $\Phi$ in the proof of \Cref{prop bound on toric rank}, as well as the argument that it is one-to-one, applies not only to the set of non-maximal boundary points of $\Upsilon$ (the domain given for $\Phi$) but to the set of points $\eta \in \Upsilon$ satisfying the hypothesis of \Cref{lemma constructing Phi}.  Then one sees from the proof of \Cref{thm main} that the arguments for the potentially degenerate case which demonstrate a bijection between the domain of $\Phi$ and $\{1, \dots, h\}$, remain valid regardless of whether we choose to include isolated points of $\Upsilon$ in the domain of $\Phi$, thus implying their non-existence.

\end{rmk}

\section{Building split degenerate models} \label{sec building models}

In this section we show how to actually construct the minimal regular model $\Cmini$ of a potentially degenerate $p$-cyclic cover $C \to X := \proj_K^1$ over an extension as given in \Cref{prop persistent nodes} (or more directly, how to construct its quotient $\Xmini$ by the $p$-cyclic Galois group), and we specify what this field extension should be, in \Cref{thm construction of mini} below; this is the goal of \S\ref{sec building models result}.  We then apply this (in \S\ref{sec building models low genus}) to classify, in the cases of genus $1$ and $2$, the possible structures of the \emph{stable model} of $C$.

In the tame setting, without the potential degeneracy hypothesis, variations of the construction presented in this section have appeared as results in the literature: when $p = 2$ and $K$ has residue characteristic $\neq 2$, a regular semistable model of $C$ is explicitly constructed in \cite[\S4,5]{dokchitser2022arithmetic}, and the stable model of $C$ as a \emph{marked} curve is explicitly constructed by Gehrunger and Pink in \cite{gehrunger2021reduction}; meanwhile, for general $p$ under the assumption that $K$ has residue characteristic $\neq p$ (or any exponent $n \geq 2$ in the standard superelliptic equation not divisible by the residue characteristic), the stable model of the marked curve $C$ is constructed in \cite[\S4]{bouw2017computing}.  Where the wild setting is concerned, strategies for the construction of certain semistable models of $C$ have been developed and explicilty implemented for genus $2$ by the author in collaboration with Leonardo Fiore in \cite{fiore2026clusters} and independently by Gehrunger and Pink (treating $C$ as a marked curve) in \cite{gehrunger2024reduction}.

As is noted in \S\ref{sec intro comparison}, in the wild setting, the construction of a semistable model of a $p$-cyclic cover $C \to \proj_K^1$ and the structure of its special fiber is generally not determined by its cluster data (as it is in the tame setting), but when the curve is potentially degenerate, the results of this section show that it is determined by its cluster data even in the wild case.

\subsection{Explicit construction of the minimal regular model} \label{sec building models result}

Throughout this subsection, we recall the spaces $\Upsilon \subseteq \Sigma_{\mathcal{B}}^0 \subset \Sigma_{\mathcal{B}}$ and $\Sigma_{\mathcal{B}}^* \subset \Sigma_{\mathcal{B}}$ given in Definitions \ref{dfn Sigma} and \ref{dfn Sigma^*} and defined at the top of \S\ref{sec proof of main}; we also assume that our ground field $K$ contains a primitive $p$th root of unity $\zeta_p$ and that we have $\mathcal{B} \smallsetminus \{\infty\} \subset K$, noting that this implies that the depth and relative depth of each cluster lie in $v(K^\times)$.  These hypotheses in particular allow us to claim the following.

\begin{lemma} \label{lemma boundary point of Upsilon is a K-point}

Assume that $C$ is potentially degenerate.  Let $\eta_{D_{\alpha, b}}$ be a boundary point of the subspace $\Upsilon \subset \Hyp$.  Then we have $b \in v(K^\times)$.

\end{lemma}

\begin{proof}
By \Cref{thm main} (together with \Cref{prop properties (i) and (ii) of main theorem}), the multiset $\underline{\mathcal{B}}$ is clustered in pairs.  It then follows from this and from \Cref{cor t_f formula} that the boundary point $\eta := \eta_{D_{\alpha, b}} \in \Upsilon$ lies at a distance of exactly $\pfrac$ from the unique nearest point in $\Sigma_{\mathcal{B}}^*$ to it, which we denote by $\xi$, and that we have $(\eta, \xi) \cap \Upsilon = \varnothing$.  If we have $\eta \vee \xi > \xi$, then as this implies $\eta \vee \xi \in (\eta, \xi)$, we get that $\xi$ is the maximal point of its connected component of $\Sigma_{\mathcal{B}}^*$ and is thus a vertex of $\Sigma_{\mathcal{B}}$ by \Cref{prop connected components Sigma} combined with \Cref{prop highest point of Lambda}.  If instead we have $\xi > \eta$, then as $\xi$ lies in the interior of some path contained in $\Sigma_{\mathcal{B}}^*$ (namely a connected component, which is itself an open path) and as the path $[\eta, \xi]$ intersects $\Sigma_{\mathcal{B}}^*$ only at the point $\xi$, this point is again a vertex of $\Sigma_{\mathcal{B}}$.

By \Cref{prop removing a point from convex hull}, any vertex is of the form $\eta_{\mathfrak{s}} = \eta_{D_{\alpha, d(\mathfrak{s})}}$ for some cluster $\mathfrak{s} \ni \alpha$.  Since we have $\mathfrak{s} \subset K$, we have $d(\mathfrak{s}) \in v(K^\times)$.  Now as $\delta(\eta, \xi) = \pfrac \in v(K^\times)$ (because $v(1 - \zeta_p) = \frac{v(p)}{p}$), it follows from the definition of the distance function that we have $b \in v(K^\times)$.
\end{proof}

We recall that in \S\ref{sec normalizations p-cyclic covers}, given the smooth model $\mathcal{X}_D$ of $X$ corresponding to any disc $D \subset \bar{K}$, we produced explicit equations for its normalization $\mathcal{C}_D$ in the function field of $C$.  When $\alpha \in \bar{K}$ and $\beta \in \bar{K}^\times$ are chosen such that $D = D_{\alpha, v(\beta)}$, let us write $\mathcal{X}_{\alpha, \beta}, \mathcal{C}_{\alpha, \beta}$ for the respective models $\mathcal{X}_D, \mathcal{C}_D$.  The following lemma describes a field extension of $K$ over which the model $\mathcal{C}_{\alpha, \beta}$ is defined.

\begin{lemma} \label{lemma field of definition}

Assume that $C$ is potentially degenerate.  Choose $\alpha \in \mathcal{B} \smallsetminus \{\infty\}$ and $\beta \in \bar{K}^\times$ such that we have $\eta := \eta_{D_{\alpha, v(\beta)}} \in \Sigma_{\mathcal{B}}$.  Let $\mathfrak{s} = D_{\alpha, v(\beta)} \cap \mathcal{B}$ be the cluster associated to $\eta$ as in \Cref{prop removing a point from convex hull}.  Let $K'_\alpha / K$ be the extension of $K$ obtained by
\begin{itemize}
\item adjoining an element of $\bar{K}$ with valuation equal to $\frac{1}{p} \delta(\mathfrak{s})$ for each cluster $\mathfrak{c} \supseteq D_{\alpha, v(\beta)} \cap \mathcal{B}$ with $p \nmid \#\underline{\mathfrak{c}}$; and, 
\item if the residue characteristic of $K$ is $p$, adjoining $p$th roots of the elements 
\begin{equation} \label{eq elements to adjoin pth roots of}
(\alpha - z_0)^{m_0}, (\alpha - z_1)^{m_1} (\alpha - w_1)^{p-m_1}, \dots, (\alpha - z_h)^{m_h} (\alpha - w_h)^{p-m_h},
\end{equation}
where the roots $z_0, z_1, w_1, \dots, z_h, w_h$ and the integers $m_0, \dots, m_h$ are as in (\ref{eq superelliptic degenerate model}).
\end{itemize}

The model $\mathcal{C}_{\alpha, \beta}$ is defined over $K'_\alpha(\beta)$ if and only if we have 
\begin{enumerate}[(i)]
\item $\eta \in \Upsilon$, or 
\item $\eta \in \Sigma_{\mathcal{B}}^0 \smallsetminus \Upsilon$ and $\frac{1}{p} \delta(\eta, \Sigma_{\mathcal{B}}^*) \in v(K'_\alpha(\beta)^\times)$, or 
\item $\eta \in \Sigma_{\mathcal{B}}^* \smallsetminus \Sigma_{\mathcal{B}}^0$ and $\frac{1}{p} \delta(\eta, \eta_{\mathfrak{s}}) \in v(K'_\alpha(\beta)^\times)$, 
\end{enumerate}
and otherwise is defined over a degree-$p$ extension of $K'_\alpha(\beta)$.

\end{lemma}

\begin{rmk} \label{rmk distance well defined}

Property (ii) in the lemma above can only occur in the wild case, and we note that the distance $\delta(\eta, \Sigma_{\mathcal{B}}^*)$ is well defined even though the subspace $\Sigma_{\mathcal{B}}^* \subset \Sigma_{\mathcal{B}}$ is not connected: \Cref{thm main} (with \Cref{prop properties (i) and (ii) of main theorem}) says that $\underline{\mathcal{B}}$ is clustered in pairs, and then \Cref{cor t_f formula} implies that when $\mathfrak{t}_f(\eta) < \pfrac$, there is a unique component of $\Sigma_{\mathcal{B}}^*$ of minimal distance from $\eta$ (that distance equaling $\mathfrak{t}_f(\eta)$).

\end{rmk}

\begin{proof}[Proof (of \Cref{lemma field of definition})]
Recall the construction of $\mathcal{C}_{\alpha, \beta}$ given in \S\ref{sec normalizations p-cyclic covers}, which involves translating $y$ by the polynomial $q_{\alpha, \beta}(x_{\alpha, \beta})$ and scaling it by an element $\gamma^{1/p}$ whose $p$th power $\gamma$ satisfies $v(\gamma) = \mathfrak{t}_f(\eta) + v(f_{\alpha, \beta})$.  We first set out to show that, if chosen a certain way, the coefficients of the polynomial $q_{\alpha, \beta}$ (and thus also of $\rho_{\alpha, \beta}$) lie in $K'_\alpha(\beta)$.

By \Cref{thm main}, the partition $\mathcal{B} = \mathfrak{r}_0 \sqcup \dots \sqcup \mathfrak{r}_h$ given by \Cref{prop connected components Sigma} is given by $\mathfrak{r}_0 = \{z_0, w_0 := \infty\}$ and $\mathfrak{r}_i = \{z_i, w_i\}$, with $z_i$ (resp. $w_i$) having multiplicity $m_i$ (resp. $p - m_i$).  Write $f_0(x_{\alpha, 1}) = (x_{\alpha, 1} - (z_0 - \alpha))^{m_0}$ and $f_i(x_{\alpha, 1}) = (x_{\alpha, 1} - (z_i - \alpha))^{m_i} (x_{\alpha, 1} - (w_i - \alpha))^{p-m_i}$ for $1 \leq i \leq h$.  Let $q_0 = ((\alpha - z_0)^{m_0})^{1/p}$; for $1 \leq i \leq h$, let $q_i(x_{\alpha, 1}) = x_{\alpha, 1} + ((\alpha - z_i)^{m_i} (\alpha - w_i)^{p-m_i})^{1/p}$; and for $0 \leq i \leq h$, let $\rho_i = f_i - q_i^p$ (where $(\cdot)^{1/p}$ indicates having chosen a $p$th root in $\bar{K}$ of $\cdot$).  It is clear that the decomposition $f_i = q_i^p + \rho_i$ is total for $0 \leq i \leq h$ and thus is good by \Cref{cor total is good}.  We now claim that the polynomials $q_{\alpha, 1} := \prod_{i = 0}^h q_i$ and $\rho_{\alpha, 1} := f_{\alpha, 1} - q_{\alpha, 1}^p$ form a good decomposition of $f_{\alpha, 1}$.  To see this, we note that \Cref{prop product part-pth-power} tells us that we have $t_{q_{\alpha, 1}, \rho_{\alpha, 1}} \geq \min\{t_{q_i, \rho_i}\}_{1 \leq i \leq s}$, with the $t_{q, \rho}$ notation as in \S\ref{sec normalizations part-pth-power}.  As the decompositions $f_{\mathfrak{r}_i} = q_i^p + \rho_i$ are each good, the inequality $\min\{t_{q_{\alpha, 1}, \rho_{\alpha, 1}}, \pfrac\} \geq \min\{\mathfrak{t}_{\mathfrak{r}_i}\}_{1 \leq i \leq s}$ immediately follows.  \Cref{cor t_f formula} now implies both that equality must hold and that the decomposition $f = q^p + \rho$ is good.

We therefore may take $q_{\alpha, \beta}(x_{\alpha, \beta}) = q_{\alpha, 1}(\beta x_{\alpha, \beta})$ and $\rho_{\alpha, \beta}(x_{\alpha, \beta}) = \rho_{\alpha, 1}(\beta x_{\alpha, \beta})$ for the polynomials used in the construction of $\mathcal{C}_{\alpha, \beta}$.  Thus, the coefficients of $q_{\alpha, \beta}$ lie in the field generated over $K$ by $\beta$ as well as the $p$th roots of the elements in (\ref{eq elements to adjoin pth roots of}), which is contained in $K'_\alpha(\beta)$.  In the case that the residue characteristic of $K$ is not $p$, on the other hand, we may take $q_{\alpha, \beta} = 0$ and still get that the coefficients of $q_{\alpha, \beta}$ and $\rho_{\alpha, \beta}$ lie in $K'_\alpha(\beta)$.

Now one can always choose $\gamma$ to be an element of $K'_\alpha(\beta)$, because what we have shown implies that the polynomials $f_{\alpha, \beta}, q_{\alpha, \beta}, \rho_{\alpha, \beta}$ all have coefficients in $K'_\alpha(\beta)$ and $\pfrac = v((1 - \zeta_p)^p) \in v(K^\times)$, and so we get $v(f_{\alpha, \beta}), \mathfrak{t}_f(\eta) = \min\{v(\rho_{\alpha, \beta}) - v(f_{\alpha, \beta}), \pfrac\} \in v(K'_\alpha(\beta)^\times)$.  As $\mathcal{C}_{\alpha, \beta}$ is defined over $K'_\alpha(\beta, \gamma^{1/p})$, it is defined over a degree-$p$ extension of $K'_\alpha(\beta)$, and it remains to show that under certain hypotheses this coincides with $K'_\alpha(\beta)$ (\textit{i.e.} we may choose $\gamma \in K'_\alpha(\beta)$ such that its $p$th roots also lie in $K'_\alpha(\beta)$, which is equivalent to saying that $\frac{1}{p}(\mathfrak{t}_f(\eta) + v(f_{\alpha, \beta})) \in K'_\alpha(\beta)$).

We use the multiplicativity of the Gauss valuation and the formula in (\ref{eq scaled and translated polynomial}) to get 
\begin{equation} \label{eq v(f_alpha,beta) formula}
v(f_{\alpha, \beta}) = \#\underline{\mathfrak{s}} v(\beta) + \sum_{z \in \underline{\mathcal{B}} \smallsetminus \underline{\mathfrak{s}}} v(z - \alpha),
\end{equation}
where the sum is taken with multiplicity, and where $\mathfrak{s} = D_{\alpha, v(\beta)} \cap \mathcal{B}$.  It is clear that the sum on the right-hand side of (\ref{eq v(f_alpha,beta) formula}) consists of terms which are each the depth of a cluster containing $\mathfrak{s}$.

Suppose that $\eta$ lies in the subspace $\Sigma_{\mathcal{B}}^0 \subset \Sigma_{\mathcal{B}}$ defined in \S\ref{sec preliminaries partitions}.  Then by \Cref{prop boundary points}, we have $p \mid \#\underline{\mathfrak{s}}$ or else we have $\eta = \eta_{\mathfrak{s}}$ and $p \mid \#\underline{\mathfrak{c}}$ for some cluster $\mathfrak{c}$ whose parent $\mathfrak{c}'$ is $\mathfrak{s}$.  In the former case, the first term on the right-hand side of (\ref{eq v(f_alpha,beta) formula}) is divisible by $p$ in $v(K'_\alpha(\beta)^\times)$, while the number of terms in the sum on the right-hand side of (\ref{eq v(f_alpha,beta) formula}) is divisible by $p$, and it is easy to see that this sum is equal to the sum of $p$ times the depths of some clusters containing $\mathfrak{s}$ of cardinality divisible by $p$ (as a multiset) and some relative depths of other clusters containing $\mathfrak{s}$ of cardinality not divisible by $p$ (as a multiset); as such relative depths are divisible by $p$ in $v(K'_\alpha(\beta)^\times)$ by hypothesis, we have $\frac{1}{p} v(f_{\alpha, \beta}) \in v(K'_\alpha(\beta)^\times)$ in this case.  In the latter case, we may assume that $\alpha \in \mathfrak{c} \subsetneq \mathfrak{s} = \mathfrak{c}'$, as the choice of $\alpha \in D_{\alpha, v(\beta)}$ does not affect $f_{\alpha, \beta}$ (by \Cref{prop t_f well defined}(a)), and on noting that we have $v(z - \alpha) = v(\beta)$ for all $z \in \mathfrak{s} \smallsetminus \mathfrak{c} = \mathfrak{c}' \smallsetminus \mathfrak{c}$, we may rewrite (\ref{eq v(f_alpha,beta) formula}) as 
\begin{equation}
v(f_{\alpha, \beta}) = \#\underline{\mathfrak{c}} v(\beta) + \sum_{z \in \underline{\mathcal{B}} \smallsetminus \underline{\mathfrak{c}}} v(z - \alpha);
\end{equation}
as $p \mid \#\underline{\mathfrak{c}}$, the same argument gives us $\frac{1}{p} v(f_{\alpha, \beta}) \in v(K'_\alpha(\beta)^\times)$.

Now if $\eta \in \Upsilon \subseteq \Sigma_{\mathcal{B}}^0$, then certainly we have $\frac{1}{p} \mathfrak{t}_f(\eta) = \frac{v(p)}{p-1} = v(1 - \zeta_p) \in v(K^\times) \leq v(K'_\alpha(\beta)^\times)$; therefore, one can find a choice of $\gamma^{1/p}$ as an element of $K'_\alpha(\beta)^\times$ with the appropriate valuation.  If instead we have $\eta \in \Sigma_{\mathcal{B}}^0 \smallsetminus \Upsilon$, then, as in \Cref{rmk distance well defined}, we have $\delta(\eta, \Sigma_{\mathcal{B}}^*) = \mathfrak{t}_f(\eta)$, and so again one can find a choice of $\gamma^{1/p} \in K'_\alpha(\beta)$ with the appropriate valuation just as long as we have $\frac{1}{p} \delta(\eta, \Sigma_{\mathcal{B}}^*) \in v(K'_\alpha(\beta)^\times)$.  Thus, the claim is proved under the assumption of property (i) or (ii).

Finally, suppose that we have $\eta \in \Sigma_{\mathcal{B}}^* \smallsetminus \Sigma_{\mathcal{B}}^0$, so that $\mathfrak{t}_f(\eta) = 0$ by \Cref{cor t_f formula}.  Now \Cref{prop boundary points} implies that we have $p \nmid \#\underline{\mathfrak{s}}$.  We need to show that $\frac{1}{p}(\mathfrak{t}_f(\eta) + v(f_{\alpha, \beta})) = \frac{1}{p} v(f_{\alpha, \beta}) \in v(K'_\alpha(\beta)^\times)$ as long as $\frac{1}{p} \delta(\eta, \eta_{\mathfrak{s}}) \in v(K'_\alpha(\beta)^\times)$.  Let $\beta' \in K^\times$ satisfy $D_{\alpha, v(\beta')} = D_{\mathfrak{s}}$, or equivalently, $v(\beta') = d(\mathfrak{s}) \in v(K^\times)$.  It follows from \Cref{lemma t_f difference formula} and \Cref{lemma mathfrakv} that we have 
\begin{equation} \label{eq v(f_alpha,beta')}
v(f_{\alpha, \beta}) = v(f_{\alpha, \beta'}) - \#\underline{\mathfrak{s}} (v(\beta') - v(\beta)) = v(f_{\alpha, \beta'}) - \#\underline{\mathfrak{s}} \delta(\eta, \eta_{\mathfrak{s}}).
\end{equation}
The point $\eta_{\mathfrak{s}}$ is a vertex of $\Sigma_{\mathcal{B}}$ by \Cref{prop removing a point from convex hull}, which means that $\geq 3$ paths having $\eta_{\mathfrak{s}}$ as one endpoint whose pairwise intersections are $\{\eta_{\mathfrak{s}}\}$ are contained in $\Sigma_{\mathcal{B}}$.  At most $2$ of these paths can have non-singleton intersection with $\Sigma_{\mathcal{B}}^*$, as from \Cref{dfn clustered in pairs berk} we see that the connected component of $\Sigma_{\mathcal{B}}^*$ containing $\eta_{\mathfrak{s}}$ is itself a single path (passing through $\eta_{\mathfrak{s}}$).  It follows that we have $\eta_{\mathfrak{s}} \in \Sigma_{\mathcal{B}}^0$.  Since $\eta_{\mathfrak{s}}$ also lies in $\Sigma_{\mathcal{B}}^*$ by definition, we have $\frac{1}{p} \delta(\eta_{\mathfrak{s}}, \Sigma_{\mathcal{B}}^*) = 0 \in v(K'_\alpha(\beta)^\times)$, so property (i) or (ii) holds for $\eta_{D_{\alpha, v(\beta')}} = \eta_{\mathfrak{s}}$, and we have shown above that this implies $\frac{1}{p} v(f_{\alpha, \beta'}) \in v(K'_\alpha(\beta)^\times)$.  Then from (\ref{eq v(f_alpha,beta')}), we have $\frac{1}{p} v(f_{\alpha, \beta}) \in v(K'_\alpha(\beta)^\times)$ if and only if $\frac{1}{p} \#\underline{\mathfrak{s}} \delta(\eta, \eta_{\mathfrak{s}}) \in v(K'_\alpha(\beta)^\times)$.  As $p \nmid \#\underline{\mathfrak{s}}$ and $\delta(\eta, \eta_{\mathfrak{s}}) = v(\beta') - v(\beta) \in v(K_\alpha(\beta)^\times)$, this proves the claim under property (iii).
\end{proof}

\begin{thm} \label{thm construction of mini}

Assume that $C$ is potentially degenerate.  Let $\mathfrak{s}_1, \dots, \mathfrak{s}_h$ be the clusters given by \Cref{prop connected components Sigma}, and for $1 \leq j \leq h$, choose a representative $\alpha_j \in \mathfrak{s}_j$.  Let $K' / K$ be the extension of $K$ obtained by 
\begin{itemize}
\item adjoining an element of $\bar{K}$ with valuation equal to $\frac{1}{p} \delta(\mathfrak{s})$ for each cluster $\mathfrak{s}$ with $p \nmid \#\underline{\mathfrak{s}}$ or $\delta(\mathfrak{s}) < \pfrac$ (which requires adjoining at most one element, which will have degree $p$ over $K$, as we have $\delta(\mathfrak{s}) \in K^\times$ for all clusters $\mathfrak{s}$); and, 
\item if the residue characteristic of $K$ is $p$, adjoining $p$th roots of the elements in 
\begin{equation}
\{(\alpha_j - z_0)^{m_0}\}_{1 \leq j \leq h} \cup \{(\alpha_j - z_i)^{m_i} (\alpha_j - w_i)^{p-m_i}\}_{1 \leq i, j \leq h},
\end{equation}
where the roots $z_0, z_1, w_1, \dots, z_h, w_h$ and the integers $m_0, \dots, m_h$ are as in (\ref{eq superelliptic degenerate model}).
\end{itemize}

The minimal regular model $\Cmini$ of $C$ satisfying the properties given by \Cref{prop persistent nodes} is defined over $K'$ (in particular, the curve $C$ attains semistable reduction over $K'$), and it can be constructed as follows.  Let $\Xmini$ be the quotient of $\Cmini$ by the Galois group, and let $\Dmini$ be the collection of discs such that we have $\Xmini = \mathcal{X}_{\Dmini}$ as in \Cref{prop models of projective line}.  Then $\Dmini$ consists of the discs $D_{\alpha, b}$ which satisfy the following properties.
\begin{enumerate}[(i)]
\item The point $\eta_{D_{\alpha, b}}$ lies in the convex hull $\hat{\Upsilon} \subset \Hyp$ of the space $\Upsilon$.
\item We have $b \in v((K')^\times)$ if $\eta_{D_{\alpha, b}} \in \Upsilon$.
\item We have $\frac{1}{p} \delta(\eta_{D_{\alpha, b}}, \Sigma_{\mathcal{B}}^*) \in v((K')^\times)$ if $\eta_{D_{\alpha, b}} \in \Sigma_{\mathcal{B}}^0 \smallsetminus \Upsilon$.
\item We have $\frac{1}{p} \delta(\eta_{D_{\alpha, b}}, \eta_{\mathfrak{s}}) \in v((K')^\times)$ if $\eta_{D_{\alpha, b}} \in \Sigma_{\mathcal{B}}^* \smallsetminus \Sigma_{\mathcal{B}}^0$, where $\mathfrak{s} = D_{\alpha, b} \cap \mathcal{B}$.
\end{enumerate}

\end{thm}

\begin{rmk} \label{rmk every pth vertex}

Let us call a point $\eta = \eta_{D_{\alpha, b}} \in \Sigma_{\mathcal{B}}$ a \emph{$(K')$-point} if we have $b \in v((K')^\times)$.  Properties (i)--(iv) may be understood, in plainer and less rigorous language, as saying that the collection of points $\eta \in \Hyp$ whose corresponding discs are members of the collection $\Dmini$ consists of every $(K')$-point of $\Upsilon$ along with ``every $p$th $(K')$-point" of $\hat{\Upsilon} \smallsetminus \Upsilon$ (in the sense of every $p$th $(K')$-point along a given path contained in $\hat{\Upsilon} \smallsetminus \Upsilon$).  The construction of $K'$ (and in particular, the $\frac{1}{p} \delta(\mathfrak{s})$ assumption) ensures that the lengths of maximal paths in $\hat{\Upsilon} \smallsetminus \Upsilon$ are divisible by $p$ in $v((K')^\times)$.

\end{rmk}

The proof of \Cref{thm construction of mini} essentially consists of two parts: first we prove the statement when $K'$ is replaced by some extension $K'' / K'$ over which $C$ attains semistable reduction and its minimal regular model has the property given by \Cref{prop persistent nodes}, which we take care of through the following lemma, and then we prove that the same assertions are true for $K'$ itself.

\begin{lemma} \label{lemma construction of mini for K''}

With the definitions and set-up of \Cref{thm construction of mini}, let $K'' / K'$ be a finite extension.  Define $\hat{\Upsilon}_{K''}$ be the set of points of $\Hyp$ satisfying properties (i)--(iv) as given in the statement of the theorem but with $K'$ replaced by $K''$.  Let $\mathcal{X}_{\hat{\Upsilon}_{K''}}$ denote the model of $X := \proj_K^1$ coming from the set of discs corresponding to the points in $\hat{\Upsilon}_{K''}$ as in \Cref{prop models of projective line}.

Suppose that $C$ attains semistable reduction over $K''$ and moreover that, letting $\Cmini \to \Xmini$ denote the minimal regular model of $C$ over $K''$ as a $p$-cyclic cover of the quotient model of $X$, each node of $\SF{\Cmini}$ maps to a node of $\SF{\Xmini}$.  Then we have $\Xmini = \mathcal{X}_{\hat{\Upsilon}_{K''}}$.

\end{lemma}

\begin{proof}
Let us denote by $\Hmini \subset \Hyp$ the collection of points whose discs are members of the collection which determines the model $\Xmini$ of the projective line as in \Cref{prop models of projective line}.

Let $\eta_{D_{\alpha, b}}$ be a boundary point of $\Upsilon$, which by \Cref{rmk no non-isolated boundary points} is not isolated.  We have $b \in v(K^\times) \leq v((K'')^\times)$ by \Cref{lemma boundary point of Upsilon is a K-point}.  Then one checks using \Cref{lemma field of definition} that the corresponding model $\mathcal{C}_{\alpha, \beta}$ (for some $\beta \in K^\times$ with $v(\beta) = b$) is defined over $K'$ (and thus certainly over $K''$), and so we have $\eta_{D_{\alpha, b}} \in \Hmini$ thanks to \Cref{cor relating nodes to Upsilon}.

Let $\hatHmini \subset \Hyp$ denote the convex hull of $\Hmini$, and let $\eta_{D_{\alpha, b}}$ be a boundary point of $\hatHmini$.  Then the component $V$ of $\SF{\Xmini}$ corresponding to $D_{\alpha, b}$ (as in \S\ref{sec normalizations models of P^1}) meets the other components of $\SF{\Xmini}$ at a unique node $P$.  The inverse image of $P$ in $\SF{\Cmini}$ consists of nodes by \Cref{prop persistent nodes}; if it is a unique node $Q$, then it follows from \Cref{cor inverse images are smooth} that the inverse image of $V$ in $\SF{\Cmini}$ must consist of a single component $W$ meeting the rest of $\SF{\Cmini}$ only at $Q$; with the notation of \S\ref{sec preliminaries geometric components}, we have $w(W) = 1$.  As the abelian rank of any component of the special fiber of a (split) degenerate model has genus $0$, this contradicts \Cref{prop part of mini}.  Therefore, there are $p$ nodes $Q_1, \dots, Q_p$ of $\SF{\Cmini}$ lying over $P$.  If there are $p$ components $W_1, \dots, W_p$ of $\SF{\Cmini}$ lying over $V$, then each $W_i$ intersects the rest of $\SF{\Cmini}$ at one of the nodes $P_j$ and we have $w(W_1) = \dots = w(W_p) = 1$, again contradicting \Cref{prop part of mini}.  Therefore, the inverse image of $V$ in $\SF{\Cmini}$ consists of a single component $W$ intersecting the rest of $\SF{\Cmini}$ at the nodes $Q_1, \dots, Q_p$; this is the strict transform of $\SF{\mathcal{C}_{\alpha, \beta}}$ (for some $\beta \in (K')^\times$ satisfying $v(\beta) = b$), and as $\SF{\mathcal{C}_{\alpha, \beta}}$ is obtained from $W$ by contracting the other components of $\SF{\Cmini}$, the point in $\SF{\mathcal{C}_{\alpha, \beta}}$ lying over $P$ has $p$ branches passing through it.  Then \Cref{cor relating nodes to Upsilon} and \Cref{lemma relating nodes to Upsilon converse} tell us that $\eta_{D_{\alpha, b}}$ must be a boundary point of $\Upsilon$.  Since all boundary points of $\Upsilon$ lie in $\Hmini$, we have $\hatHmini = \hat{\Upsilon}$.

Noting that we have $\hat{\Upsilon} \subset \Sigma_{\mathcal{B}}$ (as $\Sigma_{\mathcal{B}}$ is connected), it is straightforward to check using \Cref{lemma field of definition} that a point $\eta_{D_{\alpha, v(\beta)}} \in \hat{\Upsilon}$ lies in $\hat{\Upsilon}_{K''}$ if and only if the corresponding model $\mathcal{C}_{\alpha, b}$ is defined over $K''$.  It is therefore clear that we have $\Hmini \subseteq \hat{\Upsilon}_{K''}$; we need to show the reverse inclusion.  We have already seen that the boundary points of $\Upsilon$ lie in $\Hmini$.  Now any interior point of $\Upsilon$ lying in $\Upsilon_{K''}$ also lies in $\Hmini$ by \Cref{cor relating nodes to Upsilon}.  It therefore remains to show that we have $\hat{\Upsilon}_{K''} \smallsetminus \Upsilon_{K''} \subset \Hmini$.

Suppose that there is a point $\eta \in \hat{\Upsilon}_{K''} \smallsetminus \Upsilon_{K''}$ which does not lie in $\Hmini$.  There then exist boundary points $\xi', \xi''$ of $\Upsilon$ satisfying $\eta \in (\xi', \xi'')$ and $(\xi', \xi'') \cap \Upsilon = \varnothing$.  As $\xi', \xi''$ are boundary points, they lie in $\Hmini$.  There must be a maximal sub-path $[\eta' = \eta_{D'}, \eta'' = \eta_{D''}] \subset [\xi', \xi'']$ containing the ``missing" point $\eta$ in its interior and such that $(\eta', \eta'') \cap \Hmini = \varnothing$.  Property (iii) implies the points in the path $[\xi', \xi'']$ belonging to $\Hmini$ lie at intervals of length $p v(\pi)$ (see \Cref{rmk every pth vertex}), where $\pi$ is a uniformizer of $K''$; we therefore have $\delta(\eta', \eta), \delta(\eta'', \eta) \geq p v(\pi)$ and so $\delta(\eta', \eta'') \geq 2p v(\pi)$.  By \Cref{prop models of projective line special fiber}, the (strict transforms of the) special fibers of the smooth models $\mathcal{X}_{D'}$ and $\mathcal{X}_{D''}$ meet in $\SF{\Xmini}$ at a node $P$ of thickness $> 2p$.

Let $Q \in \SF{\Cmini}(k)$ be a point lying above $P$; it is necessarily a node (by \Cref{prop persistent nodes}), and we denote its thickness by $\theta$.  Then it follows from \cite[Proposition 10.3.48(c)]{liu2002algebraic} that the node $P$ has thickness $\leq p \theta$ (in fact, one can show that there is only one node $Q$ lying above $P$ and from this that $P$ has thickness exactly $p \theta$).  Thus, the node $Q$ has thickness $\geq 2$, which contradicts the fact that each node of the special fiber of a regular model has thickness $1$.  This completes the proof.
\end{proof}

\begin{proof}[Proof (of \Cref{thm construction of mini})]
Let $K'' / K'$ be a finite extension over which $C$ achieves semistable reduction and satisfies the property given by \Cref{prop persistent nodes} about nodes mapping to nodes.  We adopt the notation of \Cref{lemma construction of mini for K''} and write $\mathcal{C}_{\hat{\Upsilon}_{K'}}$ (resp. $\mathcal{C}_{\hat{\Upsilon}_{K''}}$) for the normalization in $K'(C)$ (resp. $K''(C)$) of the model $\mathcal{X}_{\hat{\Upsilon}_{K'}}$ (resp. $\mathcal{X}_{\hat{\Upsilon}_{K''}}$) of $X = \proj_K^1$.  We proceed to show that $\mathcal{C}_{\hat{\Upsilon}_{K'}}$ is a semistable model of $C$ defined over $K'$ such that nodes of $\SF{\mathcal{C}_{\hat{\Upsilon}_{K'}}}$ map to nodes of $\SF{\mathcal{X}_{\hat{\Upsilon}_{K'}}}$.  Given this claim, applying \Cref{lemma construction of mini for K''} combined with the second statement of \Cref{prop persistent nodes} gives us the statement of the theorem.

\Cref{lemma construction of mini for K''} implies that the minimal regular model of $C$ over $K''$ is given by $\mathcal{C}_{\hat{\Upsilon}_{K''}}$.  Let us consider the model $\mathcal{X}_{\hat{\Upsilon}_{K'}} \otimes \Spec(\mathcal{O}_{K''})$ and its normalization in $K''(C)$, which we denote by $\mathcal{C}_{\hat{\Upsilon}_{K'}}'$ and which is obtained from the minimal regular model $\mathcal{C}_{\hat{\Upsilon}_{K''}}$ by contracting all components of the special fiber lying over the components of $\SF{\mathcal{X}_{\hat{\Upsilon}_{K''}}}$ corresponding to the points in $\hat{\Upsilon}_{K''} \smallsetminus \hat{\Upsilon}_{K'}$.  It is clear from applying \Cref{lemma field of definition} that these models are given by equations with coefficients in $K'$.  We also observe that the extensions $K'(C) / K'(x)$ and $K''(x) / K'(x)$ are linearly disjoint.  It follows that we have $\mathcal{C}_{\hat{\Upsilon}_{K'}}' = \mathcal{C} \otimes \Spec(\mathcal{O}_{K''})$ for some normal cover $\mathcal{C} \to \mathcal{X}_{\hat{\Upsilon}_{K'}} / \mathcal{O}_{K'}$ of $C \to X$ (which by the universal property of normalization, has a map $\mathcal{C} \to \mathcal{C}_{\hat{\Upsilon}_{K'}}$); the model $\mathcal{C} \otimes \Spec(\mathcal{O}_{K''})$ being the normalization of $\mathcal{X}_{\hat{\Upsilon}_{K'}} \otimes \Spec(\mathcal{O}_{K''})$ then implies that $\mathcal{C}$ itself is the normalization $\mathcal{C}_{\hat{\Upsilon}_{K''}}$ of $\mathcal{X}_{\hat{\Upsilon}_{K'}}$.

No point of $\hat{\Upsilon}_{K''} \smallsetminus \hat{\Upsilon}_{K'}$ is a vertex of $\Sigma_{\mathcal{B}}$, because by \Cref{prop removing a point from convex hull}, a vertex is of the form $\eta_{\mathfrak{s}} = \eta_{\alpha, d(\mathfrak{s})}$ for some cluster $\mathfrak{s} \ni \alpha$; we have $d(\mathfrak{s}) \in v(K^\times)$ so that $\eta_{\mathfrak{s}} \in \hat{\Upsilon}_{K'}$.  Moreover, by \Cref{lemma boundary point of Upsilon is a K-point}, the boundary points of $\Upsilon$ lie in $\hat{\Upsilon}_{K'}$.  It follows from these observations 
that the points of $\hat{\Upsilon}_{K''} \smallsetminus \hat{\Upsilon}_{K'}$ are each adjacent to exactly $2$ points in $\hat{\Upsilon}_{K''}$ (where \emph{adjacent} points have a path between them which does not pass through any other point of $\hat{\Upsilon}_{K''}$).  This is equivalent to saying that in the special fiber $\SF{\mathcal{X}_{\hat{\Upsilon}_{K''}}}$, the component corresponding to a point $\eta = \eta_{D_{\alpha, v(\beta)}} \in \hat{\Upsilon}_{K''} \smallsetminus \hat{\Upsilon}_{K'}$ -- that is, the strict transform $V$ of $\SF{\mathcal{X}_{\alpha, \beta}}$ -- meets the rest of the special fiber at exactly $2$ nodes $P_1, P_2$.

If $\eta$ is an interior point of $\Upsilon$, then \Cref{cor relating nodes to Upsilon} implies that the inverse image of $V$ in $\SF{\mathcal{C}_{\hat{\Upsilon}_{K''}}}$ (which is isomorphic to the normalization of $\SF{\mathcal{C}_{\alpha, \beta}}$) has $p$ components, each of which is a copy of $\proj_k^1$ meeting the rest of the special fiber only at a node lying over $P_1$ and a node lying over $P_2$.  Otherwise, we must have $\eta \notin \Upsilon$, and then \Cref{cor relating nodes to Upsilon} implies that the special fiber $\SF{\mathcal{C}_{\alpha, \beta}}$ has $1$ component; the inverse image of $V$ is the strict transform of $\SF{\mathcal{C}_{\alpha, \beta}}$, isomorphic to its normalization, which then must have a single component $W$.  If $W$ had $p$ nodes lying over $P_1$ at which it met another component of $\SF{\mathcal{C}_{\hat{\Upsilon}_{K''}}}$, on contracting the other components to get $\SF{\mathcal{C}_{\alpha, \beta}}$, these $p$ nodes would collapse to a singular point with $p$ branches passing through it, which \Cref{lemma relating nodes to Upsilon converse} implies is impossible.  Therefore, $W$ has a unique node $Q_1$ lying over $P_1$ and similarly has a unique node $Q_2$ lying over $P_2$; these are its points of intersection with the rest of $\SF{\mathcal{C}_{\hat{\Upsilon}_{K''}}}$.

Thus, in either case, each component $W$ of $\SF{\mathcal{C}_{\hat{\Upsilon}_{K''}}}$ corresponding to $\eta \in \hat{\Upsilon}_{K''} \smallsetminus \hat{\Upsilon}_{K'}$ that we are contracting in the process of forming $\mathcal{C}_{\hat{\Upsilon}_{K'}}'$ is a line meeting the other components at exactly $2$ nodes, lying over a line $V$ in $\SF{\mathcal{X}_{\hat{\Upsilon}_{K''}}}$ meeting the other components at exactly $2$ nodes.  It is well known that contracting such a line $W$ does not affect semistability: it only creates a node in place of $W$ which in our situation lies over the node obtained by contracting $V$, and so it does not create a new node lying over a smooth point.

It follows that $\mathcal{C}_{\hat{\Upsilon}_{K'}}' = \mathcal{C}_{\hat{\Upsilon}_{K'}} \otimes \Spec(\mathcal{O}_{K''})$ is a semistable model whose nodes lie over nodes in $\mathcal{X}_{\hat{\Upsilon}_{K'}} \otimes \Spec(\mathcal{O}_{K''})$.  By \cite[Lemma 10.3.30(a)]{liu2002algebraic}, the model $\mathcal{C}_{\hat{\Upsilon}_{K'}}$ is hence semistable, and it is clear that its nodes lie over nodes in $\mathcal{X}_{\hat{\Upsilon}_{K'}}$.  We have thus proved the desired statement.
\end{proof}

\subsection{Construction of stable models and classification for low genus} \label{sec building models low genus}

\Cref{thm construction of mini} really allows us to explicitly construct an appropriate finite extension $K' / K$ and the minimal regular models $\Cmini / \mathcal{O}_{K'}$ of $p$-cyclic covers $C / K$ of the projective line in the case that they are potentially degenerate: as we have seen in the proof of \Cref{lemma field of definition}, to come up with equations defining the components of $\SF{\Cmini}$ involves finding part-$p$th-power decompositions $f_{\alpha, \beta} = q_{\alpha, \beta}^p + \rho_{\alpha, \beta}$, where $q_{\alpha, \beta}$ can be taken to be the product of some polynomials of degree at most $1$ given by straightforward formulas, and there is no need to compute a total part-$p$th-power decomposition of $f_{\alpha, \beta}$.

In the case of genus $1$ (which only occurs for $p = 2$: see \Cref{rmk genus divisible by p - 1}), this recovers the well known construction of minimal regular models of elliptic curves $C / K$ with potentially multiplicative reduction (see Case ($b_m$) in \cite[\S1.5]{bosch2012neron} for the tame case, for instance).  Let $K' / K$ be a finite extension over which $C$ attains split multiplicative reduction.  \Cref{thm main}, translated into the language of clusters by \Cref{rmk translating main thm}, implies that there must be a cardinality-$2$ cluster $\mathfrak{s}$ with $\delta(\mathfrak{s}) > 2\pfrac = 4v(2)$.  In fact, we have $\mathcal{B} \smallsetminus \{\infty\} \subset K'$ in this situation, so that such a curve can be defined over $K'$ using an equation in the so-called Legendre form $y^2 = x(x - 1)(x - \lambda) \in K'[x]$, where one may choose $\lambda \in K'$ such that $v(\lambda) > 0$ (so that $\mathfrak{s} = \{0, \lambda\}$); then the formula in terms of $\lambda$ for the $j$-invariant of this curve given in \cite[Proposition III.1.7]{silverman2009arithmetic} shows that the condition $\delta(\mathfrak{s}) > 4v(2)$ is equivalent to the $j$-invariant being non-integral and in fact that its valuation equals $8v(2) - 2v(\lambda)$.

Now one can easily compute (using \Cref{cor t_f formula} in the case of residue characteristic $2$) that $\Upsilon$ coincides with the single path $[\eta_{0, b}, \eta_{0, b'}]$, where $b = d(\mathfrak{s}) - 2v(2) = v(\lambda) - 2v(2)$ and $b' = d(\mathcal{B} \smallsetminus \{\infty\}) + 2v(2) = 2v(2)$.  Then \Cref{thm construction of mini} tells us that $\Xmini = \mathcal{X}_{\Dmini}$ is constructed using the collection $\Dmini$ consisting of the discs $D_{0, b}$ such that $b \in [v(\lambda) - 2v(2), 2v(2)] \cap v((K')^\times)$.  Letting $\pi$ be a uniformizer of $K'$, there are then exactly $(v(\lambda) - 4v(2)) / v(\pi) + 1$ components of $\SF{\Xmini}$ which form a single chain of lines.  One sees using \Cref{cor relating nodes to Upsilon} that the inverse image in $\SF{\Cmini}$ of each of the lines at the end of this chain has a single component isomorphic to $\proj_k^1$ (as each corresponds to a boundary point of $\Upsilon = \hat{\Upsilon}$), while the inverse image in $\SF{\Cmini}$ of each of the intermediate lines in the chain has $2$ components, each isomorphic to $\proj_k^1$.  It follows that the special fiber $\SF{\Cmini}$ consists of a chain of $2v(\lambda) - 8v(2)$ lines forming a loop.

When $C$ has genus $\geq 2$, one can construct not only the minimal regular model of $C$ over an appropriate extension $K' / K$ but also the \emph{stable model} of $C$ over $K'$, which has the advantage that its special fiber is simpler to describe apart from the fact that in general its nodes are of varying thickness.  Below we classify the possible configurations of the special fiber of the stable model of a $p$-cyclic cover $C$ of the projective line in the case that the genus of $C$ is $2$ (which only occurs when $p = 2$ or $p = 3$: see \Cref{rmk genus divisible by p - 1}).

\begin{rmk} \label{rmk stable models}

Assume that $C$ has genus $\geq 2$ and that $C$ has (split) degenerate reduction over a finite extension $K' / K$.  The stable model of $C$ is the minimum (with respect to dominance) among all semistable models of $C$, and it can be characterized as the unique semistable model $\mathcal{C}$ of $C$ such that each component $V \cong \proj_k^1$ of $\SF{\mathcal{C}}$ satisfies $w(V) \geq 3$, where $w$ is defined as in \S\ref{sec preliminaries geometric components}.  Consequently, \Cref{prop part of mini} implies that the stable model of $C$ can be obtained from the minimal regular model $\mathcal{C}$ of $C$ by contracting each component $V$ of $(\mathcal{C})_s$ with $w(V) = 2$.

Assume further that the extension $K'$ is as in \Cref{thm construction of mini}, which guarantees the existence of and describes the minimal regular model $\Cmini$ of $C$ as a $p$-cyclic cover of its quotient $\Xmini$ of $X$.  Using the results of \S\ref{sec proof of main}, in particular \Cref{cor relating nodes to Upsilon}, along with the geometric arguments used to prove them, one can verify the following (where we use the notation of \Cref{lemma construction of mini for K''}):
\begin{enumerate}[(a)]
\item each point of $\hat{\Upsilon}_{K'}$ which is neither a boundary point of $\Upsilon$ nor a vertex of $\Sigma_{\mathcal{B}}$ corresponds to a component of $\SF{\Xmini}$ whose inverse image in $\SF{\Cmini}$ consists of components $V$ with $w(V) = 2$;
\item if $p \geq 3$, then the points specified by (a) account for all components $V$ of $\SF{\Cmini}$ with $w(V) = 2$; and 
\item if $p = 2$, the remaining components $V$ of $\SF{\Cmini}$ with $w(V) = 2$ are the inverse images of the components of $\SF{\Xmini}$ corresponding to those (boundary) points of $\hat{\Upsilon}$ which do not lie in the interior of any path contained in $\hat{\Upsilon}$.
\end{enumerate}

In the other direction, one may always recover the minimal regular model from the stable model by desingularizing it, which amounts to replacing each node of thickness $\theta \geq 2$ with a chain of $\theta - 1$ lines connected by nodes of thickness $1$ by \cite[Corollary 10.3.25]{liu2002algebraic}.

\end{rmk}

Using \Cref{rmk stable models} as a guide, one can compute the following classification for genus $2$.

There are three possible configurations of components and nodes of the special fiber of the stable model $\Cst$ of a curve $C$ of genus $2$, depicted and labeled (as Type A, B, or C) in \Cref{fig special fibers genus 2} below.

\begin{figure}[h!]

\begin{subfigure}[b]{.3\textwidth}
\centering
\includegraphics[height=1.75cm]{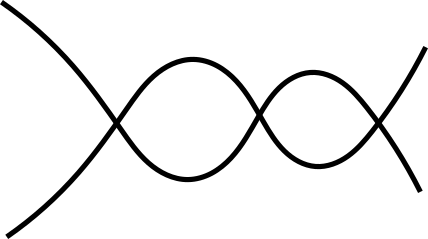}
\caption{Type A}
\end{subfigure}
~
\begin{subfigure}[b]{.3\textwidth}
\centering
\raisebox{.2cm}{\includegraphics[height=1cm]{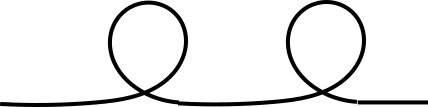}}
\caption{Type B}
\end{subfigure}
~
\begin{subfigure}[b]{.3\textwidth}
\centering
\includegraphics[height=1.5cm]{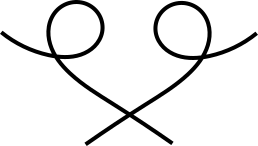}
\caption{Type C}
\end{subfigure}

\caption{The possible configurations of the special fiber of a curve of genus $2$}

\label{fig special fibers genus 2}
\end{figure}

For a fixed $p \in \{2, 3\}$, the topological structure of the subspace $\Upsilon \subset \Hyp$ associated to $C$ entirely determines the structure of the special fiber $\SF{\Cst}$.  When $p = 2$, the corresponding topological structures of $\Upsilon \subseteq \hat{\Upsilon}$ are shown in \Cref{fig Upsilon genus 2} below, where dotted lines indicate $\hat{\Upsilon} \smallsetminus \Upsilon$.

\begin{figure}[h!]

\begin{subfigure}[b]{.3\textwidth}
\centering
\resizebox{!}{1.8cm}{
\begin{tikzpicture}

\draw[very thick] (0,0) -- (1.7,1);
\draw[very thick] (3.4,0) -- (1.7,1);
\draw[very thick] (1.7,2.7) -- (1.7,1);
    
\end{tikzpicture}
}
\caption{Case A}
\end{subfigure}
~
\begin{subfigure}[b]{.3\textwidth}
\centering
\resizebox{!}{4cm/3}{
\begin{tikzpicture}

\draw[very thick] (0,0) -- (2,2);
\draw[very thick] (4,0) -- (2,2);
    
\end{tikzpicture}
}
\caption{Case B}
\end{subfigure}
~
\begin{subfigure}[b]{.3\textwidth}
\centering
\resizebox{!}{4cm/3}{
\begin{tikzpicture}

\draw[very thick] (0,0) -- (1.5,2);
\draw[very thick, dashed] (1.5,2) -- (3,2);
\draw[very thick] (4.5,0) -- (3,2);
    
\end{tikzpicture}
}
\caption{Case C}
\end{subfigure}

\caption{The possible topological structures of $\Upsilon$ in the case of genus $2$}

\label{fig Upsilon genus 2}
\end{figure}
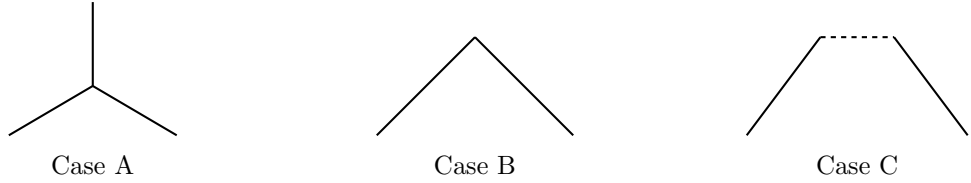

For $p = 2$, where we have $\#\mathcal{B} = 6$, let us break down the conditions on the cluster data which imply each of the above topological structures of $\Upsilon$.  In the diagrams of the subspaces $\Upsilon \subseteq \Sigma_{\mathcal{B}}^0 \subsetneq \Sigma_{\mathcal{B}}$ appearing in Figures \ref{fig spaces case A}, \ref{fig spaces case B}, and \ref{fig spaces case C} below, black lines indicate $\Upsilon$, gray lines indicate $\Sigma_{\mathcal{B}}^0 \smallsetminus \Upsilon$ (which is empty in the tame setting), and dashes indicate $\Sigma_{\mathcal{B}} \smallsetminus \Sigma_{\mathcal{B}}^0$.  Circles indicate the points of $\hat{\Upsilon}$ which are either vertices or boundary points of $\Upsilon$; those that are filled in correspond to components $V$ of $\SF{\Cmini}$ which satisfy $w(V) = 2$ and therefore are contracted to get the stable model.

\textbf{Case A:} The non-maximal clusters of $\mathcal{B}$ consist of $\mathfrak{c}_i = \{z_i, w_i\}$ for $i = 1, 2$ and of $\mathfrak{s} := \mathfrak{c}_1 \sqcup \mathfrak{c}_2$, with $\delta(\mathfrak{c}_1), \delta(\mathfrak{c}_2), \delta(\mathfrak{s}) > 2v(2)$.  In the wild setting, this corresponds to Gehrunger and Pink's Case D4 in \cite{gehrunger2024reduction}; this case is also treated in \cite{dokchitser2023note}: see Remark 1.6 of that paper.

\vspace{-1em}

\begin{figure}[h!]

\begin{subfigure}[b]{.4\textwidth}
\centering
\begin{tikzpicture}

    \coordinate (A) at (.5,1);
    \coordinate (B) at (2.25,1.125);
    \coordinate (C) at (1.5,3);
    \coordinate (D) at (1.5,2);
    \draw[very thick, dashed] (0,0) -- (A);
    \draw[very thick, dashed] (1,0) -- (A);
    \draw[very thick, dashed] (2,0) -- (B);
    \draw[very thick, dashed] (2.5,0) -- (B);
    \draw[very thick, dashed] (4,0) -- (C);
    \draw[very thick, dashed] (1.5,4) -- (C);

    \draw[very thick] (A) -- (D);
    \draw[very thick] (B) -- (D);
    \draw[very thick] (C) -- (D);

    \draw[black, fill=black] (A) circle(.5mm);
    \draw[black, fill=black] (B) circle(.5mm);
    \draw[black, fill=black] (C) circle(.5mm);
    \draw[black, fill=white] (D) circle(.5mm);

    \node[left=1pt] at (A) {$\eta_{\mathfrak{c}_1}$};
    \node[right=1pt] at (B) {$\eta_{\mathfrak{c}_2}$};
    \node[right=1pt] at (C) {$\eta_{\mathcal{B} \smallsetminus \{\infty\}}$};
    \node[left=1pt] at (D) {$\eta_{\mathfrak{s}}$};
    
\end{tikzpicture}
\caption{Case A -- tame}
\end{subfigure}
~
\begin{subfigure}[b]{.4\textwidth}
\centering
\begin{tikzpicture}

    \coordinate (A) at (.5,1);
    \coordinate (B) at (2.25,1.125);
    \coordinate (C) at (1.5,3);
    \coordinate (D) at (1.5,2);
    \coordinate (A1) at (.75,1.25);
    \coordinate (B1) at (2.02,1.393);
    \coordinate (C1) at (1.5,2.646);
    \draw[very thick, dashed] (0,0) -- (A);
    \draw[very thick, dashed] (1,0) -- (A);
    \draw[very thick, dashed] (2,0) -- (B);
    \draw[very thick, dashed] (2.5,0) -- (B);
    \draw[very thick, dashed] (4,0) -- (C);
    \draw[very thick, dashed] (1.5,4) -- (C);

    \draw[very thick, gray] (A) -- (A1);
    \draw[very thick, gray] (B) -- (B1);
    \draw[very thick, gray] (C) -- (C1);

    \draw[very thick] (A1) -- (D);
    \draw[very thick] (B1) -- (D);
    \draw[very thick] (C1) -- (D);

    \draw[black, fill=black] (A1) circle(.5mm);
    \draw[black, fill=black] (B1) circle(.5mm);
    \draw[black, fill=black] (C1) circle(.5mm);
    \draw[black, fill=white] (D) circle(.5mm);

    \node[left=1pt] at (A) {$\eta_{\mathfrak{c}_1}$};
    \node[right=1pt] at (B) {$\eta_{\mathfrak{c}_2}$};
    \node[right=1pt] at (C) {$\eta_{\mathcal{B} \smallsetminus \{\infty\}}$};
    \node[right=1pt] at (D) {$\eta_{\mathfrak{s}}$};
    
\end{tikzpicture}
\caption{Case A -- wild}
\end{subfigure}

\caption{The spaces $\Upsilon \subseteq \Sigma_{\mathcal{B}}^0 \subsetneq \Sigma_{\mathcal{B}}$ in Case A}

\label{fig spaces case A}
\end{figure}

The thicknesses of the three nodes in the configuration of Type A are given by twice the lengths of the paths from the boundary points of $\Upsilon$ (which in the tame case are $\eta_{\mathfrak{c}_1}, \eta_{\mathfrak{c}_2}, \eta_{\mathcal{B} \smallsetminus \{\infty\}} \in \Upsilon$) to the vertex $\eta_{\mathfrak{s}}$ and thus equal $2\delta(\mathfrak{c}_1) - 4v(2)$, $2\delta(\mathfrak{c}_2) - 4v(2)$, and $2\delta(\mathfrak{s}) - 4v(2)$.

\textbf{Case B:} The non-maximal clusters of $\mathcal{B}$ consist of $\mathfrak{c}_i = \{z_i, w_i\}$ for $i = 1, 2$ and, in the case that $K$ has residue characteristic $2$, of the cluster $\mathfrak{s} := \mathfrak{c}_1 \sqcup \mathfrak{c}_2$, with $\delta(\mathfrak{c}_1), \delta(\mathfrak{c}_2) > 2v(2)$ and (in the wild case) $\delta(\mathfrak{s}) = 2v(2)$.  In the wild setting, this corresponds to Gehrunger and Pink's Case D3 in \cite{gehrunger2024reduction}.  The result in this case also recovers a special case of \cite[Proposition 1.5]{dokchitser2023note}.

\begin{figure}[h!]

\begin{subfigure}[b]{.4\textwidth}
\centering
\begin{tikzpicture}

    \coordinate (A) at (.5,1);
    \coordinate (B) at (2.25,1.125);
    \coordinate (D) at (1.5,2);
    \draw[very thick, dashed] (0,0) -- (A);
    \draw[very thick, dashed] (1,0) -- (A);
    \draw[very thick, dashed] (2,0) -- (B);
    \draw[very thick, dashed] (2.5,0) -- (B);
    \draw[very thick, dashed] (1.5,0) -- (D);
    \draw[very thick, dashed] (1.5,3) -- (D);

    \draw[very thick] (A) -- (D);
    \draw[very thick] (B) -- (D);

    \draw[black, fill=black] (A) circle(.5mm);
    \draw[black, fill=black] (B) circle(.5mm);
    \draw[black, fill=white] (D) circle(.5mm);

    \node[left=1pt] at (A) {$\eta_{\mathfrak{c}_1}$};
    \node[right=1pt] at (B) {$\eta_{\mathfrak{c}_2}$};
    \node[right=1pt] at (D) {$\eta_{\mathcal{B} \smallsetminus \{\infty\}}$};
    
\end{tikzpicture}
\caption{Case B -- tame}
\end{subfigure}
~
\begin{subfigure}[b]{.4\textwidth}
\centering
\begin{tikzpicture}

    \coordinate (A) at (.5,1);
    \coordinate (B) at (2.25,1.125);
    \coordinate (C) at (1.5,2.354);
    \coordinate (A1) at (.75,1.25);
    \coordinate (B1) at (2.02,1.393);
    \coordinate (C1) at (1.5,2);
    \draw[very thick, dashed] (0,0) -- (A);
    \draw[very thick, dashed] (1,0) -- (A);
    \draw[very thick, dashed] (2,0) -- (B);
    \draw[very thick, dashed] (2.5,0) -- (B);
    \draw[very thick, dashed] (4.5,0) -- (C);
    \draw[very thick, dashed] (1.5,3) -- (C);

    \draw[very thick, gray] (A) -- (A1);
    \draw[very thick, gray] (B) -- (B1);
    \draw[very thick, gray] (C) -- (C1);

    \draw[very thick] (A1) -- (C1);
    \draw[very thick] (B1) -- (C1);

    \draw[black, fill=black] (A1) circle(.5mm);
    \draw[black, fill=black] (B1) circle(.5mm);
    \draw[black, fill=white] (C1) circle(.5mm);

    \node[left=1pt] at (A) {$\eta_{\mathfrak{c}_1}$};
    \node[right=1pt] at (B) {$\eta_{\mathfrak{c}_2}$};
    \node[left=1pt] at (C1) {$\eta_{\mathfrak{s}}$};
    \node[right=1pt] at (C) {$\eta_{\mathcal{B} \smallsetminus \{\infty\}}$};
    
\end{tikzpicture}
\caption{Case B -- wild}
\end{subfigure}

\caption{The spaces $\Upsilon \subseteq \Sigma_{\mathcal{B}}^0 \subsetneq \Sigma_{\mathcal{B}}$ in Case B}

\label{fig spaces case B}
\end{figure}
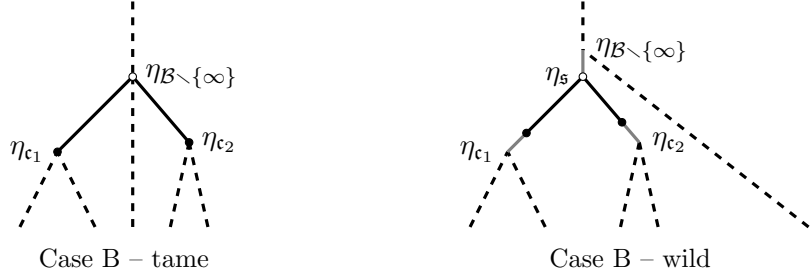

The thicknesses of the two nodes in the configuration of Type B are given by twice the lengths of the paths from the endpoints of the path $\Upsilon$ (which in the tame case are $\eta_{\mathfrak{s}_1}, \eta_{\mathfrak{s}_2} \in \Upsilon$) to the boundary point $\eta_{\mathcal{B} \smallsetminus \{\infty\}}$ (resp. $\eta_{\mathfrak{s}}$) in the tame (resp. wild) case and thus equal $2\delta(\mathfrak{c}_1) - 4v(2)$ and $2\delta(\mathfrak{c}_2) - 4v(2)$.

\textbf{Cases C1 and C2:} In these cases, which only occur in the wild setting, the non-maximal clusters of $\mathcal{B}$ consist of $\mathfrak{c}_i = \{z_i, w_i\}$ for $i = 1, 2$ and of $\mathfrak{s} := \mathfrak{c}_1 \sqcup \mathfrak{c}_2$, with either
\begin{itemize}
\item  $0 \leq \delta(\mathfrak{s}) < 2v(2)$ and $\delta(\mathfrak{c}_1), \delta(\mathfrak{c}_2) > 4v(2) - \delta(\mathfrak{s})$ (let us call this Case C1), or
\item  $0 \leq \delta(\mathfrak{c}_2) < 2v(2)$ and $\delta(\mathfrak{c}_1), \delta(\mathfrak{s}) > 4v(2) - \delta(\mathfrak{c}_2)$ (let us call this Case C2),
\end{itemize}
 where a subset of $\mathcal{B}$ having relative depth $0$ means that it is not a cluster.  Case C1 (resp. Case C2) corresponds to Gehrunger and Pink's Case C6 if $\delta(\mathfrak{s}) = 0$ (resp. $\delta(\mathfrak{c}_2) = 0$) and corresponds to Case D10 if $\delta(\mathfrak{s}) > 0$ (resp. if $\delta(\mathfrak{c}_2) > 0$) in \cite{gehrunger2024reduction}.

\textbf{Case C3:} The non-maximal clusters of $\mathcal{B}$ consist of $\mathfrak{c}_1 = \{z_1, w_1\} \subsetneq \mathfrak{d} = \{z_1, w_1, z_2\} \subsetneq \mathfrak{c}_2 = \{z_1, w_1, z_2, w_2\}$, with $\delta(\mathfrak{c}_1), \delta(\mathfrak{c}_2) > 4v(2)$.  In the wild setting, this corresponds to Gehrunger and Pink's Case G6 in \cite{gehrunger2024reduction}.

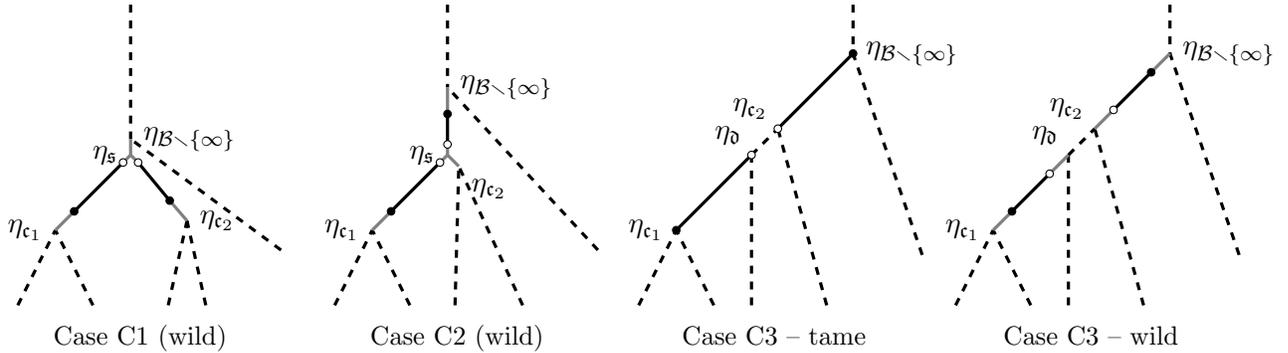
\begin{figure}[h!]

\begin{subfigure}[b]{.25\textwidth}
\centering
\begin{tikzpicture}

    \coordinate (A) at (.5,1);
    \coordinate (B) at (2.25,1.125);
    \coordinate (C) at (1.5,2.212);
    \coordinate (D) at (1.5,2);
    \coordinate (A1) at (.75,1.25);
    \coordinate (B1) at (2.02,1.393);
    \coordinate (C1a) at (1.4,1.9);
    \coordinate (C1b) at (1.6,1.9);
    \draw[very thick, dashed] (0,0) -- (A);
    \draw[very thick, dashed] (1,0) -- (A);
    \draw[very thick, dashed] (2,0) -- (B);
    \draw[very thick, dashed] (2.5,0) -- (B);
    \draw[very thick, dashed] (3.5,.735) -- (C);
    \draw[very thick, dashed] (1.5,4) -- (C);

    \draw[very thick, gray] (A) -- (A1);
    \draw[very thick, gray] (B) -- (B1);
    \draw[very thick, gray] (C1a) -- (D);
    \draw[very thick, gray] (C1b) -- (D);
    \draw[very thick, gray] (C) -- (D);

    \draw[very thick] (A1) -- (C1a);
    \draw[very thick] (B1) -- (C1b);

    \draw[black, fill=black] (A1) circle(.5mm);
    \draw[black, fill=black] (B1) circle(.5mm);
    \draw[black, fill=white] (C1a) circle(.5mm);
    \draw[black, fill=white] (C1b) circle(.5mm);

    \node[left=1pt] at (A) {$\eta_{\mathfrak{c}_1}$};
    \node[right=1pt] at (B) {$\eta_{\mathfrak{c}_2}$};
    \node[right=1pt] at (C) {$\eta_{\mathcal{B} \smallsetminus \{\infty\}}$};
    \node[left=1pt] at (D) {$\eta_{\mathfrak{s}}$};
    
\end{tikzpicture}
\caption{Case C1 (wild)}
\end{subfigure}
~
\begin{subfigure}[b]{.25\textwidth}
\centering
\begin{tikzpicture}

    \coordinate (A) at (.5,1);
    \coordinate (B) at (1.65,1.85);
    \coordinate (C) at (1.5,2.9);
    \coordinate (D) at (1.5,2);
    \coordinate (A1) at (.75,1.25);
    \coordinate (B1) at (1.5,2.141);
    \coordinate (C1) at (1.5,2.546);
    \coordinate (D1) at (1.4,1.9);

    \draw[very thick, dashed] (0,0) -- (A);
    \draw[very thick, dashed] (1,0) -- (A);
    \draw[very thick, dashed] (1.6,0) -- (B);
    \draw[very thick, dashed] (2.5,0) -- (B);
    \draw[very thick, dashed] (3.5,.735) -- (C);
    \draw[very thick, dashed] (1.5,4) -- (C);

    \draw[very thick, gray] (A) -- (A1);
    \draw[very thick, gray] (B) -- (D);
    \draw[very thick, gray] (D1) -- (D);
    \draw[very thick, gray] (B1) -- (D);
    \draw[very thick, gray] (C1) -- (C);

    \draw[very thick] (A1) -- (D1);
    \draw[very thick] (B1) -- (C1);

    \draw[black, fill=black] (A1) circle(.5mm);
    \draw[black, fill=white] (B1) circle(.5mm);
    \draw[black, fill=white] (D1) circle(.5mm);
    \draw[black, fill=black] (C1) circle(.5mm);

    \node[left=1pt] at (A) {$\eta_{\mathfrak{c}_1}$};
    \node[below right=1pt] at (B) {$\eta_{\mathfrak{c}_2}$};
    \node[right=1pt] at (C) {$\eta_{\mathcal{B} \smallsetminus \{\infty\}}$};
    \node[left=1pt] at (D) {$\eta_{\mathfrak{s}}$};
    
\end{tikzpicture}
\caption{Case C2 (wild)}
\end{subfigure}
~
\begin{subfigure}[b]{.25\textwidth}
\centering
\begin{tikzpicture}

    \coordinate (A) at (.5,1);
    \coordinate (B') at (1.5,2);
    \coordinate (B) at (1.85,2.35);
    \coordinate (C) at (2.85,3.35);
    \draw[very thick, dashed] (0,0) -- (A);
    \draw[very thick, dashed] (1,0) -- (A);
    \draw[very thick, dashed] (1.5,0) -- (B');
    \draw[very thick, dashed] (B') -- (B);
    \draw[very thick, dashed] (2.5,0) -- (B);
    \draw[very thick, dashed] (3.77,.67) -- (C);
    \draw[very thick, dashed] (2.85,4) -- (C);

    \draw[very thick] (A) -- (B');
    \draw[very thick] (B) -- (C);

    \draw[black, fill=black] (A) circle(.5mm);
    \draw[black, fill=white] (B') circle(.5mm);
    \draw[black, fill=white] (B) circle(.5mm);
    \draw[black, fill=black] (C) circle(.5mm);

    \node[left=1pt] at (A) {$\eta_{\mathfrak{c}_1}$};
    \node[above left=0pt] at (B') {$\eta_{\mathfrak{d}}$};
    \node[above left=0pt] at (B) {$\eta_{\mathfrak{c}_2}$};
    \node[right=1pt] at (C) {$\eta_{\mathcal{B} \smallsetminus \{\infty\}}$};
    
\end{tikzpicture}
\caption{Case C3 -- tame}
\end{subfigure}
~
\begin{subfigure}[b]{.25\textwidth}
\centering
\begin{tikzpicture}

    \coordinate (A) at (.5,1);
    \coordinate (B') at (1.5,2);
    \coordinate (B) at (1.85,2.35);
    \coordinate (C) at (2.85,3.35);
    \coordinate (A1) at (.75,1.25);
    \coordinate (B'1) at (1.25,1.75);
    \coordinate (B1) at (2.1,2.6);
    \coordinate (C1) at (2.6,3.1);
    \draw[very thick, dashed] (0,0) -- (A);
    \draw[very thick, dashed] (1,0) -- (A);
    \draw[very thick, dashed] (1.5,0) -- (B');
    \draw[very thick, dashed] (B') -- (B);
    \draw[very thick, dashed] (2.5,0) -- (B);
    \draw[very thick, dashed] (3.77,.67) -- (C);
    \draw[very thick, dashed] (2.85,4) -- (C);

    \draw[very thick, gray] (A) -- (A1);
    \draw[very thick, gray] (B'1) -- (B');
    \draw[very thick, gray] (B) -- (B1);
    \draw[very thick, gray] (C1) -- (C);

    \draw[very thick] (A1) -- (B'1);
    \draw[very thick] (B1) -- (C1);

    \draw[black, fill=black] (A1) circle(.5mm);
    \draw[black, fill=white] (B'1) circle(.5mm);
    \draw[black, fill=white] (B1) circle(.5mm);
    \draw[black, fill=black] (C1) circle(.5mm);

    \node[left=1pt] at (A) {$\eta_{\mathfrak{c}_1}$};
    \node[above left=0pt] at (B') {$\eta_{\mathfrak{d}}$};
    \node[above left=0pt] at (B) {$\eta_{\mathfrak{c}_2}$};
    \node[right=1pt] at (C) {$\eta_{\mathcal{B} \smallsetminus \{\infty\}}$};
    
\end{tikzpicture}
\caption{Case C3 -- wild}
\end{subfigure}

\caption{The spaces $\Upsilon \subseteq \Sigma_{\mathcal{B}}^0 \subsetneq \Sigma_{\mathcal{B}}$ in all sub-cases of Cases C}

\label{fig spaces case C}
\end{figure}

The thicknesses of the two nodes each lying in only one component in the configuration of Type C are given by twice the lengths of the segments which form each component of $\Upsilon$, while the thickness of the node connecting the two components is given by half of the length of the path connecting the two components of $\Upsilon$.  In Case C1, we have $\Sigma_{\mathcal{B}}^0 = [\eta_{\mathcal{B} \smallsetminus \{\infty\}}, \eta_{\mathfrak{c}_1}] \cup [\eta_{\mathcal{B} \smallsetminus \{\infty\}}, \eta_{\mathfrak{c}_2}]$, where the intersection $[\eta_{\mathcal{B} \smallsetminus \{\infty\}}, \eta_{\mathfrak{c}_1}] \cap [\eta_{\mathcal{B} \smallsetminus \{\infty\}}, \eta_{\mathfrak{c}_2}]$ has length $< 2v(v)$.  The subspace $\Upsilon \subset \Sigma_{\mathcal{B}}^0$ is then the disjoint union of the subpaths of the paths $[\eta_{\mathcal{B} \smallsetminus \{\infty\}}, \eta_{\mathfrak{c}_i}]$ obtained by removing segments of length $2v(2)$ from each end; therefore, each component of $\Upsilon$ has length $\delta(\mathfrak{c}_i) + \delta(\mathfrak{s}) - 4v(2)$, and the thicknesses of the nodes which each lie in only one component of the special fiber of the stable model are $2(\delta(\mathfrak{c}_i) + \delta(\mathfrak{s})) - 8v(2)$ for $i = 1, 2$.  Meanwhile, the distance between the two components of $\Upsilon$ is the distance between each path's point $\xi_i$ which is nearest to $\eta_{\mathfrak{s}}$; since the paths $[\xi_i, \eta_{\mathcal{B} \smallsetminus \{\infty\}}]$ each have length $2v(2)$ and intersect at a segment of length $\delta(\mathfrak{s})$, we have $\delta(\xi_1, \xi_2) = 4v(2) - 2\delta(\mathfrak{s})$ and so the thickness of the node connecting the two components is $2v(2) - \delta(\mathfrak{s})$.  Similar computations show that for Case C2 (resp. C3), the thicknesses of the nodes each lying in only one component of the special fiber are $2\delta(\mathfrak{c}_1) + 2\delta(\mathfrak{c}_2) - 8v(2)$ and $2\delta(\mathfrak{c}_2) + 2\delta(\mathfrak{s}) - 8v(2)$ (resp. $2\delta(\mathfrak{c}_i) - 8v(2)$ for $i = 1, 2$) and that the thickness of the node connecting the two components is $2v(2) - \delta(\mathfrak{c}_2)$ (resp. $2v(2) + \delta(\mathfrak{d})/2$).

Finally, when $p = 3$, the possibilities for genus $2$ are much simpler.  We have $\#\mathcal{B} = 4$ and, as in the genus $1$ situation, there must be a cardinality-$2$ cluster $\mathfrak{s}$ with $\delta(\mathfrak{s}) > 2\pfrac = 3v(3)$, and the only possible topological structure for $\Upsilon$ is that of a single segment of length $\delta(\mathfrak{s}) - 2\pfrac = \delta(\mathfrak{s}) - 3v(3)$.  Each endpoint of $\Upsilon$ (and none of the points in the interior) corresponds to a component of $\SF{\Cst}$, and these two components of $\SF{\Cst}$ meet at $3$ nodes, each of which has thickness equal to $\delta(\mathfrak{s}) - 3v(3)$.  In particular, the configuration of the special fiber $\SF{\Cst}$ falls under Type A.

\begin{rmk} \label{rmk multiplicity of root doesn't matter}

As one sees in the particular case of $p = 3$ and genus $2$ above, the exponents $m_0, \dots, m_h$ (root multiplicities) appearing in (\ref{eq superelliptic degenerate model}) do not affect the spaces $\Upsilon \subseteq \Sigma_{\mathcal{B}}^0 \subsetneq \Sigma_{\mathcal{B}}$ or the structure of the special fiber of the stable (or minimal regular) model of $C$.

\end{rmk}

\appendix \section{Models for $p$-cyclic covers of the projective line} \label{appendix a}

Given a Galois cover $C \to X := \proj_K^1$ whose Galois group is cyclic of order $p$ and whose set of branch points we denote by $\mathcal{B} \subset \bar{K} \cup \{\infty\}$, one obtains explicit equations defining $C$ as follows.  We denote the standard coordinate of $X$ by $x$.  As the function field $\bar{K}(C)$ is a $p$-cyclic Galois extension of $\bar{K}(X) \cong \bar{K}(x)$, it must be generated over $\bar{K}(x)$ by a $p$th root of a $p$th-power-free polynomial $f(x) \in \bar{K}(x)$ (\textit{i.e.} $f$ is not divisible by the $p$th power of any polynomial in $\bar{K}[x]$), and this polynomial must vanish precisely at the points of $\mathcal{B} \smallsetminus \{\infty\}$; in order for $K(C) / K(x)$ to be a $p$-cyclic Galois extension, we must also have $f(x) \in K[x]$ and $\zeta_p \in K$, where $\zeta_p$ denotes a primitive $p$th root of unity.  The set of roots of $f$ thus coincides with $\mathcal{B} \smallsetminus \{\infty\}$.  After possibly replacing $K$ by a degree-$p$ extension (obtained by adjoining $p$th roots of some element of $K$), we may assume that $f$ is monic, so that it can be written in the form 
\begin{equation} \label{eq polynomial for superelliptic}
f(x) = \prod_{i = 1}^d (x - z_i)^{m_i},
\end{equation}
where $z_1, \dots, z_d$ are the points of $\mathcal{B} \smallsetminus \{\infty\}$ and we have $1 \leq m_i \leq p - 1$ for $1 \leq i \leq d$.

\begin{prop} \label{prop superelliptic model}

With the above set-up, the curve $C$ is the normalization of the glued-together affine charts given respectively by the equations 
\begin{equation} \label{eq superelliptic model full}
y^p = f(x) \qquad \text{and} \qquad \check{y}^p = f^\vee(\check{x}),
\end{equation}
where $\check{x} = x^{-1}$, $\check{y} = x^{-\lceil \deg(f)/p \rceil} y$, and $f^\vee(\check{x}) = x^{-p \lceil \deg(f)/p \rceil} f(x)$.  Moreover, we have the following.

\begin{enumerate}[(a)]

\item The group of Galois automorphisms is generated by the automorphism $(x, y) \mapsto (x, \zeta_p y)$.

\item The cover $C \to X$ is branched over the point $\infty$ if and only if we have $p \nmid \deg(f)$.

\item Given $i \in \{1, \dots, d\}$ with $2 \leq m_i \leq p - 1$, the point $(x, y) = (z_i, 0)$ on the affine curve given by the first equation in (\ref{eq superelliptic model full}) is a unibranch singularity (and thus has $1$ point lying above it in the desingularization $C$).  Analogously, if we have $p \lceil \deg(f) / p \rceil - \deg(f) \geq 2$, the point $(\check{x}, \check{y}) = (0, 0)$ on the affine curve given by the second equation in (\ref{eq superelliptic model full}) is a unibranch singularity.

\end{enumerate}

\end{prop}

\begin{proof}
We note that $K(C) = K(X)(y) = K(x, y)$, where $y$ is a $p$th root of $f$.  The first statement is standard and can now be proved using a similar argument to the one used in the proof of \Cref{prop equations for normalization}, while part (a) is immediate from a description of the field automorphisms in $\Gal(K(x, y) / K(x))$.

Note that $f^\vee$, as a polynomial in the variable $\check{x}$, is exactly divisible by $\check{x}^{p \lceil \deg(f)/p \rceil - \deg(f)}$.  The point $\infty$ of $X = \proj_K^1$ is given by $\check{x} = 0$; now putting this into the second equation in (\ref{eq superelliptic model full}) yields the equation $\check{y}^p = 0$ (resp. an equation of the form $\check{y}^p = \text{[constant]}$) if $p \nmid \deg(f)$ (resp. $p \mid \deg(f)$).  Part (b) follows.

Suppose for some index $i$ that we have $m := m_i \geq 2$.  After possibly translating $x$, we may assume that we have $z_i = 0$.  It is easy to check that both partial derivatives vanish at the point $(x, y) = (0, 0)$, so this is a singular point.  Write $f(x) = x^m h(x)$, so that $h(x) \in K[x]$ is a polynomial with no constant term.

We now describe a process of desingularization based on the Euclidean algorithm.  Let $q_1, r_1 \geq 1$ be integers such that we have $p = q_1 m + r_1$ and $r_1 \leq p - 1$.  Making the substitutions $x = x_1 y^{q_1}$ and $y = y_1$, the first equation in (\ref{eq superelliptic model full}) simplifies to 
\begin{equation} \label{eq Euclidean}
y_1^{r_1} = x_1^m h(x_1 y^{q_1}).
\end{equation}
Now let $q_r, r_2 \geq 1$ be integers such that we have $m = q_2 r_1 + r_2$ and $r_2 \leq m - 1$.  Making the substitutions $x_1 = x_2$ and $y_1 = x_2^{q_2} y_2$, the equation in (\ref{eq Euclidean}) simplifies to 
\begin{equation}
y_2^{r_1} = x_2^{r_2} h(x_2^{q_1 q_2 + 1} y_2^{q_1}).
\end{equation}
Repeating this process, as $\gcd(p, m) = 1$, we eventually get $r_s = 1$ for some $s \geq 1$, and our equation simplifies to a polynomial which is either 
\begin{itemize}
\item of the form $y_s = x_s^{r_{s-1}} h(\psi(x_s, y_s))$, where $\psi(x_s, y_s)$ is a non-constant monomial, or 
\item of the form $y_s^{r_{s-1}} = x_s h(\psi(x_s, y_s))$, where $\psi(x_s, y_s)$ is a non-constant monomial.
\end{itemize}
The point $(x, y) = (0, 0)$ corresponds only to the point $(x_s, y_s) = (0, 0)$ and it is easy to see from the form of the equation in $x_s, y_s$ that the partial derivatives do not both vanish at this point, so it is non-singular.  It is clear now that the change of variables from $x, y$ to $x_s, y_s$ provides a desingularization of the affine curve given by the first equation in (\ref{eq superelliptic model full}) at $(x, y) = (0, 0)$.  The claim for the point $(\check{x}, \check{y}) = (0, 0)$ is shown analogously, and part (c) is thus proved.
\end{proof}

\begin{rmk} \label{rmk moving point to infty}

In the above situation, if $\infty \notin \mathcal{B}$, one can always pick some $z \in \mathcal{B}$ and apply a fractional linear transformation that moves $z$ to $\infty$ (\textit{e.g.} $x \mapsto (x - z)^{-1}$) to $X = \proj_K^1$ -- which one may view as choosing a different coordinate variable to be denoted $x$ -- so that the set of branch points of $C \to X$ does contain $\infty$.  This is what enables us to make the assumption that $\infty \in \mathcal{B}$ that was set at the start of this paper.

\end{rmk}

\begin{rmk} \label{rmk why we ignore this type of singularity}

Suppose, as in the proof of \Cref{prop superelliptic model}(c), that $(x, y) = (0, 0)$ is a singular point of the affine curve given by the first equation in (\ref{eq superelliptic model full}) (which is of the form $y^p = x^m h(x)$ with $2 \leq m \leq p - 1$ and $x \nmid h(x)$), and let $\tilde{Q}$ denote the point lying above it in the normalization $C$.  Let $\Cmini$ and $\Xmini$ respectively be the minimal regular model and its quotient by the Galois group as in \Cref{prop persistent nodes}, and let $K' / K$ be a finite extension such that they are defined over $K'$.  Letting $V$ be the component of $\SF{\Xmini}$ to which the point $x = 0$ reduces and letting $\mathcal{X} \leq \Xmini$ be the smooth model of $X$ corresponding to $V$, it follows from the smoothness of the strict transform in $\Cmini$ of the normalization $\mathcal{C}$ of $\mathcal{X}$ in $K(C)$ guaranteed by \Cref{cor inverse images are smooth} that the reduction $Q$ of the $\mathcal{O}_{K'}$-point of $\Cmini$ to which $\tilde{Q}$ extends is non-singular and the only point lying over the reduction of the point $x = 0$ in $\SF{\Xmini}$.

Therefore, desingularizing the singular points of the form $y^p = x^m h(x)$ with $2 \leq m \leq p - 1$ and $x \nmid h(x)$ does not create any new nodes or singular points in the special fiber $\SF{\Cmini}$, which is why we did not need to explicitly treat this issue in the main part of this paper.

\end{rmk}

\section{Constructing a polynomial whose $n$th power has prescribed multiple-of-$n$th-degree coefficients} \label{appendix b}

The purpose of this appendix is to prove the general theorem below, which is used (in the case of prime $n$) to prove \Cref{prop total existence}.  Throughout, we follow the convention that any polynomial $f$ is considered to have degree-$t$ coefficient equal to $0$ for any integer $t > \deg(f)$.

\begin{thm} \label{thm polynomial exists}

Let $m \geq 0$ and $n \geq 1$ be integers, and let $K$ be a field whose characteristic $p$ is either $0$ or satisfies $n = p^s n'$ with $p > n' \in \zz$.

Given any elements $a_0, \dots, a_m \in K$, there exists a finite algebraic extension $L / K$ and a polynomial $q(x) \in L[x]$ such that the degree-$jn$ coefficient of $q^n$ equals $a_j$ for $0 \leq j \leq m$.

\end{thm}

Let us fix integers $m \geq 0$ and $n \geq 1$.  It clearly suffices to assume that our field $K$ in which the elements $a_j$ live is algebraically closed and set out to prove that there is a polynomial $q(x) \in K[x]$ satisfying the desired property.  We now show that it also suffices to assume that we have $p = 0$ or $n = n' < p$ by the following argument.  Suppose that the field $K$ has characteristic $p > 0$ and that $n = p^s n'$ as in the statement of the theorem.  For $0 \leq j \leq m$, let $b_j \in K$ be an element satisfying $b_j^{p^s} = a_j$.  Because $K$ has characteristic $p$, we have 
\begin{equation}
(b_0 + b_1 x^{n'} + \dots + b_m x^{mn'})^{p^s} = (b_0)^{p^s} + (b_1 x^{n'})^{p^s} + \dots + (b_m x^{mn'})^{p^s} = a_0 + a_1 x^n + \dots + a_m x^{mn}.
\end{equation}
If we assume for the moment \Cref{thm polynomial exists} holds over $K$ for any exponent $n < p$, replacing our fixed $n = p^s n'$ by $n'$ and the elements $a_0, \dots, a_m \in K$ by $b_0, \dots, b_m \in K$ in the statement of the theorem, we get a polynomial $q(x) \in K[x]$ such that the degree-$jn'$ coefficient of $q^{n'}$ equals $b_j$ for $0 \leq j \leq m$.  Since taking the $p^s$th power of a polynomial over $K$ amounts to adding $p^s$th powers of each of the terms of the polynomial, for each $j$, the degree-$(jn'p^s = jn)$ coefficient of $(q^{n'})^{p^m} = q^n$ equals $b_j^{p^s} = a_j$, so $q$ is the polynomial satisfying the desired conclusion for $n$ and $a_0, \dots, a_m \in K$.

In light of the above argument, we assume from now on that we have $p = 0$ or $n < p$.

The author is grateful to Robert Lemke Oliver for suggesting to him the following method of proof.  For any integers $m \geq 0$ and $n \geq 1$, we define the map $\varphi_{m, n} : \aff_K^{m+1} \to \aff_K^{m+1}$ as the one taking an $(m+1)$-tuple $(\alpha_0, \dots, \alpha_m) \in K^{m+1}$ to the $(m+1)$-tuple $(a_0, \dots, a_m) \in K^m$, where each $a_j$ is the degree-$jn$ coefficient of the polynomial $(\alpha_0 + \alpha_1 x + \dots + \alpha_m x^m)^n \in K[x]$.  Proving \Cref{thm polynomial exists} clearly amounts to showing that $\varphi_{m, n}$ is surjective.

We observe that the map $\varphi_{m, n}$ is given by an $(m+1)$-tuple of polynomials in the variables $\alpha_0, \dots, \alpha_m$ and that these polynomials are each homogeneous of degree $n$; equivalently, it is immediate from the definition of $\varphi_{m, n}$ that given any scalar $\gamma \in K$, we have 
\begin{equation} \label{eq homogeneity}
\varphi_{m, n}(\gamma (\alpha_0, \dots, \alpha_m)) = \gamma^n \varphi_{m, n}(\alpha_0, \dots, \alpha_m),
\end{equation}
 where the notation on each side indicates scaling of an $(m+1)$-tuple in $K^{m+1}$ treated as a vector.  It follows that $\varphi_{m, n} : \aff_K^{m+1} \to \aff_K^{m+1}$ is a morphism and that it induces a morphism $\proj_K^m \to \proj_K^m$, which by slight abuse of notation we also denote by $\varphi_{m, n}$.

\begin{lemma} \label{lemma morphism conditions}

With the above set-up, the following are equivalent.

\begin{enumerate}[(i)]

\item The morphism $\varphi_{m, n} : \aff_K^{m+1} \to \aff_K^{m+1}$ is surjective.

\item The morphism $\varphi_{m, n} : \proj_K^{m+1} \to \proj_K^{m+1}$ is surjective.

\item The morphism $\varphi_{m, n} : \proj_K^{m+1} \to \proj_K^{m+1}$ is dominant.

\end{enumerate}

\end{lemma}

\begin{proof}
Let $\boldsymbol{0}$ denote the origin in $\aff_K^{m+1}$, and note that it is fixed by $\varphi_{m, n}$.  Since the obvious map $\aff_K^{m+1} \smallsetminus \{\boldsymbol{0}\} \to \proj_K^m$ is surjective, it is immediate that (i) implies (ii).  Now assume that (ii) holds and choose a $K$-point of $\aff_K^{m+1}$ viewed as an $(m+1)$-tuple $\boldsymbol{\beta} \in K^{m+1}$.  The $K$-point $P \in \proj_K^m$ that it reduces to is the image under $\varphi_{m, n}$ of some $K$-point of $\proj_K^m$, which means that it is represented by an $(m+1)$-tuple $\boldsymbol{\beta}' \in K^{m+1} \smallsetminus \{\mathbf{0}\}$ which is the image under $\varphi_{m, n}$ of an $(m+1)$-tuple $\boldsymbol{\alpha} \in K^{m+1} \smallsetminus \{\mathbf{0}\}$.  Now since $\boldsymbol{\beta}, \boldsymbol{\beta}' \in K^{m+1}$ represent the same point of $\proj_K^m$, we must have $\boldsymbol{\beta}' = \gamma \boldsymbol{\beta}$ for some $\gamma \in K^\times$.  Letting $\delta \in K^\times$ be an $n$th root of $\gamma$, using (\ref{eq homogeneity}), we then have 
\begin{equation}
\varphi_{m, n}(\delta^{-1}\boldsymbol{\alpha}) = \gamma^{-1}\boldsymbol{\beta}' = \boldsymbol{\beta}.
\end{equation}
This shows that $\boldsymbol{\beta}$, viewed as a $K$-point of $\aff_K^{m+1}$, lies in the image of $\varphi_{m, n}$, and so $\varphi_{m, n} : \aff_K^{m+1} \to \aff_K^{m+1}$ is surjective.  Thus, properties (i) and (ii) are equivalent.

The fact that (ii) implies (iii) is immediate.  The converse is a direct corollary of the well-known fact that morphisms from projective spaces are closed and thus their images are closed: see for instance \cite[\S I.5, Theorem 2]{shafarevich1994basic}.
\end{proof}

Now the key to proving that the equivalent conditions treated by \Cref{lemma morphism conditions} all hold is to show that the \emph{Jacobian matrix} of the morphism $\varphi_{m, n} : \aff_K^{m+1} \to \aff_K^{m+1}$ is invertible at some point in the domain.  This is done through the following proposition.

\begin{prop} \label{prop Jacobian}

In the above situation, assume that we have $\gcd(m, n) = 1$.  Writing $\alpha_0, \dots, \alpha_m$ for the coordinate variables of $\aff_K^{m+1}$ and writing $a_i$ for the $i$th coordinate of $\varphi_{m, n}(\alpha_0, \dots, \alpha_m)$, the Jacobian matrix $\big(\frac{\partial a_{i-1}}{\partial \alpha_{j-1}}\big)_{1 \leq i, j \leq m + 1}$ is invertible at the point $(\alpha_0, \dots, \alpha_m) = (1, 0, \dots, 0, 1)$ in the domain.

\end{prop}

\begin{proof}
The $(i, j)$th entry of the Jacobian is given by the degree-$(i-1)n$ term of the polynomial 
\begin{equation}
\begin{aligned}
\tfrac{\partial}{\partial \alpha_{j-1}}(\alpha_m x^m + \dots + \alpha_1 x + \alpha_0)^n &= n(\alpha_m x^m + \dots + \alpha_1 x + \alpha_0)^{n-1} \tfrac{\partial}{\partial \alpha_{j-1}} (\alpha_{j-1} x^{j-1}) \\
&= n(\alpha_m x^m + \dots + \alpha_1 x + \alpha_0)^{n-1} x^{j-1}.
\end{aligned}
\end{equation}
Evaluated at $(\alpha_0, \dots, \alpha_m) = (1, 0, \dots, 0, 1)$, we get 
\begin{equation}
\begin{aligned}
\tfrac{\partial}{\partial \alpha_{j-1}}(\alpha_m x^m + \dots + \alpha_1 x + \alpha_0)^n \big|_{(\alpha_0, \dots, \alpha_m) = (1, 0, \dots, 0, 1)} &= n(x^m + 1)^{n-1} x^{j-1} \\
&= \sum_{k = 0}^{n-1} n {n - 1 \choose k} x^{mk+j-1}.
\end{aligned}
\end{equation}
Denoting the reduction of any integer $a \in \zz$ modulo $n$ by $\bar{a} \in \zz / n\zz$, we note that the element $\overline{m} \in \zz / n\zz$ has a multiplicative inverse $\overline{m}^{-1} \in \zz / n\zz$ since the integers $m$ and $n$ are relatively prime.  Fixing an index $j$, we thus have a unique element $\bar{k} \in \zz / n\zz$ such that we have $\overline{m}\bar{k} + \bar{j} - \bar{1} = \bar{0} \in \zz / n\zz$, given by $\bar{k} = \overline{m}^{-1}(\bar{1} - \bar{j})$.  There is then a unique $k \in \{0, \dots, n - 1\}$ reducing modulo $n$ to $\bar{k}$, which is the only $k$ in this interval satisfying $mk + j - 1 \in n\zz$.  The inequalities $0 \leq k \leq n - 1$ and $1 \leq j \leq m + 1$ imply that $0 \leq mk + j - 1 \leq mn$, and so we have $0 \leq \frac{1}{n}(mk + j - 1) \leq m$; this gives us $mk + j - 1 = (i - 1)n$ for some unique $i \in \{1, \dots, m + 1\}$.  As we have $p = 0$ or $n < p$, the coefficient $n {n - 1 \choose k}$ is nonzero; it follows that there is a unique nonzero entry in the $j$th column of the Jacobian matrix.

Now suppose that there are two indices $j, j'$ such that the $(i, j)$th and $(i, j')$th entries of the Jacobian matrix are both nonzero.  Then we have integers $k, k' \in \{0, \dots, n - 1\}$ satisfying 
\begin{equation} \label{eq j j'}
mk + j - 1 = (i - 1)n = mk' + j' - 1.
\end{equation}
This implies that $m(k - k') = j' - j$.  Suppose that we have $j' \neq j$.  Then the only way for $m$ to divide the difference $j' - j$ is if $j' = m + 1$ and $j = 1$, which gives us $m(k - k') = (m + 1) - 1 = m$ so that $k - k' = 1$.  Now from the first equation in (\ref{eq j j'}), we get $mk = (i - 1)n$.  Since $m$ and $n$ are relatively prime, this means that $n \mid k$, which forces $k = 0$.  This is a contradiction, since $k' = k - 1$ cannot be negative.  Therefore, we have $j' = j$, which is to say that for each index $i$, the $i$th row of the Jacobian matrix has a unique nonzero entry.

Now since each row and each column of the Jacobian matrix at the point $(\alpha_0, \dots, \alpha_m) = (1, 0, \dots, 0, 1)$ has a unique nonzero entry, that matrix is invertible (being a permutation matrix times a diagonal matrix with nonzero diagonal entries).
\end{proof}

\begin{cor} \label{cor surjectivity if relatively prime}

In the above situation, assume that we have $\gcd(m, n) = 1$.  Then the map $\varphi_{m, n} : \aff_K^{m+1} \to \aff_K^{m+1}$ is surjective.

\end{cor}

\begin{proof}
\Cref{prop Jacobian} implies that the morphism $\varphi_{m, n}$ is \'{e}tale at the point $(\alpha_0, \dots, \alpha_m) = (1, 0, \dots, 0, 1)$ (as in \cite[Example 5.45]{milne2024algebraic}); it is therefore \'{e}tale on a Zariski open subvariety $U \subset \aff_K^{m+1}$ (see for instance \cite[Lemma 5.54]{milne2024algebraic}).  Since $\varphi_{m, n}$ is an \'{e}tale morphism of $U$ onto its image $\varphi_{m, n}(U)$, and since $U$ has dimension $m + 1$, its image $\varphi_{m, n}(U)$ must have dimension $m + 1$ as well.  Since the Zariski closure of any dimension-$(m+1)$ subspace of $\aff_K^{m+1}$ is the whole space $\aff_K^{m+1}$, the image of $\varphi_{m, n}$ is dense in $\aff_K^{m+1}$; in other words, the morphism $\varphi_{m, n} : \aff_K^{m+1} \to \aff_K^{m+1}$ is dominant.

Now the obvious (surjective) map $\aff_K^{m+1} \smallsetminus \{\boldsymbol{0}\} \to \proj_K^m$ (as in the proof of \Cref{lemma morphism conditions}) is easily verified to be continuous under the Zariski topology; it is an elementary exercise in point-set topology to show that the image of a dense subspace under a continuous map is dense in the image of that map, so we have that the image of $\varphi_{m, n}(\aff_K^{m+1}) \subset \aff_K^{m+1}$ in $\proj_K^m$ is a dense subspace of $\proj_K^m$ as well.  Thus, the map $\varphi_{m, n} : \proj_K^m \to \proj_K^m$ is also dominant.  Then by \Cref{lemma morphism conditions}, the map $\varphi_{m, n} : \aff_K^{m+1} \to \aff_K^{m+1}$ is surjective.
\end{proof}

The following corollary finishes the proof of \Cref{thm polynomial exists}.

\begin{cor}

In the above situation, for any integers $m \geq 0$ and $n \geq 1$ (not necessarily relatively prime), the map $\varphi_{m, n} : \aff_K^{m+1} \to \aff_K^{m+1}$ is surjective.

\end{cor}

\begin{proof}
If we have $\gcd(m, n) = 1$, then we are done by \Cref{cor surjectivity if relatively prime}, so let us assume that $\gcd(m, n) > 1$.  In that case, we certainly have $\gcd(m + 1, n) = 1$, and then \Cref{cor surjectivity if relatively prime} tells us that the morphism $\varphi_{m+1, n} : \aff_K^{m+2} \to \aff_K^{m+2}$ is surjective.  Choose any $(m+1)$-tuple $(a_0, \dots, a_m) \in K^{m+1}$.  There is an $(m+2)$-tuple $(\alpha_0, \dots, \alpha_{m+1}) \in K^{m+2}$ such that we have $\varphi_{m+1, n}(\alpha_0, \dots, \alpha_{m+1}) = (a_0, \dots, a_m, 0)$.  By definition of $\varphi_{m+1, n}$, the polynomial $q(x) := \alpha_0 x + \dots + \alpha_{m+1} x^{m+1}$ satisfies that the polynomial $q^n$ has no degree-$(m+1)n$ term.  Since we have $\deg(q) \leq m + 1$, we also have $\deg(q^n) = n\deg(q) \leq (m + 1)n$, and by what we have just observed, this inequality is strict; therefore we have the strict inequality $\deg(q) < m + 1$ as well.  In other words, we have $\alpha_{m+1} = 0$, so that $q(x) = \alpha_0 + \dots + \alpha_m x^m$.  By construction, the degree-$jn$ coefficient of $q^n$ equals $a_j$ for $0 \leq j \leq m$, meaning that $\varphi_{m, n}(\alpha_0, \dots, \alpha_m) = (a_0, \dots, a_m)$.  This proves the desired surjectivity.
\end{proof}

\bibliographystyle{plain}
\bibliography{bibfile}

\end{document}